\documentclass{colt2012} % Anonymized submission

\usepackage{algorithm,algorithmic}
\usepackage{amssymb, multirow, paralist, color}
\newtheorem{thm}{Theorem}

\newtheorem{ass}[thm]{Assumption}
\newtheorem{obs}{Observation}
\usepackage{enumitem}

\usepackage{textcase,booktabs}
\usepackage[T1]{fontenc}

\def \X {\mathcal{X}}

\def \R {\mathbb{R}}

\def \w {\mathbf{w}}
\def \v {\mathbf{v}}

\def \x {\mathbf{x}}

\def \E {\mathrm{E}}

\def \x {\mathbf{x}}

\def \u {\mathbf{u}}

\def \g {\mathbf{g}}

\def \wh {\widehat{\w}}

\def \Fh {\widehat{F}}

\usepackage{scrextend}
\usepackage{tikz}
\usetikzlibrary{fit,calc,positioning,decorations.pathreplacing,matrix,angles}

\usepackage{footnote}

\begin{document}

\title[Accelerated Stochastic Subgradient Methods ]{Accelerate Stochastic Subgradient Method by Leveraging Local Growth Condition}
 \author{\Name{Yi Xu}$^\dagger$\Email{yi-xu@uiowa.edu}\\
  \Name{Qihang Lin}$^\ddagger$\Email{qihang-lin@uiowa.edu}\\
\Name{Tianbao Yang}$^\dagger$\Email{tianbao-yang@uiowa.edu}\\
   \addr$^\dagger$ Department of Computer Science\\
      \addr$^\ddagger$ Department of Management  Sciences\\
    The University of Iowa, Iowa City, IA 52242
}

\maketitle
\vspace*{-0.5in}
\begin{center}{First version: July 4, 2016}\end{center}

\begin{abstract}
In this paper, a new theory is developed for first-order stochastic convex optimization, showing that the global convergence rate is sufficiently  quantified by a local growth rate of the objective function in a neighborhood of the optimal solutions. In particular, if the objective function $F(\w)$ in the $\epsilon$-sublevel set grows as fast as $\|\w - \w_*\|_2^{1/\theta}$, where $\w_*$ represents the closest optimal solution to $\w$ and $\theta\in(0,1]$ quantifies the local growth rate,  the iteration complexity of first-order stochastic optimization for achieving  an $\epsilon$-optimal solution can be $\widetilde O(1/\epsilon^{2(1-\theta)})$, which is {\it optimal at most} up to a logarithmic factor. To achieve the faster global convergence, we develop two different {\bf accelerated stochastic subgradient} methods by iteratively solving the original problem approximately in a local region around a historical solution with the size of the local region gradually decreasing as the solution approaches the optimal set. Besides the theoretical improvements, this work also includes new contributions towards making the proposed algorithms practical: (i) we present practical variants of accelerated stochastic subgradient methods that can run without the knowledge of  multiplicative growth constant and even the growth rate $\theta$; (ii) we consider a broad family of problems in machine learning to demonstrate that the proposed algorithms enjoy faster convergence than traditional stochastic subgradient method. We also characterize the complexity of the proposed algorithms for ensuring the gradient is small without the smoothness assumption. 
\end{abstract}

\section{Introduction}
In this paper, we are interested in solving the following stochastic optimization problem:
\begin{equation}\label{eqn:psg}
\min_{\w\in\mathcal K}F(\w)\triangleq \E_{\xi}[f(\w; \xi)],
\end{equation}
where $\xi$ is a random variable, $f(\w; \xi)$ is a convex function of $\w$, $\E_\xi[\cdot]$ is the expectation over $\xi$ and $\mathcal K$ is a convex domain. We denote by $\partial f(\w; \xi)$ a subgradient of $f(\w; \xi)$. Let $\mathcal K_*$ denote the optimal set of~(\ref{eqn:psg}) and $F_*$ denote the optimal value. 

%It has been proven to be very efficient to solve large-scale optimization problems
  
In recent years, it becomes very important to develop efficient and effective optimization algorithms for solving large-scale machine learning problems~\citep{fang2018faster,guo2017thresholded,blanchard2016convergence}.
Traditional stochastic subgradient (SSG) method updates the solution according to  
\begin{equation}\label{eqn:ssg}
\w_{t+1} =\Pi_{\mathcal K}[\w_t - \eta_t \partial f(\w_t; \xi_t)],
\end{equation}
for $t=1,\ldots, T$, where $\xi_t$ is a sampled value of $\xi$ at $t$-th iteration, $\eta_t$ is a step size and $\Pi_\mathcal K[\w]=\arg\min_{\v\in\mathcal K}\|\w - \v\|_2^2$ is a projection operator that projects a point into $\mathcal K$. Previous studies have shown that under the following assumptions  i) $\|\partial f(\w; \xi)\|_2\leq G$, ii) there exists $\w_*\in\mathcal K_*$ such that $\|\w_t - \w_*\|_2\leq B$ for $t=1,\ldots, T$~\footnote{This holds if we assume the domain $\mathcal K$ is bounded such that $\max_{\w, \v\in\mathcal K}\|\w - \v\|_2\leq B$ or if assume $dist(\w_1,\mathcal K_*)\leq B/2$ and project every solution $\w_t$ into $\mathcal K\cap \mathcal B(\w_1,B/2)$.}, and by setting the step size $\eta_t = \frac{B}{G\sqrt{T}}$ in~(\ref{eqn:ssg}),  with a high probability $1-\delta$ we have
\begin{equation}\label{eqn:ssgi}
%\textstyle F(\wh_T) - F_*\leq O\left(\frac{GB(1+ \sqrt{\log(1/\delta)})}{\sqrt{T}}\right),
\textstyle F(\wh_T) - F_*\leq O\left({GB(1+ \sqrt{\log(1/\delta)})}/{\sqrt{T}}\right),
\end{equation}
where $\wh_T=\sum_{t=1}^T\w_t/T$. 
The above convergence implies that in order to obtain an $\epsilon$-optimal solution by SSG, i.e., finding an $\w$ such that $F(\w) - F_*\leq \epsilon$ with a high probability $1-\delta$, one needs at least $\textstyle T = O({G^2B^2(1+\sqrt{\log(1/\delta)})^2}/{\epsilon^2})$ in the worst-case. %$\textstyle T = O(\frac{G^2B^2\left(1+\sqrt{\log(1/\delta)}\right)^2}{\epsilon^2})$ in the worst-case. 

It is commonly known that the slow convergence of SSG  is due to the variance in the stochastic subgradient and  the non-smoothness nature of the problem as well, which therefore requires a decreasing step size or a very small step size. Recently, there emerges a stream of studies on various variance reduction techniques to accelerate  stochastic {\bf gradient} method~\citep{DBLP:conf/nips/RouxSB12,DBLP:conf/nips/0005MJ13,NIPS2013_4937,xiao2014proximal,DBLP:conf/nips/DefazioBL14}. However, they all hinge on the smoothness assumption.  The proposed algorithms in this work  tackle the issue of variance of {{\bf stochastic  subgradient} without the smoothness assumption from another pespective. 

The main motivation for addressing this problem is from a key observation: a high probability analysis of the SSG method shows that the variance term of the stochastic subgradient is accompanied by an upper bound of distance of intermediate solutions to the {\it target} solution. This observation has also been leveraged  in previous analysis to design faster convergence for stochastic  convex optimization  that use a strong or uniform  convexity condition~\citep{hazan-20110-beyond,juditsky2014} or a global growth condition~\citep{DBLP:conf/icml/RamdasS13} to control the  distance of intermediate solutions to the {\it optimal} solution by their functional residuals. However, we find  these global assumptions are completely unnecessary, which may not only restrict their applications to a broad family of problems but also worsen the convergence rate due to the larger multiplicative growth constant that could be domain-size dependent. In contrast, we develop a new theory only relying on the local growth condition  to control the distance of intermediate solutions to the {\it $\epsilon$-optimal} solution by their functional residuals but achieving a fast global convergence. 

Besides the fundamental difference, the present work also possesses  several unique algorithmic contributions compared with previous similar work on stochastic optimization: (i) we have two different ways  to control the  distance of intermediate solutions to the {\it $\epsilon$-optimal} solution, one by explicitly imposing a bounded ball constraint and another one by implicitly regularizing the intermediate solutions, where the later one could be more efficient if the projection into the intersection of a bounded ball and the problem domain is complicated; (ii) we develop more practical variants that can be run without knowing the multiplicative growth constant though under a slightly stringent condition;  (iii) for problems whose local growth rate is unknown we still develop an improved convergence result of the proposed algorithms comparing with the SSG method. In addition, the present work will demonstrate the improved results and practicability of the proposed algorithms for many problems in machine learning, which is lacking in similar previous work.

We summarize the main results below. The proposed algorithms and their analysis are developed under the following generic local growth condition (LGC): 
\begin{equation}\label{eqn:lgc2}
\|\w - \w_*\|_2\leq c (F(\w) - F_*)^{\theta}, \quad \forall \w\in\mathcal S_\epsilon,
\end{equation}
where $\theta\in(0,1]$, $c>0$ and $\mathcal S_\epsilon$ denotes the $\epsilon$-sublevel set with $\epsilon$ being a small value. 

\begin{itemize}
\item In Section~\ref{sec:main}, we present two variants of accelerated stochastic subgradient (ASSG) methods and analyze their iteration complexities for finding an $\epsilon$-optimal solution with high probability. The two variants use different ways  to mitigate the effect of variance of stochastic subgradient with one  using shrinking ball constraints and the second variant using increasing regularization. With complete knowledge of $c$ and $\theta$, we show that both variants can find an $\epsilon$-optimal solution with a complexity of $\widetilde O(1/\epsilon^{2(1-\theta)})$ for $\theta\in(0,1]$, where $\widetilde O$ suppresses a logarithmic factor in terms of $1/\epsilon$. %We also present a simple variant of ASSG without using shrinking ball constraints or increasing regularization but only achieves an improved expectaional convergence results for $\theta\in(0, 1/2]$. 

\item In Section~\ref{subsec:prac}, we present a practical variant of ASSG with partial or no knowledge about the LGC. In particular, when $c$ is unknown and $\theta\in(0,1)$ is known the practical variant of ASSG enjoys an improved complexity of $\widetilde O(1/\epsilon^{2(1-\theta)})$. When both $c$ and $\theta$ are unknown, we show that the practical variant still enjoys a better complexity than that of traditional SSG. In particular, the dependence on the distance from the initial solution to the optimal set of SSG's complexity is reduced to a much smaller distance multiplied by a logarithmic factor dependent on the quality of the initial solution.  

\item In Section~\ref{sec:prox}, we consider an extension to proximal algorithms that  handle non-smooth but simple regularizers by a proximal mapping.
%\item
 In Section~\ref{sec:sta}, we consider the complexity of the proposed ASSG algorithms for ensuing the gradient of the objective function is small. 

\item In Section~\ref{sec:app}, we consider the applications in machine learning and present many examples with the local growth rate $\theta$ explicitly exhibited. 
In Section~\ref{sec:exp}, we present numerical experiments for demonstrating the effectiveness of the proposed algorithms. 
\end{itemize}

\section{Related Work} 
 The most similar work to the present one is \citep{DBLP:conf/icml/RamdasS13}, which studied stochastic convex optimization under a global growth condition, which they called Tsybakov noise condition. One major difference from their result is that we achieve the same order of iteration complexity up to a logarithmic factor under only a local growth condition. As observed later on, the multiplicative growth constant in local growth condition is domain-size independent that is smaller than that in global growth condition, which could be domain-size dependent. Besides,  the stochastic optimization algorithm in \citep{DBLP:conf/icml/RamdasS13} assume the {\it optimization domain $\mathcal K$ is bounded}, which is removed in this work. In addition, they do not address the issue when the multiplicative constant is unknown and lack study of applicability for machine learning problems. \cite{juditsky2014} presented primal-dual subgradient and stochastic subgradient methods for solving problems under the uniform convexity assumption (see the definition under Observation~\ref{obs:1}). As exhibited shortly, the uniform convexity condition covers only a smaller  family of problems than the considered local growth condition. However, when the problem is uniform convex, the iteration complexity obtained in this work resembles that in \citep{juditsky2014}. 
 
Recently, there emerge a wave of studies that attempt to improve the convergence of existing algorithms under no strong convexity assumption by considering certain weaker conditions than strong convexity~\citep{DBLP:journals/corr/nesterov16linearnon,Liu:2015:APS:2789272.2789282,DBLP:journals/corr/abs-1303-4645,DBLP:journals/siamjo/LiuW15,DBLP:journals/corr/GongY14,DBLP:conf/pkdd/KarimiNS16,DBLP:journals/corr/abs/1606.00269,DBLP:conf/icml/QuXO16,DBLP:journals/jmlr/WangL14}. Several recent works~\citep{DBLP:journals/corr/nesterov16linearnon,DBLP:conf/pkdd/KarimiNS16,DBLP:journals/corr/abs/1606.00269} have unified many of these conditions, implying that they are a kind of global growth condition with $\theta=1/2$. Unlike the present work, most of these developments require certain smoothness assumption except~\citep{DBLP:conf/icml/QuXO16}.
 
\cite{Luo:1992a,Luo:1992b,Luo:1993} pioneered the idea of using local error bound condition to show faster convergence of gradient descent, proximal gradient descent, and many other methods for a family of structured composite problems (e.g., the LASSO problem). Many follow-up works~\citep{DBLP:conf/nips/HouZSL13,DBLP:conf/icml/ZhouZS15,DBLP:journals/corr/ZhouS15a} have considered different regularizers (e.g., $\ell_{1,2}$ regularizer, nuclear norm regularizer). However, these works only obtained asymptotically faster (i.e., linear) convergence and they  hinge on the smoothness on some parts of the problem.  \citep{DBLP:journals/corr/arXiv:1512.03107,DBLP:journals/corr/abs-1607-03815} have considered the same local growth condition (aka local error bound condition in their work) for developing faster deterministic algorithms for non-smooth optimization. However, they did not address the problem of stochastic convex optimization, which restricts their applicability to large-scale problems in machine learning. 
 
Finally, we note that the improved iteration complexity  in this paper does not contradict to the lower bound  in~\citep{opac-b1091338,opac-b1104789}. The bad examples constructed to derive the lower bound for general non-smooth optimization do not satisfy the assumptions made in this work (in particular Assumption~\ref{ass:1}(b)). Recently, \cite{DBLP:conf/nips/ChatterjeeDLZ16} characterize the local minimax complexity of stochastic convex optimization by introducing modulus of continuity that measures the size  of the ``flat set'' where the magnitude of the subderivative is a small value. They established a local minimax complexity result when the modulus of continuity has polynomial growth and proposed an adaptive stochastic optimization algorithm for only one-dimensional problems that achieves the local minimax complexity upto a logarithmic factor. It remains unclear which is more generic between LGC and the polynomial growing modulus of continuity.  %Yet it is not exactly clear the relationship between the local growth condition and their defined modulus of continuity.% in the high-dimensional space. 
%We show in the supplement that in the one-dimensional space the global growth condition with $\theta\in(0,1)$ implies modulus of the continuity has polynomial growth, thus implying the proposed algorithms are also local minimax optimal up to a logarithmic factor for achieving a fixed precision solution, which remains unresolved  before this work. It remains an interesting problem to see whether the iteration complexity given in this work for high-dimensional problems is locally minimax optimal. 

\section{Preliminaries} 
Recall the notations $\mathcal K_*$ and $F_*$ that denote the optimal set of~(\ref{eqn:psg}) and the optimal value, respectively. 
For the optimization problem in~(\ref{eqn:psg}), we make the following assumption throughout the paper.
\begin{ass}\label{ass:1} For a stochastic optimization problem~(\ref{eqn:psg}), we assume
\begin{enumerate}%[label= (\alph*)]
\item there exist $\w_0\in\mathcal K$ and $\epsilon_0\geq 0$ such that $F(\w_0) - F_*\leq \epsilon_0$;%\min_{\w\in\mathcal K}F(\w)\leq \epsilon_0$;  
%\item $\mathcal K_*$ is a non-empty  compact set;
\item There exists a constant $G$ such that $\|\partial f(\w; \xi)\|_2 \leq G$.
%\item There exists a constant $D$ such that for any $\w, \v \in \mathcal K$, $\|\w - \v \|_2 \leq D$; 
\end{enumerate}
\end{ass}
{\bf Remark:} (1) essentially assumes the availability of a lower bound of the optimal objective value, which usually holds for machine learning problems (due to non-negativeness of the objective function).  (2) is a standard assumption also made in many previous stochastic gradient-based methods~\citep{hazan-20110-beyond,RakhlinSS12,DBLP:conf/icml/RamdasS13}. By Jensen's inequality, we also have $\|\partial F(\w)\|_2 \leq G$. It is notable that unlike previous analysis of SSG, we do not assume the domain $\mathcal K$ is  bounded.  Instead, we will assume the problem satisfies a generic local growth condition as presented shortly. %This is a relaxed condition in contrast with most previous work that assume the domain $\mathcal K$ is bounded. Even if $\mathcal K$ is unbounded, as long as the function is a proper lower-semicontinuous convex and coercive function defined on a finite dimensional space,  $\mathcal K_*$ is nonempty and compact~\citep{arxiv:1510.08234}. Note that any norm-regularized loss function minimization problem on a finite dimensional space  in machine learning satisfy this property.   %; 4) Assumption~\ref{ass:1}(d) assume that the domain $\mathcal K$ is bounded.

For any $\w\in\mathcal K$, let $\w^*$ denote the closest optimal solution in $\mathcal K_*$ to $\w$, i.e., $\w^* = \arg\min_{\v\in\mathcal K_*}\|\v - \w\|_2^2$, 
which is unique. We denote by  $\mathcal L_\epsilon$ the  $\epsilon$-level set of $F(\w)$ and  by $\mathcal S_\epsilon$  the $\epsilon$-sublevel set of $F(\w)$, respectively, i.e.,
%\begin{equation*}
$\mathcal L_\epsilon = \{\w\in\mathcal K: F(\w) = F_* + \epsilon\}$, 
$\mathcal S_\epsilon = \{\w\in\mathcal K: F(\w) \leq F_* + \epsilon\}$.
%\end{equation*}
%Given $\mathcal K_*$ is bounded,  it follows from~\cite[Corollary 8.7.1]{rockafellar1970convex} that the sublevel set $\mathcal S_\epsilon$ is bounded for any $\epsilon\geq 0$ and so as the level set $\mathcal L_\epsilon$. 
%Define $dist(\mathcal L_\epsilon, \mathcal K_*)$ to be the maximum distance of points on the level set $\mathcal L_\epsilon$ to the optimal set $\mathcal K_*$, i.e.,
%\begin{equation}\label{eqn:keyB}
%dist(\mathcal L_\epsilon, \mathcal K_*)= \max_{x\in\mathcal L_\epsilon}\left[dist(\w,\mathcal K_*)\triangleq\min_{\v\in\mathcal K_*}\|\w - \v\|\right].
%\end{equation}
%Due to that $\mathcal L_\epsilon$ and $\mathcal K_*$ are bounded, $dist(\mathcal L_\epsilon, \mathcal K_*)$ is also bounded. 
Let $\w_{\epsilon}^\dagger$ denote the closest point in the $\epsilon$-sublevel set to $\w$, i.e.,
\begin{equation}\label{eqn:xepsilon}
\w_\epsilon^{\dagger}=\arg\min_{\v\in\mathcal S_\epsilon}\|\v - \w\|_2^2.
\end{equation}
It is easy to show that $\w_\epsilon^{\dagger}\in\mathcal L_\epsilon$ when $\w\notin\mathcal S_\epsilon$ (using the KKT condition).  %A function $H(\w)$ is $\lambda$-strongly convex, if for all $\w, \v \in \mathcal K$ and any subgradient $\g\in\partial H(\v)$ of $H$ at $\v$, 
%\begin{equation*}
%H(\w) \geq H(\v) + \langle \g, \w-\v\rangle + \frac{\mu}{2} \| \w - \v \|_2^2.
%\end{equation*}
%The big $O(\cdot)$ and big $\Omega(\cdot)$ notations follow the standard meanings, i.e., $f = O(g)$ means that there exists $C>0$ such that  $f\leq Cg$ and $f=\Omega(g)$ means that there exist $C>0$ such that $ f\geq Cg$. 
Let $\mathcal B(\w, r)=\{\u\in\R^d: \|\u - \w\|_2\leq r\}$ denote an Euclidean ball centered  at $\w$ with a radius $r$. Denote by $dist(\w,\mathcal K_*)= \min_{\v\in\mathcal K_*}\|\w - \v\|_2$ the distance between $\w$ and the set $\mathcal K_*$, by  $\partial^0 F(\w)$ the projection of $0$ onto the nonempty closed convex set $\partial F(\w)$, i.e., $\|\partial^0 F(\w)\|_2=\min_{\v\in\partial F(\w)}\|\v\|_2$. 
%\begin{equation*}
%    dist(\w,\mathcal K_*) = \min_{\v\in\mathcal K_*}\|\w - \v\|_2.
%\end{equation*}
%Let $\Pi_{\mathcal K}[\cdot]=\arg\min_{\v\in\mathcal K}\|\w - \v\|_2^2$ be a projection operator. 
%\begin{equation}\label{eqn:proj}
%\Pi_{\mathcal K}[\w] = \arg\min_{\v\in\mathcal K}\|\w - \v\|_2^2.
%\end{equation}equation
   
\subsection{Functional Local Growth Rate}\label{subsec:LGC}
We quantify the  functional local growth rate by measuring how fast the functional value increase when moving a point away from the optimal solution in the $\epsilon$-sublevel set.  In particular, we state the local growth condition in the following assumption. 

\begin{ass}\label{ass:2} The objective function $F(\cdot)$ satisfies a local growth condition on $\mathcal S_\epsilon$ if there exists a constant $c>0$ and $\theta\in(0,1]$ such that:
\begin{equation}\label{eqn:lgc}
\|\w - \w_*\|_2\leq c (F(\w) - F_*)^{\theta}, \quad \forall \w\in\mathcal S_\epsilon,
\end{equation}
where $\w_*$ is the closest solution in the optimal set $\mathcal K_*$ to $\w$. 
\end{ass}
%In particular, a function $F(\w)$ has a local growth rate $\theta\in(0,1]$ in the $\epsilon$-sublevel set ($\epsilon\ll 1$) if there exists a constant $\lambda>0$ such that:
%\begin{equation}\label{eqn:lgc}
%\lambda \|\w - \w_*\|_2^{1/\theta}\leq F(\w) - F_*, \quad \forall \w\in\mathcal S_\epsilon,
%\end{equation}
Note that  the local growth rate $\theta$ is at most $1$. This is due to that $F(\w)$ is $G$-Lipschitz continuous and $\lim_{\w\rightarrow\w_*}\|\w - \w_*\|_2^{1-\alpha}=0$ if $\alpha<1$. The inequality in~(\ref{eqn:lgc}) %can be  equivalently written as 
%\begin{equation}\label{eqn:lgc2}
%\|\w - \w_*\|_2\leq c (F(\w) - F_*)^{\theta}, \quad \forall \w\in\mathcal S_\epsilon,
%\end{equation}
%where $c = 1/\lambda^\theta$, which 
is also called as local error bound condition in~\citep{DBLP:journals/corr/arXiv:1512.03107}. In this work, to avoid confusion with earlier work by~\cite{Luo:1992a,Luo:1992b,Luo:1993} who also explored a related but different local error bound condition, we refer to the inequality in~(\ref{eqn:lgc}) or~(\ref{eqn:lgc2}) as local growth condition (LGC). {It is worth noting that LGC is a general condition, comparing with several other error bound conditions. For example, the polyhedral error bound condition~\citep{DBLP:journals/corr/arXiv:1512.03107} implies LGC with $\theta=1$; while the function has a Lipschitz-continuous gradient, then the Polyak-$\L$ojasiewicz condition is equivalent to the LGC with $\theta = 1/2$. In Section~\ref{sec:app}, we will present several applications in risk minimization problems that satisfying LGC. For more details about the relationship between LGC and other conditions, we refer the reader to~\citep{DBLP:conf/pkdd/KarimiNS16, arxiv:1510.08234, zhang2017restricted, DBLP:journals/corr/arXiv:1512.03107}.}
If the function $F(\x)$ is assumed to satisfy~(\ref{eqn:lgc}) for all $\w\in\mathcal K$, it is referred to as global growth condition (GGC). Note that since we do not assume a bounded $\mathcal K$, the GGC might be ill posed. In the following discussions, when compared with GGC we simply assume the domain is bounded.

Below, we present several observations mostly from existing work to clarify  the relationship between the LGC~(\ref{eqn:lgc2}) and previous conditions, and also justify our choice of  LGC that covers a much broader family of functions than previous conditions and induces a smaller multiplicative growth constant $c$ than that induced by GGC. 
\begin{obs}\label{obs:1}
Strong convexity or uniform convexity condition implies  LGC with $\theta=1/2$, but not vice versa. 
\end{obs}
$F(\w)$ is said to satisfy a  uniform convexity condition  on $\mathcal K$ with convexity parameters $p\geq 2$ and $\mu$ if: 
\begin{equation*}
F(\u) \geq F(\v) + \partial F(\v)^{\top}(\u - \v) + \frac{\mu\|\u - \v\|_2^p}{2}, \forall \u, \v\in\mathcal K.
\end{equation*}
If we let $\u = \w$, $\v= \w_*$, then $\partial F(\w_*)^{\top}(\w -\w_*)\geq 0$ for any $\w\in\mathcal K$, and we have~(\ref{eqn:lgc}) with $\theta = 1/p\in(0,1/2]$. Clearly LGC covers a broader family of functions than uniform convexity. 

\begin{obs}
The weak strong convexity~\citep{DBLP:journals/corr/nesterov16linearnon}, essential strong convexity~\citep{Liu:2015:APS:2789272.2789282}, restricted strong convexity~\citep{DBLP:journals/corr/abs-1303-4645}, optimal strong convexity~\citep{DBLP:journals/siamjo/LiuW15}, semi-strong convexity~\citep{DBLP:journals/corr/GongY14} and other error bound conditions considered in several recent work~\citep{DBLP:conf/pkdd/KarimiNS16,DBLP:journals/corr/abs/1606.00269} imply  a GGC on the entire optimization domain $\mathcal K$  with $\theta=1/2$ for a convex function. 
\end{obs}
Some of these conditions are also equivalent to the GGC with $\theta=1/2$. We refer the reader to~\citep{DBLP:journals/corr/nesterov16linearnon}, \citep{DBLP:conf/pkdd/KarimiNS16} and~\citep{DBLP:journals/corr/abs/1606.00269} for more discussions of these conditions. 

The third observation shows that LGC could imply faster convergence than that induced by GGC. 
\begin{obs}
The LGC  could induce a smaller constant $c$ in~(\ref{eqn:lgc2})  that is domain-size independent than that induced by  the GGC on the entire optimization domain $\mathcal K$. 
\end{obs}
To illustrate this, we consider a function $f(x)=x^2$ if $|x|\leq 1$ and $f(x)= |x|$ if $1< |x|\leq s$, where $s$ specifies the size of the domain. In the $\epsilon$-sublevel set ($\epsilon<1$), the LGC~(\ref{eqn:lgc2}) holds with $\theta=1/2$ and $c=1$. In order to make the inequality $|x|\leq cf(x)^{1/2}$ hold for all $x\in[-s,s]$, we can see that $c=\max_{|x|\leq s}\frac{|x|}{f(x)^{1/2}} = \max_{|x|\leq s}\sqrt{|x|} = \sqrt{s}$. As a result, GGC induces a larger $c$ that depends on the domain size. 

The next observation shows that Luo-Tseng's local error bound condition is closely related to the LGC with $\theta=1/2$. To this end, we first give the definition of Luo-Tseng's local error bound condition. Let $F(\w) = h(\w) + P(\w)$, where $h(\w)$ is a proper closed function with an open domain containing $\mathcal K$ and is continuously differentiable with a locally Lipschitz continuous gradient on any compact set within $dom(h)$ and $P(\w)$ is a proper closed convex function. Such a function $F(\w)$ is said to satisfy Luo-Tseng's local error bound if for any $\zeta>0$, there exists $c, \varepsilon>0$ so that
\begin{equation*}
\|\w - \w_*\|_2 \leq c \| \textrm{prox}_P(\w - \nabla h(\w)) - \w\|_2,
\end{equation*}
whenever $\|\textrm{prox}_P(\w - \nabla h(\w)) - \w\|_2\leq \varepsilon$ and $F(\w)- F_*\leq \zeta$, where $\textrm{prox}_P(\w) = \arg\min_{\u\in\mathcal K}\frac{1}{2}\|\u - \w\|_2^2 + P(\w)$. 
\begin{obs}
 If $F(\w) = h(\w) + P(\w)$ is defined above and satisfies the Luo-Tseng's local error bound condition, it then implies that there exists a sufficiently small $\epsilon'>0$ and $C>0$ such that $\|\w - \w_*\|_2 \leq C(F(\w) - F_*)^{1/2}$ for any $\w\in\mathcal B(\w_*, \epsilon')$.  
 \end{obs}
This observation was established in~\cite[Theorem 4.1]{guoyincalculus2016}. Note that the LGC condition with $\epsilon = G\epsilon'$ and $\theta=1/2$ also implies that $\|\w - \w_*\|_2 \leq C(F(\w) - F_*)^{1/2}$ for any $\w\in\mathcal B(\w_*, \epsilon')$. Nonetheless, Luo-Tseng's local error bound imposes  some smoothness assumption on $h(\w)$. 

The last observation is that the LGC is equivalent to a Kurdyka - \L ojasiewicz inequality (KL), which was proved in~\cite[Theorem 5]{arxiv:1510.08234}. 
\begin{obs}
If $F(\w)$ satisfies a KL inequality, i.e.,  $\varphi'(F(\w)  -F_*)\|\partial^0 F(\w)\|_2\geq 1$ for $\w\in\{\x\in\mathcal K, F(\x) - F_*<\epsilon\}$ with $\varphi(s) = cs^{\theta}$, then LGC~(\ref{eqn:lgc2}) holds, and vice versa. 
\end{obs}
The above KL inequality has been established for continuous semi-algebraic and subanalytic functions~\citep{journals/mp/AttouchBS13,Bolte:2006:LIN:1328019.1328299,arxiv:1510.08234}, which cover a broad family of functions therefore justifying the generality of the LGC. 

Finally, we present a key lemma that can leverage the LGC to control the distance of intermediate solutions to an $\epsilon$-optimal solution, which is due to~\citep{DBLP:journals/corr/arXiv:1512.03107}. %The local error bound  will be explored with the following lemma in the proof of the main theorems. 
\begin{lemma}\label{lem:key}
For any $\w\in\mathcal K$ and $\epsilon>0$, we have
\begin{equation*}
\|\w - \w^\dagger_\epsilon\|_2\leq \frac{dist(\w^\dagger_\epsilon, \mathcal K_*)}{\epsilon}(F(\w) - F(\w_\epsilon^\dagger)),
\end{equation*}
where $\w^\dagger_\epsilon\in\mathcal S_\epsilon$ is the closest point in the $\epsilon$-sublevel set to $\w$ as defined in~(\ref{eqn:xepsilon}). 
\end{lemma}
{\bf Remark:}  In view of LGC,  we can see that $\|\w - \w^\dagger_\epsilon\|_2\leq \frac{c}{\epsilon^{1-\theta}}(F(\w) - F(\w^\dagger_\epsilon))$ for any $\w\in\mathcal K$. } Yang and Lin~\citep{DBLP:journals/corr/arXiv:1512.03107} have leveraged this relationship to improve the convergence of the standard subgradient method.  In this work,  we will build on this relationship to  further  develop novel stochastic optimization algorithms with faster convergence in high probability.  %If the local error bound is tight, the main results presented later reveal that the iteration complexity of the proposed ASSG is scaled by $\frac{B_\epsilon^2}{B^2}$ ignoring the constant compared to that of SSG, where $B$ is the width of the domain or the distance from the initial solution to the optimal set.

\section{Accelerated Stochastic Subgradient Methods under LGC}   \label{sec:main}
In this section, we will present the proposed accelerated stochastic subgradient (ASSG)   methods and establish their improved iteration complexity with a high probability. The key to our development is to control the distance of intermediate solutions to the {\it $\epsilon$-optimal} solution by their functional residuals that are decreasing as the solutions approach the optimal set. It is this decreasing factor that help mitigate the non-vanishing variance issue in the stochastic subgradient. To formally illustrate this, we consider the following  stochastic subgradient update: 
\begin{equation}\label{eqn:sgd}
\w_{\tau+1} = \Pi_{\mathcal K\cap \mathcal B(\w_1, D)}[\w_\tau - \eta\nabla f(\w_\tau; \xi_\tau)].
\end{equation}
Then we present a lemma regarding the update of ~(\ref{eqn:sgd}).
\begin{lemma}\label{thm:ssg}
Given $\w_1\in\mathcal K$, apply $t$ iterations of~(\ref{eqn:sgd}).  For any fixed $\w\in\mathcal K\cap \mathcal B(\w_1, D)$ and  $\delta\in(0,1)$, with a probability at least $1-\delta$, the following inequality holds
\begin{equation*}
F(\wh_{t})  - F(\w)\leq  \frac{\eta G^2}{2}+ \frac{\|\w_1 - \w\|_2^2}{2\eta t}  + \frac{4GD\sqrt{3\log(\frac{1}{\delta})}}{\sqrt{t}},
\end{equation*}
where $\wh_t = \sum_{\tau=1}^t\w_t/t$. 
\end{lemma}
{\bf Remark: }  The proof of the above lemma follows similarly as that of Lemma 10  in~\citep{hazan-20110-beyond}. We note that the last term is due to the variance of the stochastic subgradients. In fact, due to the non-smoothness nature of the problem the variance of the stochastic subgradients cannot be reduced, we therefore propose to address this issue by reducing $D$ in light of the inequality in Lemma~\ref{lem:key}. 

The updates in (\ref{eqn:sgd}) can be also understood as  approximately solving the original problem in the neighborhood of $\w_1$. In light of this, we will also develop a regularized variant of the proposed method. %In the sequel, all omitted proofs can be found in the supplement. 

\subsection{Accelerated Stochastic Subgradient Method: the Constrained variant (ASSG-c)}
\begin{algorithm}[t]
\caption{ASSG-c($\w_0, K, t, D_1, \epsilon_0$)} \label{alg:rssg}
\begin{algorithmic}[1]
\STATE \textbf{Input}: $\w_0\in\mathcal K$, $K$, $t$,  $\epsilon_0$ and $D_1\geq \frac{c\epsilon_0}{\epsilon^{1-\theta}}$%the  number of stages $K$, the number of iterations $t$ per stage, and the initial solution $\w_0$, $\eta_1=\epsilon_0/(3G^2)$ and $D_1\geq \frac{c\epsilon_0}{\epsilon^{1-\theta}}$
\STATE Set $\eta_1=\epsilon_0/(3G^2)$
\FOR{$k=1,\ldots, K$}
\STATE Let $\w^k_1 = \w_{k-1}$
\FOR{$\tau=1,\ldots, t-1$}
    \STATE  $\w^k_{\tau+1} = \Pi_{\mathcal K\cap \mathcal B(\w_{k-1},D_k)}[\w^k_{\tau} - \eta_k \partial f(\w^k_{\tau}; \xi^k_\tau)]$
   \ENDFOR
\STATE Let $\w_k = \frac{1}{t}\sum_{\tau=1}^t\w^k_\tau$
\STATE Let $\eta_{k+1} = \eta_k/2$ and $D_{k+1} = D_k/2$. 
\ENDFOR
\STATE \textbf{Output}:  $\w_K$
\end{algorithmic}
\end{algorithm}
In this subsection, we present the constrained variant of ASSG that iteratively solves the original problem approximately in an explicitly constructed  local neighborhood of the recent historical solution. The detailed steps are presented in Algorithm~\ref{alg:rssg}. We refer to this variant as ASSG-c.  The algorithm runs in stages  and each stage runs $t$ iterations of updates similar to~(\ref{eqn:sgd}). Thanks to Lemma~\ref{lem:key}, we gradually decrease the radius $D_k$ in a stage-wise manner. % except that the intermediate solutions are projected into the intersection of the problem domain $\mathcal K$ and a ball $\mathcal B(\w_{k-1}, D_k)$.  The radius $D_k$ geometrically decreases as $\w_{k-1}$ approaches to the optimal set. T
The step size keeps the same during each stage and geometrically decreases between stages.  We notice that ASSG-c is similar to the Epoch-GD method by~\cite{hazan-20110-beyond} and the (multi-stage) AC-SA method with domain shrinkage by~\cite{DBLP:journals/siamjo/Lan13b} for stochastic strongly convex optimization, and is also similar to the restarted subgradient method (RSG) proposed by~\cite{DBLP:journals/corr/arXiv:1512.03107}.  
However, the difference between ASSG and Epoch-GD/AC-SA  lies at the initial radius $D_1$ and the number of iterations per-stage, which is due to difference between the strong convexity assumption and  Lemma~\ref{lem:key}. % is that the number of iterations $t$ for all stages are the same in ASSG while it geometrically increases between stages in Epoch-GD/AC-SA. 
Compared to RSG, the solutions updated along gradient direction in ASSG are projected back into a local neighborhood around $\w_{k-1}$, which is the key to establish the faster convergence of ASSG.  
The convergence of  ASSG-c is presented in the theorem below.
\begin{theorem}\label{thm:RSSG}
Suppose Assumptions \ref{ass:1} and \ref{ass:2}  hold for a target $\epsilon\ll 1$. Given $\delta\in(0,1)$, let $\tilde\delta = \delta/K$,  $K = \lceil \log_2(\frac{\epsilon_0}{\epsilon})\rceil$, $D_1\geq \frac{c\epsilon_0}{\epsilon^{1-\theta}}$ and $t$ be the smallest integer such that $t \geq \max\{9, 1728\log(1/\tilde\delta)\} \frac{G^2D_1^2}{\epsilon_0^2}$. Then ASSG-c guarantees that, with a probability $1-\delta$, %there exists $k \in \{1, 2, \dots, K \}$ such that
%\begin{equation*}
$F(\w_K) - F_* \leq 2 \epsilon$. 
%\end{equation*}
As a result, the iteration complexity of ASSG-c for achieving an $2\epsilon$-optimal solution with a high probability $1-\delta$ is $ O (c^2G^2\lceil \log_2(\frac{\epsilon_0}{\epsilon})\rceil\log(1/\delta)/\epsilon^{2(1-\theta)})$ provided $D_1=O(\frac{c\epsilon_0}{\epsilon^{(1-\theta)}})$.
\end{theorem}
{\bf Remark:} It is notable that the faster local growth rate $\theta$ implies the faster global convergence, i.e., lower iteration complexity. In light of the lower bound presented in~\citep{DBLP:conf/icml/RamdasS13} under a GGC, our iteration complexity under the LGC is optimal up to at most a logarithmic factor. It is worth mentioning that unlike traditional high-probability analysis of SSG that usually requires the domain to be bounded, the convergence analysis of ASSG does not rely on such a condition. Furthermore, the iteration complexity of ASSG has a better dependence on the quality of the initial solution or the size of domain if it is bounded. In particular,  if we let $\epsilon_0 = GB$ assuming $dist(\w_0, \mathcal K_*)\leq B$, though this is not necessary in practice, then the iteration complexity of ASSG has only a logarithmic dependence on the distance of the initial solution to the optimal set, while that of SSG has a quadratic dependence on this distance. The above theorem requires a target precision $\epsilon$ in order to set $D_1$. In Section~\ref{subsec:prac}, we alleviate this requirement to make the algorithm more practical.   %It is notable that Epoch-GD~\citep{hazan-20110-beyond} and AC-SA~\citep{DBLP:journals/siamjo/Lan13b} have a complexity of 
%$O(\log(1/\delta)/\epsilon)$ and $O((\log(1/\delta))^2/\epsilon)$ respectively. Compared to their results, our ASSG method can have a better complexity if $\theta>1/2$. % and we assume local error bound condition - a much more general condition than strong convexity.
%To prove  Theorem~\ref{thm:RSSG}, we first present a lemma regarding each stage of ASSG-c. 
Next, we prove Theorem~\ref{thm:RSSG} regarding the convergence of ASSG-c.
\begin{proof}%[Proof of Theorem~\ref{thm:RSSG}]
%We assume that the stage solutions $\w_k \notin \mathcal S_{2\epsilon}$ for $k = 1, 2, \dots, K-1$ and induce how many stages are sufficient for reaching to a point in $2\epsilon$-sublevel set with a high probability. Let $\w_{k,\epsilon}^\dag$ denote the closest point to $\w_k$ in $\mathcal S_\epsilon$. Since $\mathcal S_{\epsilon}\subseteq \mathcal S_{2\epsilon}$, therefore $\w_k \not\in \mathcal S_\epsilon, k=1,\ldots, K-1$ and consequentially  $F(\w_{k,\epsilon}^\dagger)=F_* + \epsilon$ (using the KKT condition).
Let $\w_{k,\epsilon}^\dag$ denote the closest point to $\w_k$ in $\mathcal S_\epsilon$. Define $\epsilon_k = \frac{\epsilon_0}{2^k}$. Note that $D_k = \frac{D_1}{2^{k-1}} \geq\frac{c\epsilon_{k-1}}{\epsilon^{1-\theta}}$ and $\eta_k = \frac{\epsilon_{k-1}}{3G^2}$. We will show by induction that $F(\w_k) - F_*\leq \epsilon_k +\epsilon$ for $k=0,1,\dots$ with a high probability, which leads to our conclusion when $k=K$. The inequality holds obviously for $k=0$. Conditioned on $F(\w_{k-1}) - F_*\leq \epsilon_{k-1} + \epsilon$, we will show that $F(\w_k) - F_*\leq \epsilon_k +\epsilon$ with a high probability. By Lemma~\ref{lem:key}, we have
\begin{equation} \label{eqn:assgc:leb}
\|\w^\dagger_{k-1,\epsilon} - \w_{k-1} \|_2 \leq \frac{c}{\epsilon^{1-\theta}} (F( \w_{k-1}) - F(\w_{k-1,\epsilon}^\dagger)) \leq  \frac{c \epsilon_{k-1}}{\epsilon^{1-\theta}}\leq D_k.
\end{equation}
We apply Lemma~\ref{thm:ssg} to the $k$-th stage of Algorithm~\ref{alg:rssg} conditioned on randomness in previous stages. With a probability $1-\tilde\delta$  we have
\begin{equation}\label{eqn:assgc}
F(\w_k)  - F(\w_{k-1,\epsilon}^\dagger)\leq  \frac{\eta_k G^2}{2}+ \frac{\|\w_{k-1} - \w_{k-1,\epsilon}^\dagger\|_2^2}{2\eta_k t}  + \frac{4GD_k\sqrt{3\log(1/\tilde\delta)}}{\sqrt{t}}.
\end{equation}
%%%%%%%%%%%%%%%%%%%%%%%%%%%%%%%%%%%%%%%%%%%%%%%%%%%%%%%%%%%%%%%%%%%%%%%%%%%%%%%%%%%%
%%%%%%%%%%%%%%%%%%%%%%%%%%%%%%%%%%%%%%%%%%%%%%%%%%%%%%%%%%%%%%%%%%%%%%%%%%%%%%%%%%%%
%%%%%%%%%%%%%%%%%%%%%%%%%%%%%%%%%%%%%%%%%%%%%%%%%%%%%%%%%%%%%%%%%%%%%%%%%%%%%%%%%%%%
\iffalse
We now consider two cases for $\w_{k-1}$. First, we assume $F(\w_{k-1}) - F_* \leq \epsilon$, i.e. $\w_{k-1}\in\mathcal S_{\epsilon}$. Then we have $\w_{k-1,\epsilon}^\dagger=\w_{k-1}$ and
\begin{equation*}
F(\w_k)  - F(\w_{k-1,\epsilon}^\dagger) \leq  \frac{\eta_k G^2}{2} + \frac{4GD_k\sqrt{3\log(1/\tilde\delta)}}{\sqrt{t}} \leq  \frac{\epsilon_{k}}{3} +  \frac{\epsilon_{k-1}}{6} = \frac{2\epsilon_{k}}{3}.
\end{equation*}
The second inequality using the fact that $\eta_k = \frac{2\epsilon_k}{3G^2}$ and $t \geq 1728\log(1/\tilde\delta) \frac{G^2D_1^2}{\epsilon_0^2}$. 
As a result,
\begin{equation*}
F(\w_k)  - F_* \leq F(\w_{k-1,\epsilon}^\dagger)  -F_* + \frac{2\epsilon_{k}}{3} \leq \epsilon + \epsilon_k. 
\end{equation*}
Next, we consider $F(\w_{k-1}) - F_* > \epsilon$, i.e. $\w_{k-1}\notin\mathcal S_{\epsilon}$. Then we have $F(\w_{k-1,\epsilon}^\dagger)-F_* = \epsilon$. 
\fi
%%%%%%%%%%%%%%%%%%%%%%%%%%%%%%%%%%%%%%%%%%%%%%%%%%%%%%%%%%%%%%%%%%%%%%%%%%%%%%%%%%%%
%%%%%%%%%%%%%%%%%%%%%%%%%%%%%%%%%%%%%%%%%%%%%%%%%%%%%%%%%%%%%%%%%%%%%%%%%%%%%%%%%%%%
%%%%%%%%%%%%%%%%%%%%%%%%%%%%%%%%%%%%%%%%%%%%%%%%%%%%%%%%%%%%%%%%%%%%%%%%%%%%%%%%%%%%
Combining (\ref{eqn:assgc:leb}) and (\ref{eqn:assgc}), we get
\begin{equation*}
F(\w_k)  - F(\w_{k-1,\epsilon}^\dagger) \leq  \frac{\eta_k G^2}{2}+ \frac{D_k^2}{2\eta_k t}   + \frac{4GD_k\sqrt{3\log(1/\tilde\delta)}}{\sqrt{t}}.
\end{equation*}
Since $\eta_k = \frac{2\epsilon_k}{3G^2}$ and $t \geq \max\{9, 1728\log(1/\tilde\delta)\} \frac{G^2D_1^2}{\epsilon_0^2}$, we have each term in the R.H.S of above inequality bounded by $\epsilon_k/3$. 
%\begin{equation*}
% \frac{\eta_kG^2}{2}=\frac{\epsilon_k}{3},\\
% \frac{D_k^2}{2\eta_k t}  \leq \frac{(D_1/2^{k-1})^2}{2\epsilon_{k-1}/(3G^2)} \frac{\epsilon_0^2}{9G^2D_1^2} =  \frac{\epsilon_k}{3},\\
%  \frac{4GD_k\sqrt{3\log(1/\tilde\delta)}}{\sqrt{t}} \leq  \frac{4G(D_1/2^{k-1})\sqrt{3\log(1/\tilde\delta)} \epsilon_0}{GD_1\sqrt{1728\log(1/\tilde\delta)}} = \frac{\epsilon_k}{3}.
%\end{equation*} 
As a result,
%\begin{equation*}
%    F(\w_k)- F(\w_{k-1,\epsilon}^\dagger)\leq \epsilon_k\Rightarrow F(\w_k) - F_*\leq \epsilon_k + \epsilon.
%\end{equation*}
%with a probability $1-\tilde\delta$.
\begin{equation*}
F(\w_k) - F(\w_{k-1, \epsilon}^\dagger) \leq  \epsilon_k,
\end{equation*}
which together with the fact that $ F(\w^\dagger_{k-1,\epsilon})-F_*\leq\epsilon$ by definition of $\w^\dagger_{k-1,\epsilon}$ implies
\begin{equation*}
F(\w_k) - F_* \leq \epsilon  + \epsilon_k.
\end{equation*}
Therefore by induction,  with a probability at least $(1-\tilde\delta)^K$ we have 
\begin{equation*}
F(\w_K) - F_*\leq \epsilon_K + \epsilon \leq 2\epsilon. 
\end{equation*}
Since $\tilde \delta = \delta / K$, then $(1-\tilde \delta)^K \geq 1 - \delta$ and we complete the proof. 
\end{proof}
Theorem~\ref{thm:RSSG} shows the high probability convergence bound for ASSG-c. We also prove the following expectational convergence bound, which is an immediate consequence of 
Theorem~\ref{thm:RSSG}. Its proof is provided in~\ref{app:cor:assgc}.
\begin{corollary}\label{cor:ASSGc}
 Suppose Assumptions \ref{ass:1} and \ref{ass:2}  hold for a target $\epsilon\ll 1$.  %Assumption  \ref{ass:1}  holds and $F(\w)$ obeys the local error bound condition. 
 Given $\delta\in(0,1)$, let $\delta \leq \frac{\epsilon}{2GD_1+\epsilon_0}$,  $K = \lceil \log_2(\frac{\epsilon_0}{\epsilon})\rceil$, $D_1\geq \frac{c\epsilon_0}{\epsilon^{1-\theta}}$ and $t$ be the smallest integer such that $t \geq \max\{9, 1728\log(K/\delta)\} \frac{G^2D_1^2}{\epsilon_0^2}$. Then ASSG-c achieves that $\mathbb E \left[F(\w_K) - F_*\right] \leq 2 \epsilon$ using 
 at most $ O \left(\lceil \log_2(\frac{\epsilon_0}{\epsilon})\rceil \log\left(\frac{2GD_1+\epsilon_0}{\epsilon}\right)c^2G^2/\epsilon^{2(1-\theta)}\right)$ iterations provided $D_1=O(\frac{c\epsilon_0}{\epsilon^{(1-\theta)}})$.
\end{corollary}

\subsection{Accelerated Stochastic Subgradient Method: the Regularized variant (ASSG-r)}
One potential issue of ASSG-c is that the projection into the intersection of the problem domain and an Euclidean ball  might increase the computational cost per-iteration  depending on the problem domain $\mathcal K$. To address this issue, we present a regularized variant of ASSG. Before delving into the details of ASSG-r (Algorithm~\ref{alg:2}), we first present a common strategy that solves the non-strongly convex problem~(\ref{eqn:psg}) by stochastic strongly convex optimization.  The basic idea is from the classical deterministic \emph{proximal point algorithm} \citep{rockafellar76} which adds a strongly convex regularizer  to the original problem and solve the resulting proximal problem. In particular, we construct a new problem
\begin{equation*}%\label{eqn:R}
\min_{\w \in \mathcal K}\Fh(\w) = F(\w) +  \frac{1}{2\beta}\|\w - \w_1\|_2^2,
\end{equation*}
where $\w_1\in\mathcal K$ is called the regularization reference point. Let $\wh_*$ denote the optimal solution to the above problem given $\w_1$.   It is easy to know $\Fh(\w)$ is a $\frac{1}{\beta}$-strongly convex function on $\mathcal K$. There are many stochastic methods can be used to solve the above strongly convex optimization problem with an $\widetilde O(\beta/T)$ convergence, including stochastic subgradient, proximal stochastic subgradient~\citep{duchi2010composite}, Epoch-GD~\citep{hazan-20110-beyond}, stochastic dual averaging~\citep{xiao2010dual}, etc. 
We employ the stochastic subgradient method suited for strongly convex problems to solve the above problem. The update is given by 
\begin{equation}\label{eqn:sg}
\w_{t+1} = \Pi_{\mathcal K}[\w'_{t+1}]=\arg\min_{\w\in\mathcal K}\left\|\w -\w'_{t+1}\right\|_2^2,
\end{equation}
where $\w'_{t+1}= \w_t - \eta_t (\partial f(\w_t; \xi_t) + \frac{1}{\beta} (\w_t - \w_1))$, 
and  $\eta_t = \frac{2\beta}{t}$~\footnote{The factor $2$ in the step size is used for proving the high probability convergence.}.  %The steps are presented in Algorithm~\ref{alg:SG}, which will be referred to as SSGS.
We present a lemma below to bound $\|\wh_*-\w_t\|_2$  and $\|\w_t - \w_1\|_2$ by the above update, which will be used in the proof of convergence of ASSG-r for solving~(\ref{eqn:psg}). 
\begin{lemma}\label{lem:b}
For any $t\geq 1$, we have $\|\wh_* - \w_t\|_2\leq 3\beta G$ and $\|\w_t - \w_1\|_2\leq 2\beta G$.
\end{lemma}
{\bf Remark:} The lemma implies that the regularization term implicitly imposes a constraint on the intermediate solutions to center around the regularization reference point, which achieves a similar effect as the ball constraint in Algorithm~\ref{alg:rssg}.  We include its proof in~\ref{app:lemb}. 
%\begin{algorithm}[t]
%\caption{SSG for solving~(\ref{eqn:psg}) with a Strongly convex regularizer: $\textrm{SSGS}(\w_1,\beta, T)$} \label{alg:SG}
%\begin{algorithmic}[1]
%\FOR{$t=1,\ldots, T$}
%    \STATE Let  $\w_{t+1}'= \left(1 - \frac{2}{t}\right)\w_t + \frac{2}{t}\w_1 - \frac{2\beta}{t} \partial f(\w_t; \xi_t) $
%    \STATE Let $\w_{t+1} = \Pi_\mathcal K(\w_{t+1}')$
%\ENDFOR
%\STATE \textbf{Output}:  $\widehat \w_T = \sum_{t=1}^{T}\w_t / T$
%\end{algorithmic}
%\end{algorithm}
\begin{algorithm}[t]
\caption{the ASSG-r algorithm for solving~(\ref{eqn:psg})} \label{alg:2}
\begin{algorithmic}[1]
\STATE \textbf{Input}: $\w_0\in\mathcal K$, $K$, $t$, $\epsilon_0$  and $\beta_1 \geq \frac{2c^2\epsilon_0}{\epsilon^{2(1-\theta)}}$%the number of stages $K$ and  the  number of iterations $t$ per-stage,  the initial solution $\w_0\in\mathcal K$, and $\beta_1 \geq \frac{2c^2\epsilon_0}{\epsilon^{2(1-\theta)}}$
\FOR{$k=1,\ldots, K$}	
   	%\STATE Let $\w_{k} = \textrm{SSGS}(\w_{k-1}, \beta_k, t)$  
	\STATE Let $\w_1^k=\w_{k-1}$
	\FOR{$\tau=1,\ldots, t-1$}
    \STATE Let  $\w_{\tau+1}'= \left(1 - \frac{2}{\tau}\right)\w^k_\tau + \frac{2}{\tau}\w^k_1 - \frac{2\beta_k}{\tau} \partial f(\w^k_\tau; \xi^k_\tau) $
    \STATE Let $\w^k_{\tau+1} = \Pi_\mathcal K(\w_{\tau+1}')$
\ENDFOR   
        \STATE Let $\w_k= \frac{1}{t}\sum_{\tau=1}^t\w^k_\tau$, and  $\beta_{k+1} = \beta_k/2$ 
   \ENDFOR
\STATE \textbf{Output}:  $\w_K$
\end{algorithmic}
\end{algorithm}

Next, we present a high probability convergence bound, whose proof can be found in~\ref{app:SG}.
\begin{lemma}\label{lemma:ssgs}
Given $\w_1\in\mathcal K$, apply $T$-iterations of~(\ref{eqn:sg}).  For any fixed $\w\in\mathcal K$, $\delta\in(0,1)$, and $T\geq 3$, with a probability at least $1-\delta$,   following inequality holds
\begin{equation*}
F(\wh_{T})  - F(\w)\leq  \frac{\|\w-\w_1\|_2^2}{2\beta} + \frac{34\beta G^2 \left( 1+\log T + \log(4\log T / \delta)\right)}{T},
\end{equation*}
where $\wh_t = \sum_{\tau=1}^t\w_t/t$. 
\end{lemma}
%Next, we present a high probability convergence bound  of the averaged solution of the updates~(\ref{eqn:sg})  for optimizing  $F(\w)$. 
%%\begin {theorem}\label{thm:SG}
%%Suppose  Assumption \ref{ass:1}.c holds. Let $\wh_T$ be the returned solution of Algorithm~\ref{alg:SG}.  Given $\w_1\in\mathcal K$, $\delta<1/e$ and $T\geq 3$, with a high probability $1-\delta$ we have
%%\begin{equation}\label{eqn:thm:PSG}
%%\Fh(\widehat \w_T) - \Fh(\wh_*) \leq   \frac{34\beta G^2(1+\log T + \log(4\log T/\delta)}{T}.  
%%\end{equation}
%%\end {theorem}
%%Note that the constant factor $34$ is not optimized. Its proof can be found in Appendix \ref{app:SG}. The corollary below will be used in our development. 
%\begin{corollary}\label{cor:SSGS}
%Suppose  Assumption \ref{ass:1}.c holds. Let $\wh_T$ be the averaged solution of the updates~(\ref{eqn:sg}), i.e., $\wh_T = \sum_{t=1}^T\w_t/T$ . Given $\w_1\in\mathcal K$, $\delta<1/e$ and $T\geq 3$, for any $\w\in\mathcal K$   with a high probability $1-\delta$ we have, 
%\begin{equation}\label{eqn:cor:SG}
%F(\widehat \w_T)  - F(\w)  \leq \frac{1}{2\beta}\|\w - \w_1\|_2^2\nonumber \\
% +  \frac{34\beta G^2(1+\log T + \log(4\log T/\delta))}{T}.  
%\end{equation}
%\end{corollary}
{\bf Remark:} From the above result, we can see that one can set $\beta$ to be a large value to ensure convergence. In particular, by assuming that $dist(\w_1, \mathcal K_*)\leq B$, we can set $\beta= \frac{B^2}{\epsilon}$ and $T\geq \frac{68G^2B^2(1+\log(4\log T/\delta)+\log T)}{\epsilon^2}$ so as to obtain $F(\wh_T) - F_*\leq \epsilon$ with a high probability $1-\delta$, which yields the same order of  iteration complexity to SSG for directly solving~(\ref{eqn:psg}). 

Recall that the main iteration of the proximal point algorithm \citep{rockafellar76} is
\begin{equation} \label{eqn:PPA}
\w_k\approx\arg\min_{\w\in\mathcal K} F(\w) +  \frac{1}{2\beta_k}\|\w - \w_{k-1}\|_2^2,
\end{equation}
where $\w_k$ approximately solves the minimization problem above with $\beta_k$ changing with $k$. With the same idea, our regularized variant of ASSG generates $\w_k$ from stage $k$ by solving the minimization problem (\ref{eqn:PPA}) approximately using~(\ref{eqn:sg}). The detailed steps are presented in Algorithm~\ref{alg:2}, which starts from a relatively large value of the parameter $\beta=\beta_1$ and gradually decreases $\beta$ by a constant factor  after running  a number of $t$ iterations~(\ref{eqn:sg}) using the solution from the previous stage as the new regularization reference point. Despite of its similarity to the proximal point algorithm, ASSG-r incorporates the LGC into the choices of $\beta_k$ and the number of iterations per-stage and obtains new iteration complexity described below. 
\begin{theorem}\label{thm:RPSG}
Suppose Assumptions \ref{ass:1} and \ref{ass:2}  hold for a target $\epsilon\ll 1$. Given $\delta\in(0,1/e)$, let $\tilde\delta = \delta/K$, $K = \lceil \log_2(\frac{\epsilon_0}{\epsilon})\rceil$,  $\beta_1 \geq \frac{2c^2\epsilon_0}{\epsilon^{2(1-\theta)}}$ and $t$ be the smallest integer such that $t \geq  \max\{3,\frac{136\beta_1G^2(1+\log (4\log t/\tilde\delta)+\log t)}{\epsilon_0}\}$.  Then ASSG-r guarantees that, with a probability $1-\delta$, %there exists $k \in \{1, 2, \dots, K \}$ such that
%\begin{equation*}
$F(\w_K) - F_* \leq 2 \epsilon$. 
%\end{equation*}
As a result, the iteration complexity of ASSG-r for achieving an $2\epsilon$-optimal solution with a high probability $1-\delta$ is $ O (c^2G^2\lceil \log_2(\frac{\epsilon_0}{\epsilon})\rceil\log(1/\delta)/\epsilon^{2(1-\theta)})$ provided $\beta_1=O(\frac{2c^2\epsilon_0}{\epsilon^{2(1-\theta)}})$.
\end{theorem}
%The proof of Theorem~\ref{thm:RPSG} is similar to that of Theorem~\ref{thm:RSSG}, and is included in the supplement. 

{With Lemma~\ref{lemma:ssgs}, the proof of Theorem~\ref{thm:RPSG} is similar to the proof of Theorem~\ref{thm:RSSG}. For completeness, we include it in~\ref{proof:thm:RPSG}.}

\subsection{A Simple Variant of ASSG under GGC}
\begin{algorithm}[t]
\caption{ASSG-s($\w_0, K, t, \epsilon_0$)} \label{alg:assg:wo}
\begin{algorithmic}[1]
\STATE \textbf{Input}: $\w_0\in\mathcal K$, $K$, $t$,  $\epsilon_0$
\STATE Set $\eta_1=\epsilon_0/(3G^2)$
\FOR{$k=1,\ldots, K$}
\STATE Let $\w^k_1 = \w_{k-1}$
\FOR{$\tau=1,\ldots, t-1$}
    \STATE  $\w^k_{\tau+1} = \Pi_{\mathcal K}[\w^k_{\tau} - \eta_k \partial f(\w^k_{\tau}; \xi^k_\tau)]$
   \ENDFOR
\STATE Let $\w_k = \frac{1}{t}\sum_{\tau=1}^t\w^k_\tau$ and $\eta_{k+1} = \eta_k/2$. 
\ENDFOR
\STATE \textbf{Output}:  $\w_K$
\end{algorithmic}
\end{algorithm}
%When the dimenstion of problem $d$ is very large, the additional computation of the projection onto a domain or the gradient of a regularizer in each iterative step makes the ASSG algorithms comutational expensive, comparing to SSG algorithm. 
As a byproduct of similar analysis, we can show that a simpler variant of ASSG without using shrinking domain constraint or increasing regularization can have an improved complexity  in expectation under GGC for $\theta\in (0, 1/2]$. %In this section, we will propose a user-friendly variant of ASSG without domain constraint or regularizer and establish its expectational iteartion complexity under GGC with $\theta\in(0,1/2]$. 
When the problems satisfy GGC with $\theta \in (1/2, 1]$ and $F(\w) - F_*$ is bounded over $\mathcal K$, one can always show that the problem satisfies a GGC with $\theta=1/2$~\citep{xu2017AdaSVRG}. The details of updates are presented in Algorithm~\ref{alg:assg:wo}, which is referred to ASSG-s. The algorithm is almost the same to Algorithm~\ref{alg:rssg} except that the projection is simply done onto the original domain $\mathcal K$ without intersecting with a bounded ball at each epoch. 
At each epoch, the update is exactly the same to the stochastic subgradient update
\begin{equation}\label{eqn:sgd:wo}
\w_{\tau+1} = \Pi_{\mathcal K}[\w_{\tau} - \eta \partial f(\w_{\tau}; \xi_\tau)].
\end{equation}
%We present the detailed updates in Algorithm~\ref{alg:assg:wo},  which is referred to ASSG-s. 
To establish the convergence result, we first need the following lemma, whose proof is included in~\ref{app:lemm:wo}.
\begin{lemma}\label{lemm:wo}
Given $\w_1\in\mathcal K$, apply $t$ iterations of~(\ref{eqn:sgd:wo}).  For any fixed $\w\in\mathcal K$, the following inequality holds
\begin{equation*}
\E [F(\wh_{t})  - F(\w)]\leq  \frac{\eta G^2}{2}+ \frac{\|\w_1 - \w\|_2^2}{2\eta t}, 
\end{equation*}
where $\wh_t = \sum_{\tau=1}^t\w_t/t$. 
\end{lemma}

We then give the convergence result of ASSG-s in the following theorem. 
\begin {theorem}\label{thm:ASSG:wo}
Suppose Assumption \ref{ass:1} holds and $F(\w)$ obeys {GGC~(\ref{eqn:lgc2})  with $\theta\in(0, 1/2]$}. Given $\epsilon > 0$, $K = \lceil \log_2(\frac{\epsilon_0}{\epsilon})\rceil$ and $t$ be the smallest integer such that $t \geq \frac{18c^2G^2}{\epsilon^{2(1-\theta)}}$. Then ASSG-s guarantees that $\E[F(\w_K) - F_*] \leq  \epsilon$. As a result, the iteration complexity of ASSG-s for achieving an $\epsilon$-optimal solution is $ O (c^2G^2\lceil \log_2(\frac{\epsilon_0}{\epsilon})\rceil/\epsilon^{2(1-\theta)})$ in expectation.
\end {theorem}
%{\bf Remark.} Although we consider the GGC with $\theta\in(0,1/2]$, if a GGC condition with $\theta \in (1/2, 1]$ holds, it can be always reduced to the GGC with $\theta=1/2$ provided that $F(\w) - F_*$ is bounded over $\mathcal K$~\citep{xu2017AdaSVRG}.
\begin{proof}
Let define $\epsilon_k = \frac{\epsilon_0}{2^k}$. Note that $\eta_k = \frac{\epsilon_{k-1}}{3G^2}$. We will show by induction that $\E[F(\w_k) - F_*]\leq \epsilon_k$ for $k=0,1,\dots$, which leads to our conclusion when $k=K$. The inequality holds obviously for $k=0$. Conditioned on $\E[F(\w_{k-1}) - F_*]\leq \epsilon_{k-1}$, we will show that $\E[F(\w_k) - F_*]\leq \epsilon_k$. 
By GGC, we have for any $\w_{k-1}\in\mathcal K$,
\begin{equation*} 
\|\w_{k-1} - \w^*_{k-1} \|_2 \leq c (F(\w_{k-1} ) - F_*)^{\theta}.
\end{equation*}
Then by the condition $\E[F(\w_{k-1}) - F_*]\leq \epsilon_{k-1}$, we have
\begin{equation} \label{eqn:assgc:leb:wo}
\E[\|\w_{k-1} - \w^*_{k-1} \|_2] \leq c \epsilon_{k-1}^\theta.
\end{equation}

We apply Lemma~\ref{lemm:wo} to the $k$-th stage of Algorithm~\ref{alg:assg:wo} conditioned on randomness in previous stages. For any $\w_*\in\mathcal K_*$ we have
\begin{equation}\label{eqn:assgc:wo}
\E[F(\w_k)  - F(\w_*)]  \leq  \frac{\eta_k G^2}{2}+ \frac{\E[\|\w_{k-1} - \w_*\|_2^2]}{2\eta_k t}.
\end{equation}
By using {GGC~(\ref{eqn:lgc2}) with $\theta\in(0, 1/2]$} we have
\begin{align*}
\E[F(\w_k)  - F(\w_*)] \leq &   \frac{\eta_k G^2}{2}+ \frac{c^2 \E[F(\w_{k-1}) - F(\w_*)]^{2\theta}}{2\eta_k t}\\
\leq &   \frac{\eta_k G^2}{2}+ \frac{c^2 \{\E[F(\w_{k-1}) - F(\w_*)]\}^{2\theta}}{2\eta_k t}\\
\leq &   \frac{\eta_k G^2}{2}+ \frac{c^2  \epsilon^{2\theta}}{2\eta_k t} \leq \frac{\epsilon_k}{3} + \frac{\epsilon_k}{3} \leq \epsilon_k, 
\end{align*}
where the second inequality uses the concavity of $\E[X^\alpha] \leq \{\E[X]\}^\alpha$ whith $0<\alpha\leq 1$; the fourth inequality using the fact that $\eta_k = \frac{\epsilon_k}{3G^2}$ and $t \geq \frac{18c^2G^2}{\epsilon^{2(1-\theta)}}$ with $\epsilon \leq \epsilon_k$. 
Therefore by induction, we have 
\begin{equation*}
\E[F(\w_K) - F_*]\leq \epsilon_K  \leq \epsilon. 
\end{equation*}
\end{proof}

\section{Practical Variants of ASSG}\label{subsec:prac}
Readers may have noticed that the presented algorithms require appropriately setting up the initial values of $D_1$ or $\beta_1$  or $t$ that depend on potentially unknown $c$ and  unknown $\theta$. 
%{That is to say, the required number of iterations per-stage $t$ for finding an $\epsilon$-optimal solution also depends on parameters $\theta$ and $c$. 
As we show later, the value of $\theta$ is exhibited for many problems. However, the parameter $c$ is usually difficult to estimate, which leads to a challenge to set the value of $t$. Overestimate of $t$ leads to waste of iterations while underestimate of $t$ leads to a less accurate solution so that it may not reach the target level of accuracy.
This section is devoted to more practical variants of ASSG {that can be implemented without knowing parameter $c$ or $\theta$.} %The key idea is to use an increasing squence of $t$ and another level of restarting for ASSG.} % that can be implemented with an underestimated  $D_1$ due to unknown $c$ and even unknown local growth rate  $\theta$. %We also present a proximal variant of ASSG for solving a non-smooth composite optimization problem. 
%All omitted proofs are included in the supplement.  %
For ease of presentation, we focus on the constrained variant of ASSG. Similar extensions can be made for the regularized variant ASSG-r and the simple variant ASSG-s, which are omitted here.  In the following subsections, we divide the problem into two cases: (1) unknown $c$;  (2) unknown $\theta$. 

\subsection{ASSG with unknown $c$}
%\subsection{ASSG-c with Restarting: RASSG-c}
When $c$ is unknown, we present the details of a restarting variant of ASSG in Algorithm~\ref{alg:rassg-c}, to which we refer as RASSG.  When discussing the restarting variants of ASSG-c, ASSG-r and ASSG-s, we refer to them as RSSG-c, RSSG-r, and RSSG-s, respectively,  for clarity.   The key idea  is to use an increasing sequence of $t$ and another level of restarting for ASSG. %In addition, we start with a relatively small $D_1^{(1)}$ (compare with $D_1$ in Algorithm~\ref{alg:rssg}) and then increase it by a multiple factor $2^{1-\theta}$ at each restart of ASSG-c. While for RASSG-r, we start with a relatively small $\beta_1^{(1)}$ (compare with $\beta_1$ in Algorithm~\ref{alg:2}) and then increase it by a multiple factor $2^{2(1-\theta)}$ at each restart of ASSG-r. The convergence analysis for both methods are presented in the following two theorems.
The convergence analysis for RASSG without knowing $c$ is presented in the following theorem.
\begin{algorithm}[t]
\caption{ASSG with Restarting: RASSG} \label{alg:rassg-c}
\begin{algorithmic}[1]
\STATE \textbf{Input}: $\w^{(0)}$, $K$, $D_1^{(1)}$, $t_1$, $\epsilon_0$ and $ \omega \in (0,1]$
\STATE Set $\epsilon_0^{(1)}=\epsilon_0$, $\eta_1=\epsilon_0/(3G^2)$
\FOR{$s=1, 2, \ldots, S$}
\STATE Let $\w^{(s)} = $ASSG-c$(\w^{(s-1)},K,t_s,D_1^{(s)},\epsilon_0^{(s)})$
\STATE Let $t_{s+1} = t_{s}2^{2(1-\theta)}$, $D_1^{(s+1)} = D_1^{(s)}2^{1-\theta}$, and $\epsilon_0^{(s+1)} = \omega \epsilon_0^{(s)}$
%\STATE \textbf{If} a stopping criterion satisfies, \textbf{then} terminate the algorithm
\ENDFOR
\STATE \textbf{Output}:  $\w^{(S)}$
\end{algorithmic}
\end{algorithm}\begin{theorem}[RASSG with unknown $c$] \label{thm:RASSG-c}
Let $\epsilon \leq \epsilon_0/4$, $\omega=1$, and $K = \lceil \log_2(\frac{\epsilon_0}{\epsilon})\rceil$  in Algorithm~\ref{alg:rassg-c}. 
Suppose $D_1^{(1)}$ is sufficiently large so that there exists $\hat\epsilon_1\in[\epsilon, \epsilon_0/2]$, with which $F(\cdot)$ satisfies a LGC~(\ref{eqn:lgc}) on $\mathcal S_{\hat\epsilon_1}$ with $\theta \in (0,1)$ and the constant $c$, and $D_1^{(1)} = \frac{c\epsilon_0}{\hat\epsilon_1^{1-\theta}}$. Let  $\hat\delta = \frac{\delta}{K(K+1)}$, and $t_1 =\max\{9, 1728\log(1/\hat\delta)\} \left( GD_1^{(1)}/\epsilon_0\right)^2$. Then with at most $S = \lceil \log_2(\hat\epsilon_1/\epsilon) \rceil + 1$ calls of ASSG-c,  Algorithm~\ref{alg:rassg-c} finds a solution $\w^{(S)}$ such that $F(\w^{(S)})-F_*\leq 2\epsilon$ with probability $1-\delta$. The total number of iterations of RASSG for obtaining $2\epsilon$-optimal solution is upper bounded by $T_S =  O(\lceil \log_2(\frac{\epsilon_0}{\epsilon})\rceil\log(1/\delta)c^2/\epsilon^{2(1-\theta)})$.
\end{theorem}
{\bf Remark:}  The above theorem requires a slightly stringent LGC condition on $\mathcal S_{\hat\epsilon_1}$ that is induced by the initial value of $D_1$. If the problem satisfies the LGC with $\theta=1$, we can give a slightly smaller value for $\theta$ in order to run Algorithm~\ref{alg:rassg-c}. If the target precision $\epsilon$ is not specified, we can give it a sufficiently small value $\epsilon'$ (e.g., the machine precision) that only affects $K$ marginally.  The corresponding iteration complexity for achieving an $\epsilon$-optimal solution is given by $O(\lceil \log_2(\frac{\epsilon_0}{\epsilon'})\rceil\log(1/\delta)/\epsilon^{2(1-\theta)})$. 
 The parameter $\omega\in(0,1]$ is introduced to increase the practical performance of RASSG, which accounts for  decrease of the objective gap of the initial solutions for each call of ASSG-c. 
\begin{proof}
Since  $K =  \lceil \log_2(\frac{\epsilon_0}{\epsilon})\rceil \geq  \lceil \log_2(\frac{\epsilon_0}{\hat\epsilon_1})\rceil$,  $D_1^{(1)} = \frac{c\epsilon_0}{\hat\epsilon_1^{1-\theta}}$, and $t_1 =\max\{9, 1728\log(1/\hat\delta)\} \left(\frac{ GD_1^{(1)}}{\epsilon_0}\right)^2$, following the proof of Theorem~\ref{thm:RSSG}, we can show that with a probability $1-\frac{\delta}{K+1}$,
\begin{equation} \label{eqn:RASSG-c}
F(\w^{(1)}) - F_* \leq 2\hat\epsilon_1.
\end{equation}
By running ASSG-c starting from $\w^{(1)}$ which satisfies (\ref{eqn:RASSG-c}) with $K =  \lceil \log_2(\frac{\epsilon_0}{\epsilon})\rceil \geq  \lceil \log_2(\frac{2\hat\epsilon_1}{\hat\epsilon_1/2})\rceil$, $D_1^{(2)} = \frac{c\epsilon_0}{(\hat\epsilon_1/2)^{1-\theta}} \geq \frac{c2\hat\epsilon_1}{(\hat\epsilon_1/2)^{1-\theta}}$, and $t_2 = \max\{9, 1728\log(1/\hat\delta)\} \left( GD_1^{(2)}/\epsilon_0\right)^2$, Theorem~\ref{thm:RSSG} ensures that
\begin{equation*} 
F(\w^{(2)}) - F_* \leq \hat\epsilon_1
\end{equation*}
with a probability at least $(1-\delta/(K+1))^2$.
By continuing the process, with $S =  \lceil \log_2(\hat\epsilon_1/\epsilon) \rceil + 1$ we can prove that with a probability at least $(1-\delta/(K+1))^S\geq 1-\delta\frac{S}{K+1}\geq 1- \delta$,
\begin{equation*} 
F(\w^{(S)}) - F_* \leq 2\hat\epsilon_1/2^{S-1} \leq 2\epsilon.
\end{equation*}
The total number of iterations for the $S$ calls of ASSG-c is bounded by
\begin{equation*} 
T_S = K\sum_{s=1}^{S}T_s = K\sum_{s=1}^{S}t_12^{2(s-1)(1-\theta)} = Kt_12^{2(S-1)(1-\theta)}\sum_{s=1}^{S} \left( 1/2^{2(1-\theta)}\right)^{S-s} \\ 
\end{equation*}
\begin{equation*} 
\leq  \frac{Kt_12^{2(S-1)(1-\theta)}}{1-1/2^{2(1-\theta)}} \leq  O\left( Kt_1 \left( \frac{\hat\epsilon_1}{\epsilon}\right)^{2(1-\theta)} \right) \leq \widetilde O(\log(1/\delta)/\epsilon^{2(1-\theta)}).
\end{equation*}
\end{proof}
{
%To help readers understand the result in Theorem~\ref{thm:RASSG-c}, 
As a corollary of the above theorem, we present a result of RASSG for problems satisfying GGC with $\theta=1/2$ but without knowing the value of $c$ (or satisfying strong convexity  but without knowing the strong convexity parameter), which is of interest to a broad audience who are familiar with stochastic strongly convex optimization.   It has been shown many machine learning problems satisfy GGC with $\theta=1/2$ (see examples presented in Section~\ref{sec:app}). Almost all existing algorithms and analysis for stochastic strongly convex optimization or problems satisfying GGC with $\theta=1/2$ require knowing the value of strong convexity parameter in order to run the algorithms~\citep{hazan-20110-beyond,RakhlinSS12}. The result is presented below. 
 %We give a following corollary for Theorem~\ref{thm:RASSG-c}. %for a special case of GGC with $\theta=1/2$. 
\begin{corollary}
Suppose $F(\cdot)$ satisfies a GGC on $\mathcal K$ with $\theta=1/2$ and some unknown constant $c>0$. Let $\epsilon \leq \epsilon_0/4$, $\omega=1$, and $K = \lceil \log_2(\frac{\epsilon_0}{\epsilon})\rceil$  in Algorithm~\ref{alg:rassg-c}. 
Suppose $D_1^{(1)}$ is sufficiently large so that there exists $\hat\epsilon_1\in[\epsilon, \epsilon_0/2]$ such that $D_1^{(1)} = \frac{c\epsilon_0}{\sqrt{\hat\epsilon_1}}$. Let  $\hat\delta = \frac{\delta}{K(K+1)}$, and $t_1 =\max\{9, 1728\log(1/\hat\delta)\} \left( GD_1^{(1)}/\epsilon_0\right)^2$. Then with at most $S = \lceil \log_2(\hat\epsilon_1/\epsilon) \rceil + 1$ calls of ASSG-c,  Algorithm~\ref{alg:rassg-c} finds a solution $\w^{(S)}$ such that $F(\w^{(S)})-F_*\leq 2\epsilon$ with probability $1-\delta$. The total number of iterations of RASSG for obtaining $2\epsilon$-optimal solution is upper bounded by $T_S =  \widetilde O(\log(1/\delta)c^2/\epsilon)$.
\end{corollary}
}
{\bf Remark:} It is notable that when the objective function is $\lambda$-strongly convex, then $c^2 = 1/\lambda$ and the above complexity  $\widetilde O(\log(1/\delta)/\lambda\epsilon)$ is optimal up to a logarithmic factor. The advantage of RASSG over previous stochastic algorithms for strongly convex optimization is that RASSG does not need to know the value of strong convexity parameter. 

\subsection{ASSG with unknown $\theta$}
When $\theta$ is unknown, we can set $\theta = 0$. Then the problem will satisfy the LGC~(\ref{eqn:lgc}) with $\theta=0$ and $c = B_{\varepsilon}$ with any $\varepsilon \geq \epsilon$, where $B_{\varepsilon} = \max_{\w\in\mathcal L_{\varepsilon}} \min_{\v\in\mathcal K_*}\|\w-\v\|_2$ is the maximum distance between the points in the $\varepsilon$-level set $\mathcal L_{\varepsilon}$ and the optimal set $\mathcal K_*$.  The following theorem states the convergence result. % and its proof is is very similar to that of Theorem~\ref{thm:RASSG-c} except that we set $\theta=0$ and replace $c$ by $B_{\hat\epsilon_1}$.
\begin{theorem}[RASSG with unknown $\theta$] \label{thm:RASSG-c2}
Let $\theta=0$,  $\epsilon \leq \epsilon_0/4$ , $\omega=1$, and $K = \lceil \log_2(\frac{\epsilon_0}{\epsilon})\rceil$ in Algorithm~\ref{alg:rassg-c}. Assume $D_1^{(1)}$ is sufficiently large  so that there exists $ \hat \epsilon_1\in[\epsilon, \epsilon_0/2]$ rendering that $D_1^{(1)} = \frac{B_{\hat\epsilon_1}\epsilon_0}{\hat\epsilon_1}$. Let $\hat\delta = \frac{\delta}{K(K+1)}$, and $t_1 =\max\{9, 1728\log(1/\hat\delta)\} \left( GD_1^{(1)}/\epsilon_0\right)^2$. Then with at most $S = \lceil \log_2(\hat\epsilon_1/\epsilon) \rceil + 1$ calls of ASSG-c, Algorithm~\ref{alg:rassg-c} finds a solution $\w^{(S)}$ such that $F(\w^{(S)})-F_*\leq 2\epsilon$. The total number of iterations of RASSG for obtaining $2\epsilon$-optimal solution is upper bounded by $T_S = O( \lceil \log_2(\frac{\epsilon_0}{\epsilon})\rceil \log(1/\delta) \frac{G^2 B_{\hat\epsilon_1}^2}{\epsilon^{2}})$.
\end{theorem}
%\begin {theorem}[RASSG-r, unkown $\theta$]\label{thm:RASSG-r2}
%Suppose $\epsilon \leq \epsilon_0/4$ and $K = \lceil \log_2(\frac{\epsilon_0}{\epsilon})\rceil$. 
%Let $\beta_1^{(1)} $ in Algorithm~\ref{alg:rassg-r} be large enough so that there exists $ \hat \epsilon_1$ satisfying $\epsilon \leq \hat \epsilon_1 %\leq \epsilon_0/2$ with
%$\beta_1^{(1)} = \frac{2B_{\epsilon'}^2\epsilon_0}{\hat\epsilon_1^2}$.
%Let $S = \lceil \log_2(\hat\epsilon_1/\epsilon) \rceil + 1$, $\hat\delta = \frac{\delta}{KS}$, and $t_1 = \frac{136\beta_1^{(1)}G^2(1+\log (4\log (t_12^{2S})/\hat\delta)+\log (t_12^{2S}))}{\epsilon_0} \geq 3$. Then with at total number of $S$ recalls of RASSG-c in Algorithm~\ref{alg:rassg-r}, we find a solution $\w^{(S)}$ such that $F(\w^{(S)})-F_*\leq 2\epsilon$. The total number of iterations of RASSG-c for obtaining $2\epsilon$-optimal solution is upper bounded by $T_S = \widetilde O(\log(1/\delta)/\epsilon^{2(1-\theta)})$.
%\end {theorem}
{\bf Remark:} The Lemma~\ref{lem:key2}  shows that $\frac{B_\epsilon}{\epsilon}$ is a monotonically decreasing function in terms of $\epsilon$, which guarantees the existence of $\hat\epsilon_1$ given a sufficiently large $D_1^{(1)}$.  The iteration complexity  of RASSG could be still better with a smaller factor $B_{\hat\epsilon_1}$ than  the $B$ in the iteration complexity of SSG  (see (\ref{eqn:ssgi})), where $B$ is the domain size or the distance of initial solution to the optimal set.

\begin{proof}
The proof is similar to the proof of Theorem \ref{thm:RASSG-c}, and we reprove it for completeness. It is easy to show that $t_1\geq \frac{136 \beta_1^{(1)} G^2(1+\log (4\log t_1/\hat\delta)+\log t_1)}{\epsilon_0}$.
Following the proof of Theorem~\ref{thm:RPSG}, we then can show that with a probability $1-\frac{\delta}{S}$,
\begin{equation} \label{eqn:RASSG-r}
F(\w^{(1)}) - F_* \leq 2\hat\epsilon_1
\end{equation}
with $K =  \lceil \log_2(\frac{\epsilon_0}{\epsilon})\rceil \geq  \lceil \log_2(\frac{\epsilon_0}{\hat\epsilon_1})\rceil$ and $\beta_1^{(1)} = \frac{2c^2\epsilon_0}{\hat\epsilon_1^{2(1-\theta)}}$. By running ASSG-r starting from $\w^{(1)}$ which satisfies (\ref{eqn:RASSG-r}) with $K =  \lceil \log_2(\frac{\epsilon_0}{\epsilon})\rceil \geq  \lceil \log_2(\frac{2\hat\epsilon_1}{\hat\epsilon_1/2})\rceil$, $t_2 = t_12^{2(1-\theta)} \geq \frac{136\beta_1^{(2)} G^2(1+\log (4\log t_2/\hat\delta)+\log t_2)}{\epsilon_0}$ and $\beta_1^{(2)} = \frac{2c^2\epsilon_0}{(\hat\epsilon_1/2)^{2(1-\theta)}} \geq \frac{2c^2\hat\epsilon_1/2}{(\hat\epsilon_1/2)^{2(1-\theta)}}$, Theorem~\ref{thm:RPSG} ensures that
\begin{equation*} 
F(\w^{(2)}) - F_* \leq \hat\epsilon_1
\end{equation*}
with a probality at least $(1-\delta/S)^2$.
By continuing the process, with $S =  \lceil \log_2(\hat\epsilon_1/\epsilon) \rceil + 1$, we can prove that with a probality at least $(1-\delta/S)^S\geq 1-\delta$
\begin{equation*} 
F(\w^{(S)}) - F_* \leq 2\hat \epsilon_1/2^{S-1} \leq 2\epsilon
\end{equation*}
The total number of iterations for the $S$ calls of ASSG-c is bounded by
\begin{equation*} 
T_S = K\sum_{s=1}^{S}T_s = K \sum_{s=1}^{S}t_12^{2(s-1)(1-\theta)} = Kt_12^{2(S-1)(1-\theta)}\sum_{s=1}^{S} \left( 1/2^{2(1-\theta)}\right)^{S-s} \\ 
% \leq Kt_12^{2(S-1)(1-\theta)} \frac{1}{1-1/2^{2(1-\theta)}} \leq  O\left(K t_1 \left( \frac{\hat\epsilon_1}{\epsilon}\right)^{2(1-\theta)} \right) \leq \widetilde O(\log(1/\delta)/\epsilon^{2(1-\theta)})
\end{equation*}
\begin{equation*} 
%T_S = K\sum_{s=1}^{S}T_s = K \sum_{s=1}^{S}t_12^{2(s-1)(1-\theta)} = Kt_12^{2(S-1)(1-\theta)}\sum_{s=1}^{S} \left( 1/2^{2(1-\theta)}\right)^{S-s} \\ 
 \leq  \frac{Kt_12^{2(S-1)(1-\theta)}}{1-1/2^{2(1-\theta)}} \leq  O\left(K t_1 \left( \frac{\hat\epsilon_1}{\epsilon}\right)^{2(1-\theta)} \right) \leq \widetilde O(\log(1/\delta)/\epsilon^{2(1-\theta)})
\end{equation*}
\end{proof}
Finally, we make several remarks about the Algorithm~\ref{alg:rassg-c}: (1) if $\theta = 1$, in order to obtain an increasing sequence of $t_s$, $\theta$ can be set to a little smaller value than $1$ (for example, 0.95); (2) if $D_1^{(1)}$  in RASSG-c and $\beta_1^{(1)}$  in RASSG-r are determined, the starting number of iterations $t_1$ can be automatically set since $t_1 \propto D_1^{(1)}$ in RASSG-c  and $t_1 \propto \beta_1^{(1)}$ in RASSG-r; (3) after the first call of ASSG, one can re-calibrate the $\epsilon_0$ in the implementation to improve the performance or equivalently tune $\omega$ in practice; (4) the tradeoff is that the stopping criterion for RASSG is not as automatic as ASSG.

%{\bf Remark:} In~\ref{app:mono}, we prove that $\frac{B_\epsilon}{\epsilon}$ is a monotonically decreasing function in terms of $\epsilon$, which guarantees the existence of $\hat\epsilon_1$ given a sufficiently large $D_1^{(1)}$.  The iteration complexity  of RASSG could be still improved with a smaller factor $B_{\hat\epsilon_1}$ compared to $B$ in the iteration complexity of SSG, where $B$ satisfies $\|\w_t-\w_*\|_2\leq B, t = 1, \dots, T$.

\section{Proximal ASSG for Non-smooth Composite Optimization}\label{sec:prox}
To obtain solutions with certain structures, many machine learning problems add a regularizer to the objective function (e.g., adding $\ell_1$ regularizer for sparsity). 
When the regularizers are non-smooth but have closed form of proximal mapping, some proximal algorithms can be employed to solve the regularized problems. 
As an extension of ASSG, in this section, we will present a proximal variant of ASSG for solving the following non-smooth composite optimization problem:
\begin{equation}\label{eqn:composite}
\min_{\w\in\R^d}F(\w)\triangleq \underbrace{\E_{\xi}[f(\w; \xi)]}\limits_{f(\w)} + R(\w), 
\end{equation}
where both $f(\w)$ and $R(\w)$ are non-smooth convex functions. The above problem commonly appears in machine learning, which is also known as regularized risk minimization. We assume that the function $R(\w)$ is simple enough such that the proximal mapping given below is easy to compute
\begin{equation*}%\label{eqn:proxmap}
\textrm{Prox}_{\Omega}^{\eta, R}[\w] = \arg\min_{\u \in \Omega} \frac{1}{2} \|\u-\w\|_2^2 + \eta R(\u),
\end{equation*}
where  $\Omega \subseteq \R^d$ is a bounded ball. An example of $R(\w)$ is the $\ell_1$-norm $R(\w)=\lambda\|\w\|_1$.  We also make the following assumption throughout this section.
\begin{ass}\label{ass:3} For a stochastic optimization problem~(\ref{eqn:composite}), we assume
\begin{enumerate}%[label= (\alph*)]
\item there exist $\w_0\in\R^d$ and $\epsilon_0\geq 0$ such that $F(\w_0) - F_*\leq \epsilon_0$;%\min_{\w\in\mathcal K}F(\w)\leq \epsilon_0$;  
%\item The optimal set $\mathcal K_*$ is a non-empty convex compact set;
\item There exist two constants $G$ and $\rho$ such that $\|\partial f(\w; \xi)\|_2 \leq G$ and $\|\partial R(\w; \xi)\|_2 \leq \rho$.
%\item There exists a constant $D$ such that for any $\w, \v \in \mathcal K$, $\|\w - \v \|_2 \leq D$; 
\end{enumerate}
\end{ass}
Assumption~\ref{ass:3} is quite similar as Assumption~\ref{ass:1} except for an additional assumption of $\|\partial R(\w; \xi)\|_2 \leq \rho$.

%\subsection{Proximal ASSG: the Constraint variant (ProxASSG-c)}
%In this subsection, we present a proximal ASSG-c (ProxASSG-c) algorithm for solving problem~(\ref{eqn:composite}).
We present the detail steps of proximal ASSG (ProxASSG) in Algorithm~\ref{alg:assgc-prox}, which is similar to Algorithm~\ref{alg:rssg} except that Step 5  is replaced by a proximal mapping:
\begin{equation*}
\w^k_{\tau+1} = \textrm{Prox}_{\Omega_k}^{\eta_k, R} \left[\w^k_{\tau} - \eta_k \partial f(\w^k_{\tau}; \xi^k_\tau)\right],
\end{equation*}
where $\Omega_{k}$ is a ball centered at $\w_{k-1}$ with a radius $D_{k}$.  The convergence result is stated in the following theorem:
\begin{theorem}\label{thm:proxASSGc}
Suppose Assumptions \ref{ass:3} and \ref{ass:2} hold for a target $\epsilon\ll 1$. Given $\delta\in(0,1)$, let $\tilde\delta = \delta/K$,  $K = \lceil \log_2(\frac{\epsilon_0}{\epsilon})\rceil$, $D_1\geq \frac{c\epsilon_0}{\epsilon^{1-\theta}}$ and $t$ be the smallest integer such that $t \geq \max\left\{\max(16, 3072\log(1/\tilde\delta)) \frac{G^2D_1^2}{\epsilon_0^2}, \frac{8\rho D_1}{\epsilon_0} \right\}$. Then ProxASSG guarantees that, with a probability $1-\delta$, %there exists $k \in \{1, 2, \dots, K \}$ such that
\begin{equation*}
F(\w_K) - F_* \leq 2 \epsilon.
\end{equation*}
As a result, the iteration complexity of ProxASSG for achieving an $2\epsilon$-optimal solution with a high probability $1-\delta$ is $\widetilde O (\log(1/\delta)/\epsilon^{2(1-\theta)})$ provided $D_1=O(\frac{c\epsilon_0}{\epsilon^{(1-\theta)}})$.
\end{theorem}
%{\bf Remark:}  It is worth mentioning that unlike traditional high-probability analysis of SSG that usually requires the domain to be bounded, the convergence analysis of ASSG does not rely on such a condition. Furthermore, the iteration complexity of ASSG has a logarithmic dependence on $\epsilon_0$. If we known $dist(\w_0, \mathcal K_*)\leq B$, then we can set $\epsilon_0 = GB$. Hence, the iteration complexity of ASSG has only a logarithmic dependence on the distance of the initial solution to the optimal set. To achieve the same guarantee as above, Epoch-GD~\citep{hazan-20110-beyond} and AC-SA~\citep{DBLP:journals/siamjo/Lan13b} have a complexity of $O(\log(1/\delta)/\epsilon)$ and $O((\log(1/\delta))^2/\epsilon)$ respectively. Compared to their results, our ASSG method can have a better complexity if $\theta>1/2$ and we assume local error bound condition while they assume strong convexity.
\begin{algorithm}[t]
\caption{the ProxASSG algorithm for solving~(\ref{eqn:composite})} \label{alg:assgc-prox}
\begin{algorithmic}[1]
\STATE \textbf{Input}: the  number of stages $K$, the number of iterations $t$ per stage, and the initial solution $\w_0$, $\eta_1=\epsilon_0/(4G^2)$ and $D_1\geq \frac{c\epsilon_0}{\epsilon^{1-\theta}}$
\FOR{$k=1,\ldots, K$}
\STATE Let $\w^k_1 = \w_{k-1}$, $\Omega_k = \mathcal B(\w_{k-1},D_k)$
\FOR{$\tau=1,\ldots, t$}
    \STATE Update $\w^k_{\tau+1} = \textrm{Prox}_{\Omega_k}^{\eta_k,R} \left[\w^k_{\tau} - \eta_k \partial f(\w^k_{\tau}; \xi^k_\tau)\right]$
   \ENDFOR
\STATE Let $\w_k = \frac{1}{t}\sum_{\tau=1}^t\w^k_\tau$
\STATE Let $\eta_{k+1} = \eta_k/2$ and $D_{k+1} = D_k/2$. 
%\STATE Let $\Omega_k = \mathcal K\cap \mathcal B(\w_{k-1},D_k)$.
\ENDFOR
\STATE \textbf{Output}:  $\w_K$
\end{algorithmic}
\end{algorithm}
To prove Theorem~\ref{thm:proxASSGc}, we need  the following lemma for each stage of ProxASSG.
\begin{lemma}\label{thm:proxSSG}
Let $D$ be the upper bound of $\|\w_1 - \w_{1,\epsilon}^\dagger\|_2$. Apply $t$-iterations of following steps:
\begin{equation*}%\label{eqn:proxssg:1}
\w_{\tau+1} =  {\arg\min}_{\w \in  \mathcal B(\w_1,D)} \frac{1}{2} \|\w- \w_{\tau} \|_2^2 + \eta \partial f(\w_{\tau}; \xi_\tau)^\top \w + \eta R(\w). 
\end{equation*}
Given $\w_1\in\R^d$, for any $\delta\in(0,1)$, with a probability at least $1-\delta$,
\begin{equation*}
F(\wh_t)  - F(\w_{1,\epsilon}^{\dagger}) \leq  \frac{\eta G^2}{2} +\frac{ \|\w_1 -\w_{1,\epsilon}^{\dagger}\|_2^2 }{2\eta t} + \frac{4GD\sqrt{3\log(1/\delta)}}{\sqrt{t}}  +  \frac{\rho D}{t},
\end{equation*}
where $\wh_t = \sum_{\tau=1}^t\w_t/t$. 
\end{lemma}
The proof of Lemma~\ref{thm:proxSSG} is deferred to~\ref{app:proxSSG}. With the above lemma, the proof of Theorem~\ref{thm:proxASSGc}  is similar to that of Theorem~\ref{thm:RSSG}. We include the details in~\ref{app:proxASSGc}.

Before ending this section, we note that the presented ProxASSG algorithm in Algorithm~\ref{alg:assgc-prox} is based on the constrained version of ASSG. One can also develop a proximal variant based on the regularized version of ASSG. We include the details in~\ref{app:proxASSGR}. However, the convergence guarantee of proximal ASSG based on the regularized version is slightly worse than that based on the constrained version by a constant factor depending on $G$ and $\rho$.

\section{Complexity of ASSG for Ensuing the Gradient is Small}\label{sec:sta}
Recently, there has been an increasing interest in the complexity of stochastic algorithms for finding a solution for a convex optimization problem with a small gradient~\citep{DBLP:conf/nips/Allen-Zhu18,DBLP:journals/corr/abs-1902-04686}. However, these studies assume the smoothness of the objective function. The non-smoothness of the objective function make it more challenging to design stochastic algorithms and characterize their complexity of making the gradient small.

The first challenge is how to quantify the convergence in terms of gradient for a non-smooth problem. A traditional measure is using the distance from $0$ to the subgradient (a set) of the objective function at a solution $\x\in\mathcal K$, i.e., $\text{dist}(0, \partial (f(\x) + 1_{\mathcal K}(\x))$, where $1_{\mathcal K}$ is the indicator function of the domain $\mathcal K$. However, for a non-smooth function  finding an $\epsilon$-level stationary point (i.e.,  $\text{dist}(0, \partial (f(\x) + 1_{\mathcal K}(\x))\leq \epsilon$) is difficult. For example, considering the simple function $f(x) = |x|$, as long as $x\neq 0$ the traditional measure $\text{dist}(0, \partial (f(\x) + 1_{\mathcal K}(\x)) = 1$ is never 0. To address this challenge, previous studies on non-smooth optimization have used a new convergence measure based on the Moreau envelop of the objective function.  A Moreau envelope of $F(\w)$ associated with a positive constant $\lambda>0$ is defined as: 
%The Moreau envelope of $F(\w)$ is
\begin{equation}\label{moreau:env}
F_{\lambda}(\w) = \min_{\v\in\mathcal K}\left\{ F(\v) + \frac{\lambda}{2}\|\v -\w\|^2\right\},
\end{equation}
and the  associated proximal mapping is defined as 
\begin{equation}\label{moreau:prox}
\widetilde \w = \text{Prox}_{F/\lambda}(\w):=\arg \min_{\v\in\mathcal K}\left\{ F(\v) + \frac{\lambda}{2}\|\v -\w\|^2\right\}.
\end{equation}
It is easy to show that  $F_\lambda(\cdot)$ is a smooth function whose gradient is $\lambda$-Lipchitz continuous~\citep{Bauschke:2011:CAM:2028633} and $\widetilde \w$ satisfies~\citep{davis2018stochastic}:
\begin{equation*}
\begin{aligned}
F(\widetilde\w) &\leq  F(\w),\\
\nabla F_{\lambda}(\w) & = \lambda (\w - \widetilde\w ) \\
\textrm{dist}(0;\partial F(\widetilde\w))& \leq   \|\nabla F_{\lambda}(\w)\|.
\end{aligned}
\end{equation*}
It means that if  $\|\nabla F_{\lambda}(\w)\|\leq \epsilon$ then $\w$ is close to  some point $\widetilde\w$ that is an $\epsilon$-stationary solution for the problem (\ref{eqn:psg}). 
This gives a new convergence measure  in terms of gradient for a non-smooth function. We call a solution $\w$ an $\epsilon$-nearly stationary point if the following inequality holds for some constant $\lambda>0$:
\begin{align}
\|\nabla F_{\lambda}(\w)\| \leq \epsilon.
\end{align}
It is also notable that when $F$ is $L$-smooth~\footnote{whose gradient is $L$-Lipchitz continuous.} and the constraint domain is the whole space $\mathcal K = \R^d$, then an $\epsilon$-nearly stationary point $\w$ also implies that it is $O(\epsilon)$-stationary in the traditional sense, i.e., $\|\nabla F(\w)\|\leq O(\epsilon)$. This can be easily seen from $\|\nabla F(\w)\|\leq \|\nabla F(\widetilde \w)\| + \|\nabla F(\w) - \nabla F(\widetilde \w)\|\leq \|\nabla F_{\lambda}(\w)\| + L\|\w - \widetilde \w\| = (1 + \frac{L}{\lambda})\|\nabla F_{\lambda}(\w)\|\leq (1+L/\lambda)\epsilon$. 

Next, we  give a simple lemma that will be useful for our analysis later.  
\begin{lemma}\label{lemm:moreau}
%Denote by $\widetilde{\mathcal K}_*$ the optimal set of the Moreau envelope function, i.e.,  $ \widetilde{\mathcal K}_*=  \arg\min_{\w\in\mathcal K} F_{\lambda}(\w) $. Then we have
%\begin{itemize}
%\item  $\mathcal K_* \subseteq \widetilde{\mathcal K}_*$.  %for any $\widetilde \w_* \in \widetilde{\mathcal K}_*$ we have
%\begin{equation*}
%F_{\lambda}(\widetilde\w_*) = \arg\min_{\w\in\Omega}F(\w) 
%\end{equation*}
%\item and 
For any $\w\in\mathcal K$, it holds
\begin{equation}\label{grad:con:1}
\|\nabla F_{\lambda}(\w)\|^2 \leq  2 \lambda (F(\w) - F(\w_*)),
\end{equation}
where $\w_*\in\mathcal K_*$. 
%\end{itemize}
\end{lemma}
\begin{proof}
%To prove the first point,  let us consider any $\w_*\in\mathcal K_*$. According to the definition of $\mathcal K_*$, we  have $\text{dist}(0, \partial (f(\w_*) + 1_{\mathcal K}(\w_*)) =0$. As a result, $\text{Prox}_{F/\lambda}(\w_*)=\w_*$.  It then implies that $\nabla F_{\lambda}(\w_*) =\lambda(\w_* -  \text{Prox}_{F/\lambda}(\w_*))=0$, i.e., $\w_*\in \widetilde{\mathcal K}_*$. To prove the second point, 
We first show that $ \arg\min_{\w\in\mathcal K}F_{\lambda}(\w) = \arg\min_{\w\in\mathcal K}F(\w)$. Let us consider any $\widetilde\w_*\in\widetilde{\mathcal K}_*$.  Then for any  $\v, \w \in\mathcal K$, $F_{\lambda}(\widetilde\w_*) \leq F_{\lambda}(\w) \leq F(\v) + \frac{\lambda}{2}\|\v -\w\|^2 $.
Let $\v = \w = \w_*$, we have 
\begin{equation}\label{moreau:inq1}
F_{\lambda}(\widetilde \w_*) \leq F_{\lambda}(\w_*) \leq F(\w_*).
\end{equation}
On the other hand, if we let $\widehat \v := \arg\min_{\v\in\mathcal K}\{ F(\v) + \frac{\lambda}{2}\|\v - \widetilde \w_*\|^2\} $, then
\begin{equation}\label{moreau:inq2}
F(\w_*) \leq F(\widehat \v) \leq  F(\widehat \v)  + \frac{\lambda}{2}\|\widehat\v -\widetilde\w_*\|^2 =  F_{\lambda}(\widetilde \w_*).
\end{equation}
Therefore, by (\ref{moreau:inq1}) and (\ref{moreau:inq2}) we have $F(\w_*) =  F_{\lambda}(\widetilde \w_*)$. 
Next, let $\widetilde\w : =  \arg\min_{\v\in\mathcal K} F(\v) + \frac{\lambda}{2}\|\v -\w\|^2 $. %hen
%\begin{equation*}
%\textrm{dist}(0, \partial F( \widetilde\w)) \leq \|\nabla F_{\lambda}(\w)\|,
%\end{equation*}
%where $\nabla F_{\lambda}(\w)$ is given by
%\begin{equation*}
%\nabla F_{\lambda}(\w) = \lambda(\w - \widetilde\w).
%\end{equation*}
By the smoothness of $F_{\lambda}(\w)$, we have 
\begin{align*}
F_{\lambda}(\widetilde\w) - F_{\lambda}(\w) \leq  & \nabla F_{\lambda}(\w)^\top (\widetilde\w - \w) + \frac{\lambda}{2}\|\widetilde\w - \w\|^2 \\
=&  - \frac{1}{\lambda} \| \nabla F_{\lambda}(\w)\|^2 + \frac{1}{2 \lambda}\|\nabla F_{\lambda}(\w)\|^2 
\end{align*}
Rewriting above inequality and combining with $F_{\lambda}(\widetilde\w_*) \leq F_{\lambda}(\widetilde\w)$ we get
\begin{equation*}%\label{grad:con:1}
\frac{1}{2 \lambda}\|\nabla F_{\lambda}(\w)\|^2 \leq  F_{\lambda}(\w) - F_{\lambda}(\widetilde\w_*).
\end{equation*}
By the definition of $F_{\lambda}(\w)$, for any $\w\in\mathcal K$, we have $F_{\lambda}(\w) \leq F(\w)$. %., which implies that $F_{\lambda}(\w_K) \leq F(\w_K)$.
Therefore, we have
\begin{equation*}%\label{grad:con:1}
\frac{1}{2 \lambda}\|\nabla F_{\lambda}(\w)\|^2 \leq  F(\w) - F(\w_*).
\end{equation*}
\end{proof}

%\subsection{Complexity of ASSG for Making Gradient Small under LGC}
Next, we will characterize the complexity of ASSG for finding an $\epsilon$-nearly stationary point for the problem (\ref{eqn:psg}) under the LGC by leveraging the result in Lemma~\ref{lemm:moreau}.
%We present nearly stationary point convergence results for ASSG-c, ASSG-r and ASSG-s as follows. Similar results can be obtained for proximal versions and practical versions of ASSG. 
\begin {theorem}\label{thm:ASSG:wo:grad}
%We have the following two results:\\
Under the same setting in Theorem~\ref{thm:RSSG} or Theorem~\ref{thm:RPSG}, then with a high probability $1-\delta$, ASSG-c or ASSG-r guarantees that $\|\nabla F_{1/4}(\w_K)\| \leq  \epsilon$ with the iteration complexity of $ \widetilde O (1/\epsilon^{4(1-\theta)})$.\\
%(2) Under the same setting in Theorem~\ref{thm:ASSG:wo}, then ASSG-s guarantees that $\E[\|\nabla F_{1/2}(\w_K)\|] \leq  \epsilon$ with the iteration complexity of $ \widetilde O (1/\epsilon^{4(1-\theta)})$.
\end {theorem}
{\bf Remark.} \cite{DBLP:conf/nips/Allen-Zhu18} considered  stochastic gradient descent (SGD) with recursive regularization to solve smooth and convex problems and provided a complexity $\widetilde O(1/\epsilon^2)$ for achieving an $\epsilon$-stationary point. In contrast, we focus on non-smooth problems in this paper. When $\theta>\frac{1}{2}$, our methods achieve better complexities.

\begin{proof}
%Let $\widetilde\w := \v(\w)= \arg\min_{\v\in\mathcal K} F(\v) + \frac{\lambda}{2}\|\v -\w\|^2 $, then
%\begin{equation*}
%\textrm{dist}(0, \partial F( \widetilde\w)) \leq \|\nabla F_{\lambda}(\w)\|,
%\end{equation*}
%where $\nabla F_{\lambda}(\w)$ is given by
%\begin{equation*}
%\nabla F_{\lambda}(\w) = \lambda(\w - \widetilde\w).
%\end{equation*}
%By the smoothness of $F_{\lambda}(\w)$, we have 
%\begin{align*}
%F_{\lambda}(\widetilde\w) - F_{\lambda}(\w) \leq  & \nabla F_{\lambda}(\w)^\top (\widetilde\w - \w) + \frac{\lambda}{2}\|\widetilde\w - \w\|^2 \\
%=&  - \frac{1}{\lambda} \| \nabla F_{\lambda}(\w)\|^2 + \frac{1}{2 \lambda}\|\nabla F_{\lambda}(\w)\|^2 
%\end{align*}
%Rewriting above inequality and combining with $F_{\lambda}(\widetilde\w_*) \leq F_{\lambda}(\widetilde\w)$ we get
%\begin{equation}\label{grad:con:1}
%\frac{1}{2 \lambda}\|\nabla F_{\lambda}(\w)\|^2 \leq  F_{\lambda}(\w) - F_{\lambda}(\widetilde\w_*).
%\end{equation}
%By the definition of $F_{\lambda}(\w)$, for any $\w\in\mathcal K$, we have $F_{\lambda}(\w) \leq F(\w)$, which implies that $F_{\lambda}(\w_K) \leq F(\w_K)$. 
%Then by Theorem~\ref{thm:ASSG:wo} and Lemma~\ref{lemm:moreau}, we get
%\begin{equation}\label{grad:con:2}
%\E[F_{\lambda}(\w_K)  - F_{\lambda}(\widetilde\w_*) ] \leq \epsilon.
%\end{equation}
%We complete the proof by combining (\ref{grad:con:1}) and  (\ref{grad:con:2}) and setting $\lambda = \frac{1}{2}$.
Let $\lambda = \frac{1}{4}$ and $\w = \w_K$ in (\ref{grad:con:1}) of Lemma~\ref{lemm:moreau}, we get
\begin{equation}\label{thm:assg:grad:ineq1}
\|\nabla F_{1/4}(\w_K)\|^2 \leq  \frac{1}{2}(F(\w_K) - F(\w_*)).
\end{equation}
Let $\epsilon=\epsilon^2$ in Theorem~\ref{thm:RSSG} or Theorem~\ref{thm:RPSG}, we know that with high probability $1-\delta$,
\begin{equation}\label{thm:assg:grad:ineq2}
F(\w_K) - F(\w_*) \leq 2\epsilon^2.
\end{equation}
By (\ref{thm:assg:grad:ineq1}) and (\ref{thm:assg:grad:ineq2}) we have
\begin{equation}\label{thm:assg:grad:ineq3}
\|\nabla F_{1/4}(\w_K)\| \leq  \epsilon.
\end{equation}
%As a result, the iteration complexity of ASSG-c or ASSG-r for achieving an $\epsilon$-nearly stationary point with a high probability $1-\delta$ is $ O (c^2G^2\lceil \log_2(\frac{\epsilon_0}{\epsilon^2})\rceil\log(1/\delta)/\epsilon^{4(1-\theta)})$.
%
%Next, let $\lambda = \frac{1}{2}$ and $\w = \w_K$ in (\ref{grad:con:1}) of Lemma~\ref{lemm:moreau}, we get
%\begin{equation}\label{thm:assg:grad:ineq4}
%\|\nabla F_{1/2}(\w_K)\|^2 \leq  F(\w_K) - F(\w_*).
%\end{equation}
%Let $\epsilon=\epsilon^2$ in Theorem~\ref{thm:ASSG:wo}, we know
%\begin{equation}\label{thm:assg:grad:ineq5}
%\E[F(\w_K) - F(\w_*)] \leq \epsilon^2.
%\end{equation}
%By (\ref{thm:assg:grad:ineq4}) and (\ref{thm:assg:grad:ineq5}) and $(\E[\|\nabla F_{1/2}(\w_K)\|])^2\leq\E[\|\nabla F_{1/2}(\w_K)\|^2]$we have
%\begin{equation}\label{thm:assg:grad:ineq6}
%\E[\|\nabla F_{1/2}(\w_K)\|] \leq  \epsilon.
%\end{equation}
%As a result, the iteration complexity of ASSG-sim for achieving an $\epsilon$-nearly stationary point in expectation is $O (c^2G^2\lceil \log_2(\frac{\epsilon_0}{\epsilon})\rceil/\epsilon^{4(1-\theta)})$.
\end{proof}

\section{Applications in Risk Minimization}\label{sec:app}
In this section, we present some applications of the proposed ASSG to risk minimization in machine learning. Let $(\x_i, y_i), i = 1, \dots, n$ denote a set of pairs of feature vectors and labels that follow a distribution $\mathcal P$, where $\x_i\in\X\subset\R^d$ and $y_i\in \mathcal Y$. Many machine learning problems end up solving the regularized empirical risk minimization problem:
\begin{equation}\label{eqn:ERM}
\min_{\w \in \R^d} F(\w) = \frac{1}{n}\sum_{i=1}^{n}\ell( \w^\top\x_i, y_i) + \lambda R(\w),
\end{equation}
where $R(\w)$ is a regularizer, $\lambda$ is the regularization parameter and $\ell(z,y)$ is a loss function. Below we will present several examples in machine learning that enjoy faster convergence by the proposed ASSG than by SSG.

\subsection{Piecewise Linear Minimization}
%\subsection{Empirical Risk Minimization}
%In this subsection, we consider some examples of empirical risk minimization. A similar argument as above can be used to show that ASSG enjoys $\widetilde O\left(\frac{1}{\epsilon}\right)$ for the empirical square loss minimization problem with and without $\ell_1$ regularizer~\citep{DBLP:journals/corr/nesterov16linearnon}. Next, we consider empirical  polyhedral loss minimization and Huber loss minimization with an $\ell_1$ regularization. 
First, we consider some examples of non-smooth and non-strongly convex problems such that ASSG can achieve linear convergence. In particular, we consider the problem~(\ref{eqn:ERM}) with a piecewise linear loss and $\ell_1$, $\ell_\infty$ or $\ell_{1,\infty}$ regularizers. 

%{\bf $\ell_1$ regularized polyhedral  loss minimization.}
Piecewise linear loss includes hinge loss~\citep{Vapnik1998}, generalized hinge loss~\citep{bartlett2008classification}, absolute loss~\citep{The-Elements-of-Statistical-Learning-2009}, and $\epsilon$-insensitive loss~\citep{rosasco2004loss}. For particular forms of these loss functions, please refer to~\citep{DBLP:journals/ml/YangMJZ14Non}.  
%We can consider the polyhedral loss minimization with an $\ell_1$ regularization: 
%\begin{equation*}
%\min_{\w\in\R^d}F(\w) = \frac{1}{n}\sum_{i=1}^n\ell(\w^{\top}\x_i, y_i) + \tau \|\w\|_1
%\end{equation*}
%where $\ell(z,y)$ denotes a polyhedral loss. 
The epigraph of $F(\w)$ defined by sum of a piecewise linear loss function and an $\ell_1$, $\ell_{\infty}$ or $\ell_{1,\infty}$ norm regularizer is a polyhedron. According to the polyhedral error bound condition~\citep{DBLP:journals/corr/arXiv:1512.03107}, for any $\epsilon>0$  there exists a constant $0<c<\infty$ such that 
\begin{equation*}
dist(\w,\mathcal K_*) \leq c (F(\w) - F_*)
\end{equation*}
for any $\w\in\mathcal S_\epsilon$, meaning that the proposed ASSG has an $O(\log(\epsilon_0/\epsilon))$ iteration complexity for solving such family of problems. Formally, we state the result in the following corollary. 
\begin{corollary}
Assume the loss function $\ell(z,y)$ is piecewise linear, then the problem in~(\ref{eqn:ERM}) with $\ell_1$, $\ell_{\infty}$ or $\ell_{1,\infty}$ norm regularizer satisfy the LGC in~(\ref{eqn:lgc2}) with $\theta=1$. Hence ASSG can have an iteration complexity of $O(\log(1/\delta)\log(\epsilon_0/\epsilon))$ with a high probability $1-\delta$.  
\end{corollary}
%Next, we consider expected logistic regression over a bounded domain. 
%\begin{equation*}
%\min_{\w\in\R^d, \|\w\|_2\leq R}\E[\log(1+\exp(-y\w^{\top}\x))]
%\end{equation*}
%Let $f(\w; \x, y) = \log(1+\exp(-y\w^{\top}\x))$
%It is not difficult to show that 
%\begin{equation*}
%f(\w; \x, y)\geq f(\v; \x, y) + \nabla f(\v; \x, y)^{\top}(\w - \v) + \frac{\beta}{2} (\w-\v)^{\top}\x\x^{\top}(\w - \v), \forall \w, \v\in\mathcal B_R
%\end{equation*}
%where $\beta = 1/(2(1+\exp(R))$. 
%Taking expectation over $(\x,y)$ on both sides, we have
%\begin{equation*}
%F(\w)\geq F(\v)+ \nabla F(\v)^{\top}(\w - \v) + \frac{\beta}{2} (\w-\v)^{\top}\Sigma (\w - \v) 
%\end{equation*}
%which implies that  $F(\w)$ is  a strongly convex when $\Sigma\in\R^{d\times d}$ has non zero eigen-values. 

\subsection{Piecewise Convex Quadratic  Minimization}
%\subsection{Expected Risk Minimization}
In this subsection, we consider some examples of piecewise quadratic  minimization problems in machine learning  and show that ASSG enjoys an iteration complexity of $\widetilde O\left(\frac{1}{\epsilon}\right)$. We first give an definition of piecewise convex quadratic functions, which is from~\citep{DBLP:journals/mp/Li13}. A function $g(\w)$ is a real polynomial if there exists $k\in\mathbb N^+$ such that $g(\w) = \sum_{0\leq |\alpha^j|\leq k}\lambda_j \prod_{i=1}^dw_i^{\alpha^j_i}$, where $\lambda_j\in\R$ and $\alpha^j_i\in\mathbb N^+\cup\{0\}$, $\alpha^j=(\alpha^j_1,\ldots, \alpha^j_d)$, and $|\alpha^j| =\sum_{i=1}^d\alpha^j_i$. The constant $k$ is called the degree of $g$. A continuous function $F(\w)$ is said to be a piecewise convex polynomial if there exist finitely many polyhedra $P_1,\ldots, P_m$ with $\cup_{j=1}^mP_j = \R^d$ such that the restriction of $F$ on each $P_j$ is a convex polynomial. Let $F_j$ be the restriction of $F$ on $P_j$. The degree of a piecewise convex polynomial function $F$  is the maximum of the degree of each $F_j$. If the degree is $2$, the function is referred to as a piecewise convex quadratic function. Note that a piecewise convex quadratic function is not necessarily a smooth function nor a convex function~\citep{DBLP:journals/mp/Li13}.

For examples of piecewise convex quadratic problems in machine learning, one can consider the problem~(\ref{eqn:ERM}) with a huber loss, squared hinge loss or square loss, and $\ell_1$, $\ell_{\infty}$, $\ell_{1,\infty}$, or huber norm regularizer~\citep{DBLP:conf/pkdd/ZadorozhnyiBMSK16}. The huber function is defined as
\begin{equation*}%\label{prob:huber}
    \ell_\delta(z) = \left\{\begin{array}{ll}\frac{1}{2}z^2   \textrm{ if } |z|\leq \delta,\\ \delta(|z| - \frac{1}{2}\delta)   \textrm{otherwise},\end{array}\right.
\end{equation*}
which is a piecewise convex quadratic function. The huber loss function $\ell(z,y)=\ell_\delta(z-y)$ has been used for robust regression. A huber regularizer is defined as $R(\w)=\sum_{i=1}^d\ell_\delta(w_i)$. 

%Thus, the empirical Huber loss minimization with an $\ell_1$ norm regularization is formulated as:
%\begin{equation*}
%    \min_{\w\in\R^d}F(\w) = \frac{1}{n}\sum_{i=1}^n\ell_\delta(\w^{\top}\x_i - y_i)  +\tau \|\w\|_1
%\end{equation*}
%It is obvious that $F(\w)$ is also a piecewise convex polynomial with degree $2$. 
It has been shown that~\citep{DBLP:journals/mp/Li13}, if $F(\w)$ is convex and  piecewise convex quadratic, then it satisfies the LGC~(\ref{eqn:lgc2}) with $\theta=1/2$. %for any $\epsilon>0$ there exists a constant $0<c<\infty$ such that
%\begin{equation}\label{eqn:pq}
%    dist(\w,\mathcal K_*)\leq c  (F(\w) - F_*)^{1/2}, \quad \forall \w\in\mathcal S_\epsilon.
%\end{equation}
%As a result, ASSG-c has an $O(\log_2(1/\epsilon))$ iteration complexity for solving $\ell_1$-regularized Huber loss minmization problem. %Please notice %that the non-optimal $\log_2(1/\epsilon)$ factor in the iteration complexity can be shaved off. 
%We put the details in the supplement.  
The corollary below summarizes the iteration complexity of ASSG for solving these problems. %leveraging the above LGC. 
\begin{corollary}
Assume the loss function $\ell(z,y)$ is a convex and  piecewise convex quadratic, then the problem in~(\ref{eqn:ERM}) with $\ell_1$, $\ell_{\infty}$,  $\ell_{1,\infty}$ or huber norm regularizer satisfy the LGC in~(\ref{eqn:lgc2}) with $\theta=1/2$. Hence ASSG can have an iteration complexity of $\widetilde O(\frac{\log(1/\delta)}{\epsilon})$ with a high probability $1-\delta$.  
\end{corollary}
%\iffalse
%{\color{blue}
{\bf Remark:} The Lipschitz continuity assumption for some loss functions (e.g., squared hinge loss and square loss) can be easily satisfied by adding a boundness constraint on the solution.  
%After the preliminary version of this manuscript was published on arXiv, we 
We note that a recent work~\citep{Ming2017adaAGC} also studied the piecewise convex quadratic minimization  problems under the error bound condition. They explore the smoothness of the loss functions and develop deterministic accelerated gradient methods with a linear convergence. In contrast, the proposed ASSG is a stochastic algorithm and does not rely on the smoothness assumption. One might also notice that several recent works~\citep{DBLP:journals/corr/GongY14,DBLP:conf/pkdd/KarimiNS16} have showed the linear convergence of SVRG by exploring the smoothness of the loss function and a similar condition as in~(\ref{eqn:lgc2}) with $\theta=1/2$. However, their required condition is a global growth condition that is required to hold for any $\w\in\R^d$. 

Indeed, a convex and  piecewise convex quadratic function enjoy a global growth condition~\citep{DBLP:journals/mp/Li13}:
\begin{equation*}%\label{eqn:cpq}
    dist(\w,\mathcal K_*)\leq c  [F(\w) - F_* + (F(\w) - F_*)^{1/2}], \quad \forall \w\in\R^d.
\end{equation*}
It remains an open problem that how to leverage such a global growth condition to develop a linear convergence for SVRG and other similar algorithms for solving finite-sum smooth problems, which is beyond the scope of this work. Nevertheless, using the above global growth condition we can reduce the iteration complexity by a $\log(\epsilon_0/\epsilon)$ factor for ASSG. We include the details in~\ref{app:PCQM}.

\subsection{Structured composite non-smooth problems}
Next, we present a corollary of our main result regarding the following structured problem:
\begin{equation}\label{eqn:ss}
\min_{\w\in\R^d}F(\w) \triangleq  h(X\w) + R(\w).
\end{equation}
where $X\in\R^{n\times d}$, $h(\u)$ is a strongly convex function (not necessarily a smooth function) on any compact set and $R(\w)$ is $\ell_1$, $\ell_\infty$ or $\ell_{1,\infty}$ norm regularizer. The corollary below formally states the LGC of the above problem and the iteration complexity of ASSG. 
\begin{corollary}\label{cor:ss}
Assume $h(\u)$ is a strongly convex function on any compact set and $P(\w)$ is polyhedral, then the problem in~(\ref{eqn:ss})  satisfies the LGC in~(\ref{eqn:lgc2}) with $\theta=1/2$. Hence ASSG can have an iteration complexity of $\widetilde O(\frac{\log(1/\delta)}{\epsilon})$ with a high probability $1-\delta$.  
\end{corollary}
The proof of the first part of Corollary~\ref{cor:ss} can be found in~\citep{DBLP:journals/corr/arXiv:1512.03107}. One example of $h(\u)$ is $p$-norm error ($p\in(0,1)$), where  $h(\u)=\frac{1}{n} \sum_{i=1}^{n} |u_i-y_i|^p$. The local strong convexity of the $p$-norm error ($p\in(1,2)$) is shown in~\citep{Goebel_localstrong}. 

Finally, we give an example that satisfies the LGC with intermediate values $\theta\in(0, 1/2)$.  We can consider an $\ell_1$ constrained $\ell_p$ norm regression~\citep{doi:10.1080/03610928308828618}: 
 \begin{equation*}
\min_{\|\w\|_1\leq s}F(\w)\triangleq \frac{1}{n}\sum_{i=1}^n(\x_i^{\top}\w - y_i)^p,\quad p\in 2\mathbb N^+.
\end{equation*}
\cite{Ming2017adaAGC} have shown that the problem above satisfies the LGC in~(\ref{eqn:lgc2}) with $\theta=\frac{1}{p}$.

\begin{table*}[t]
		\caption{Statistics of real datasets}
		\centering
		\label{table1}
		{\begin{tabular}{l|l|l|ll}
			\toprule
			Dataset &  \#Training $(n)$ &  \#Features $(d)$  & Problem Type \\
			\midrule
covtype.binary   & 581,012  & 54   &Classification \\
real-sim   &72,309 &  20,958  &Classification \\
url  & 2,396,130	&  3,231,961  & Classification \\
avazu  & 40,428,967  & 1,000,000   &Classification \\
gisette  & 6,000	&  5,000  & Classification \\
kdd 2010 raw   & 19,264,097  & 1,163,024   &Classification \\
news20.binary &19,996 &  1,355,191	  & Classification \\
rcv1.binary   & 20,242  & 47,236   &Classification \\
webspam & 350,000 & 16,609,143   &Classification \\
 \midrule
 million songs &  463,715 &  90 & Regression \\
E2006-tfidf   &16,087 &  150,360 &  Regression\\
E2006-log1p	& 16,087	&  4,272,227&  Regression\\
     			\bottomrule
		\end{tabular}}
	\end{table*}

\begin{figure}[th]
\centerline{\includegraphics[scale=.25]{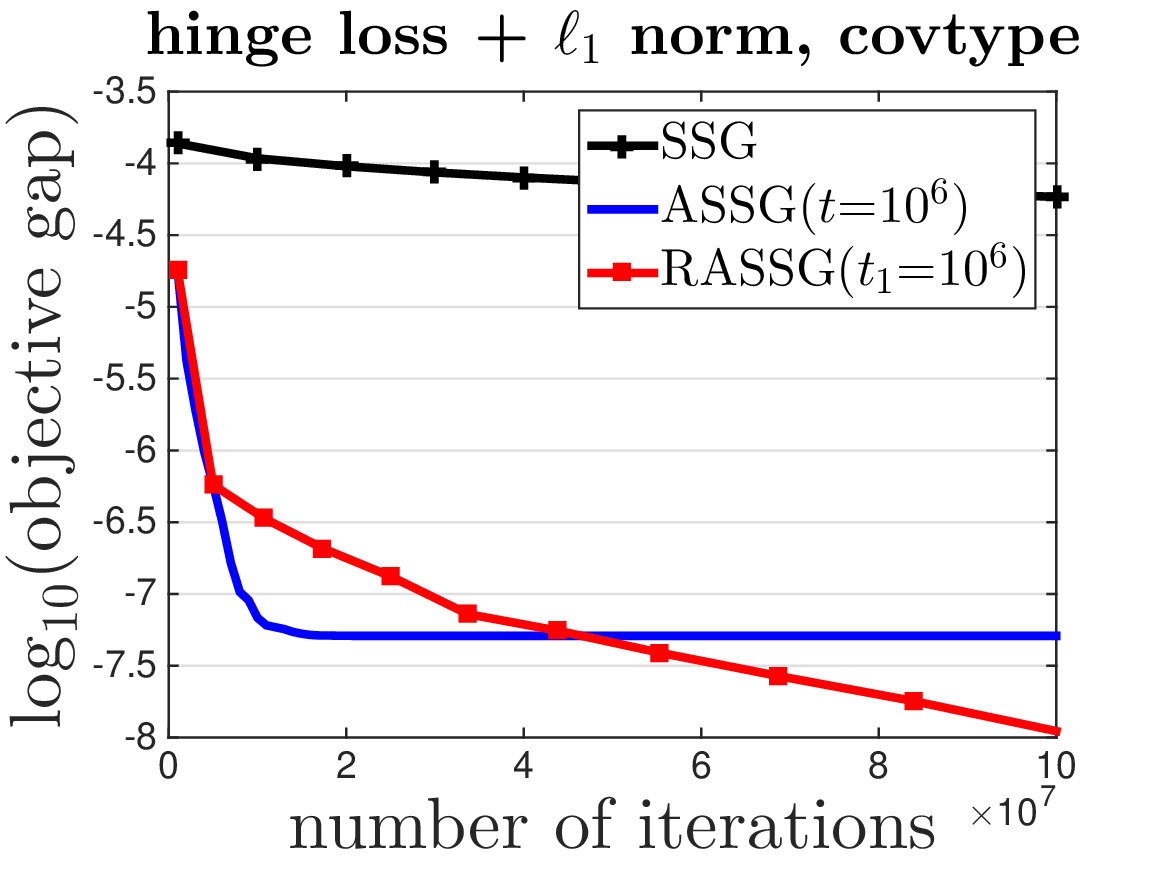}
		\includegraphics[scale=.25]{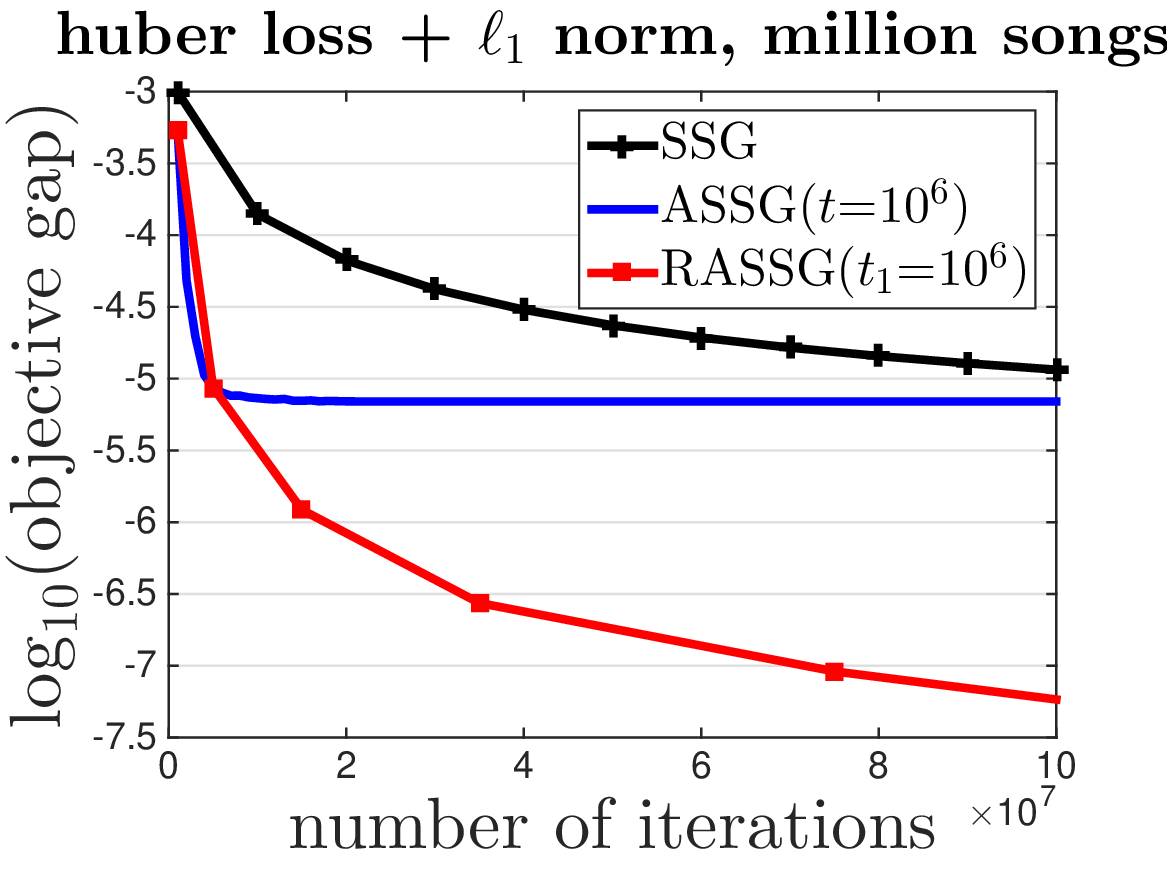}
		\includegraphics[scale=.25]{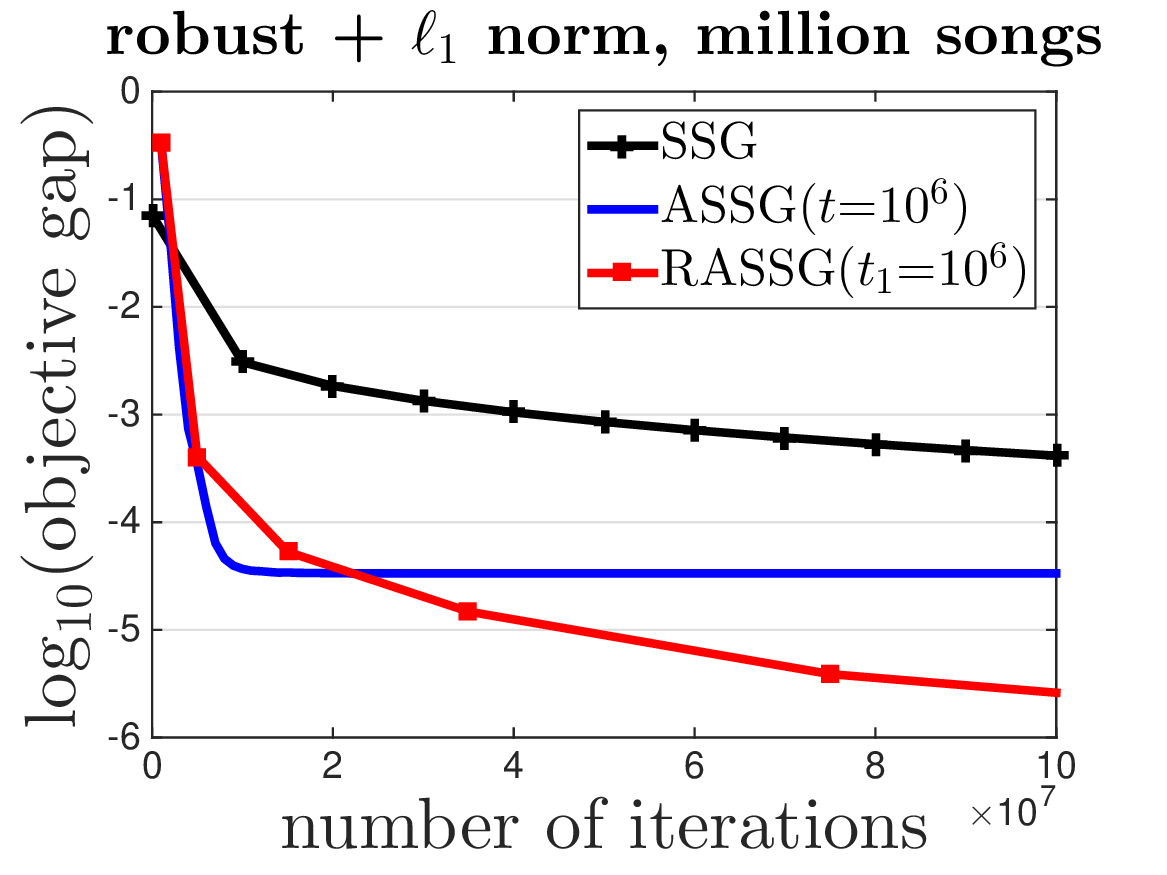}
}
\centerline{\includegraphics[scale=.25]{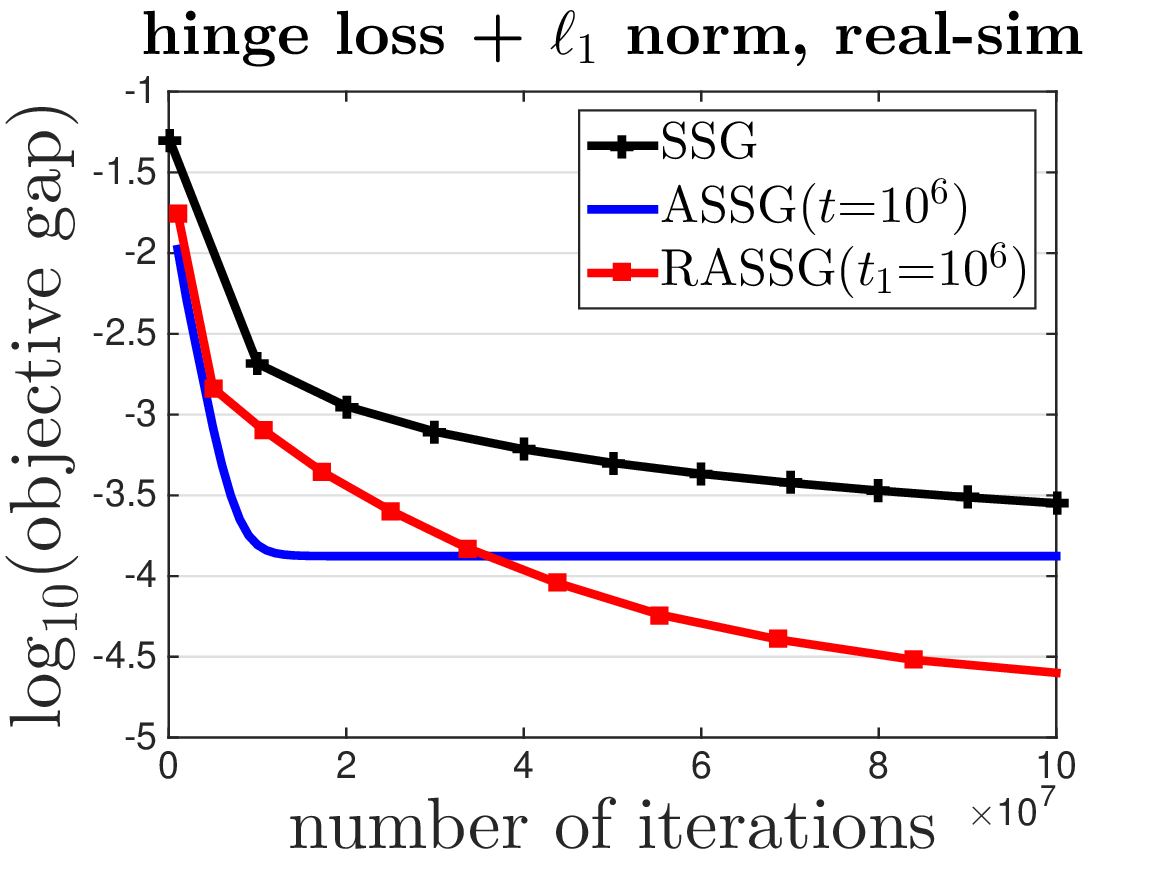}
		\includegraphics[scale=.25]{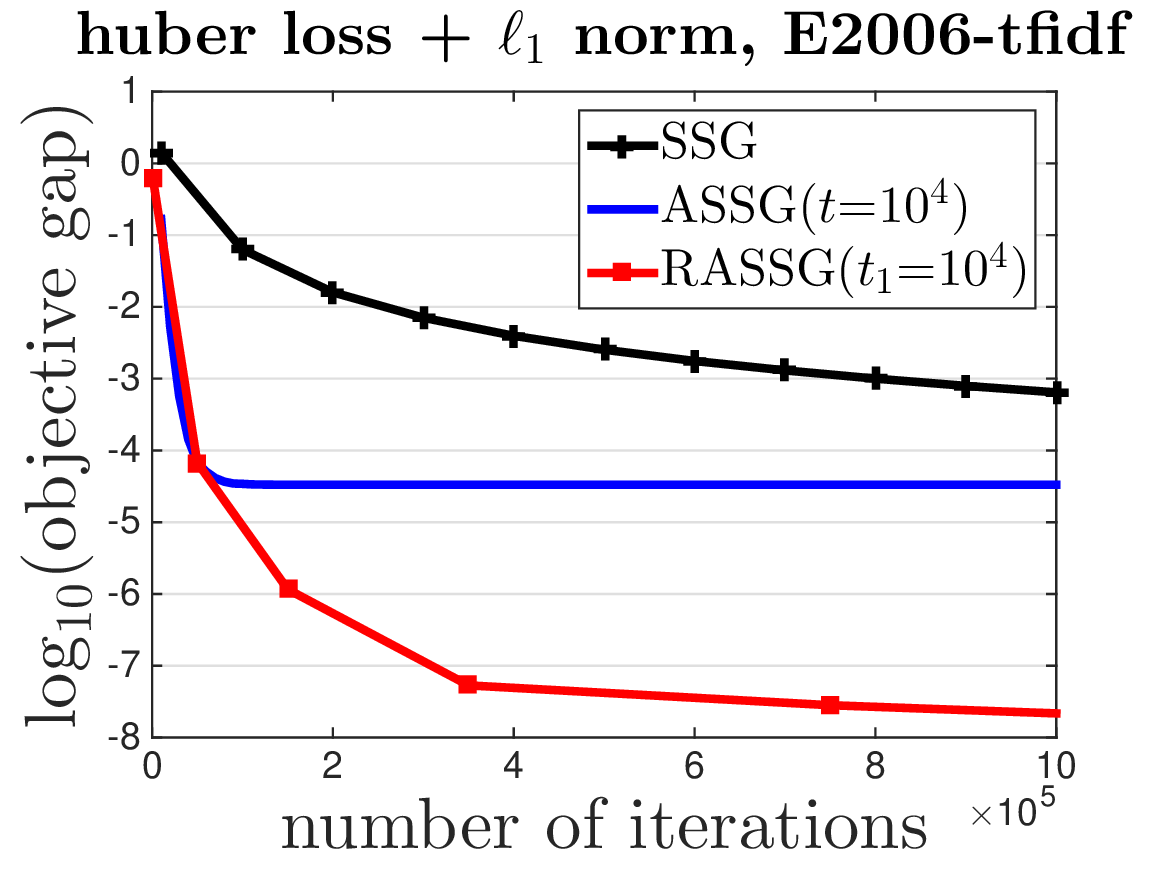}
		\includegraphics[scale=.25]{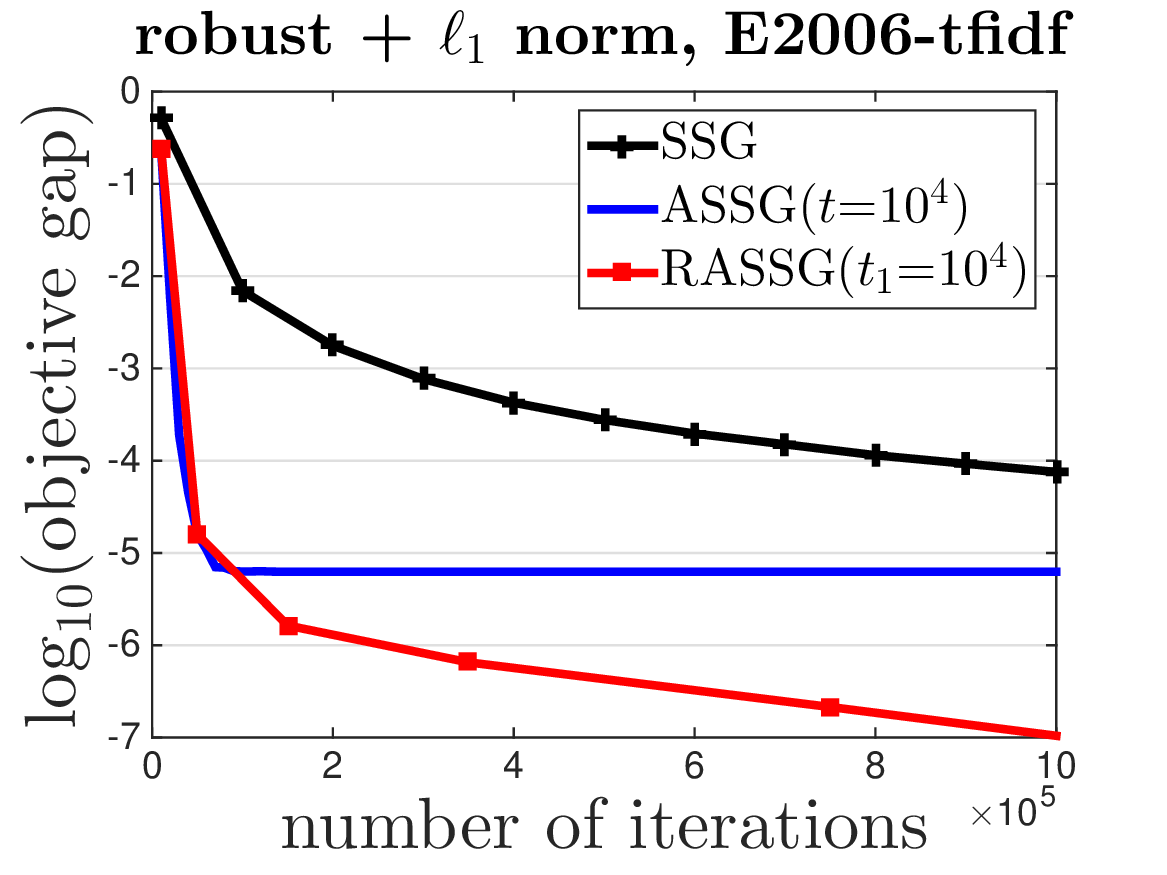}
}
\vspace*{8pt}
\caption{Comparison of different algorithms for solving different problems on different datasets ($\lambda=10^{-4}$).}
\label{fig01}
\end{figure}

\begin{figure}[th]
\centerline{\includegraphics[scale=.25]{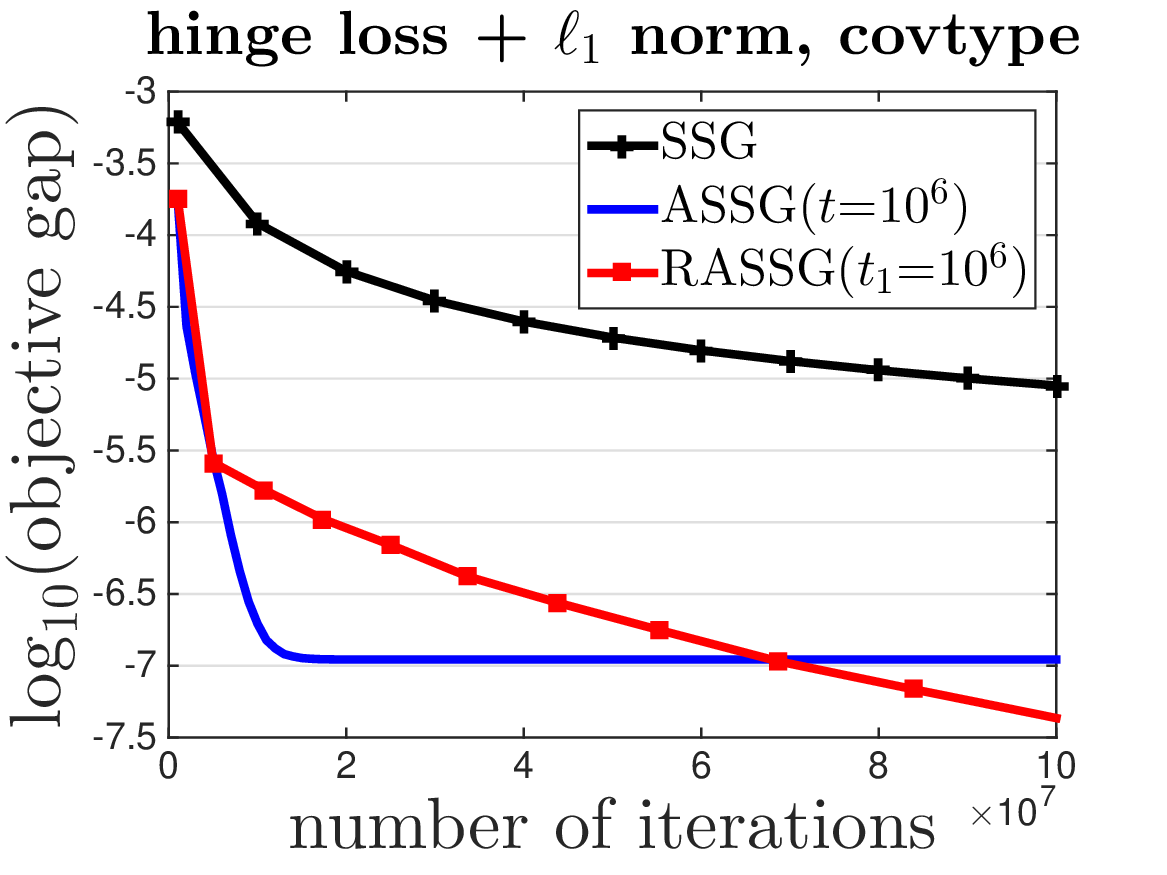}
		\includegraphics[scale=.25]{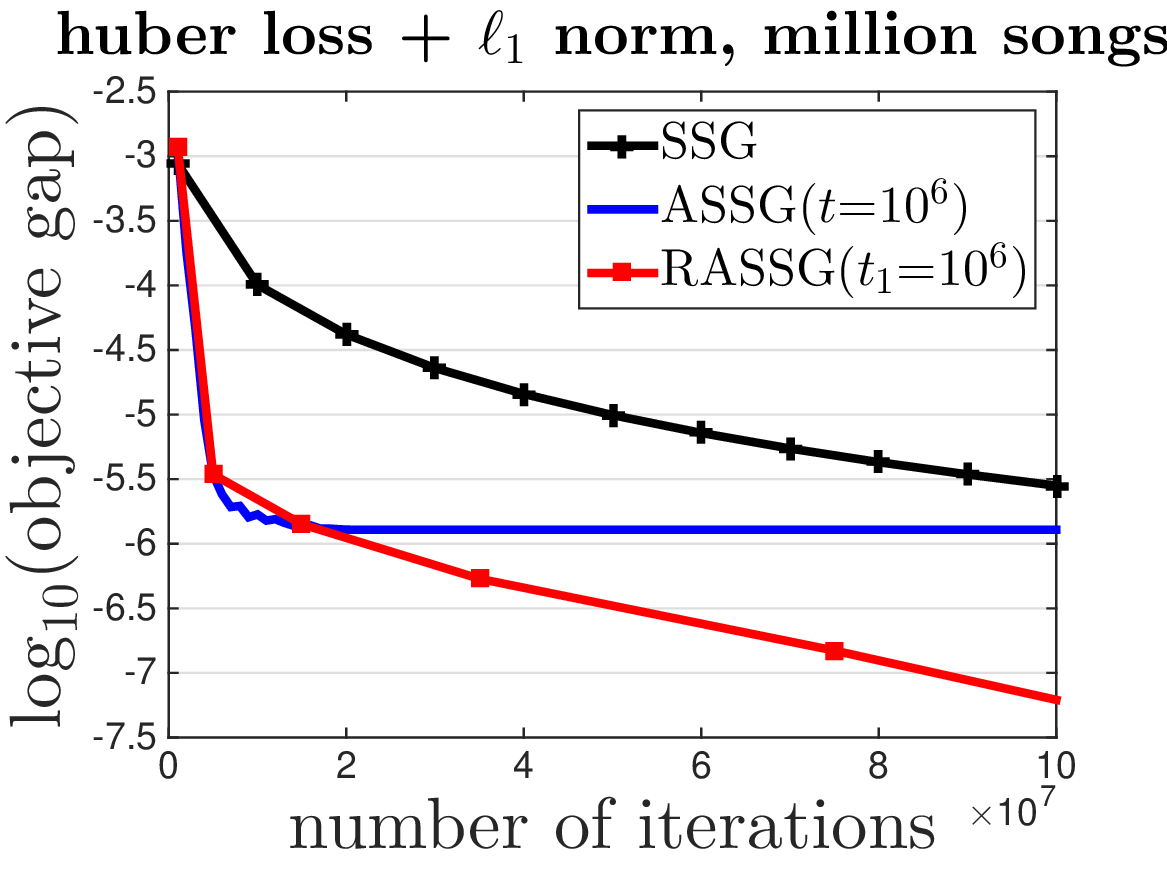}
		\includegraphics[scale=.25]{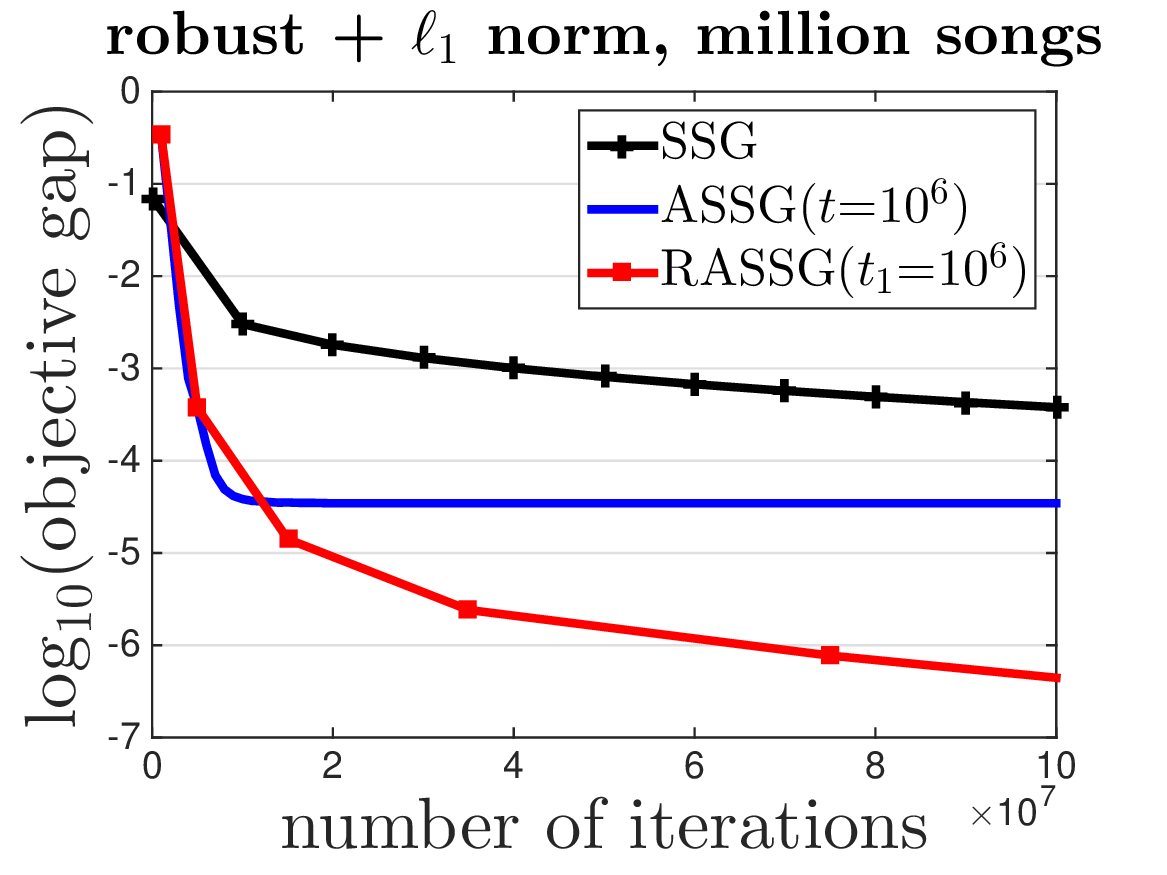}
}
\centerline{\includegraphics[scale=.25]{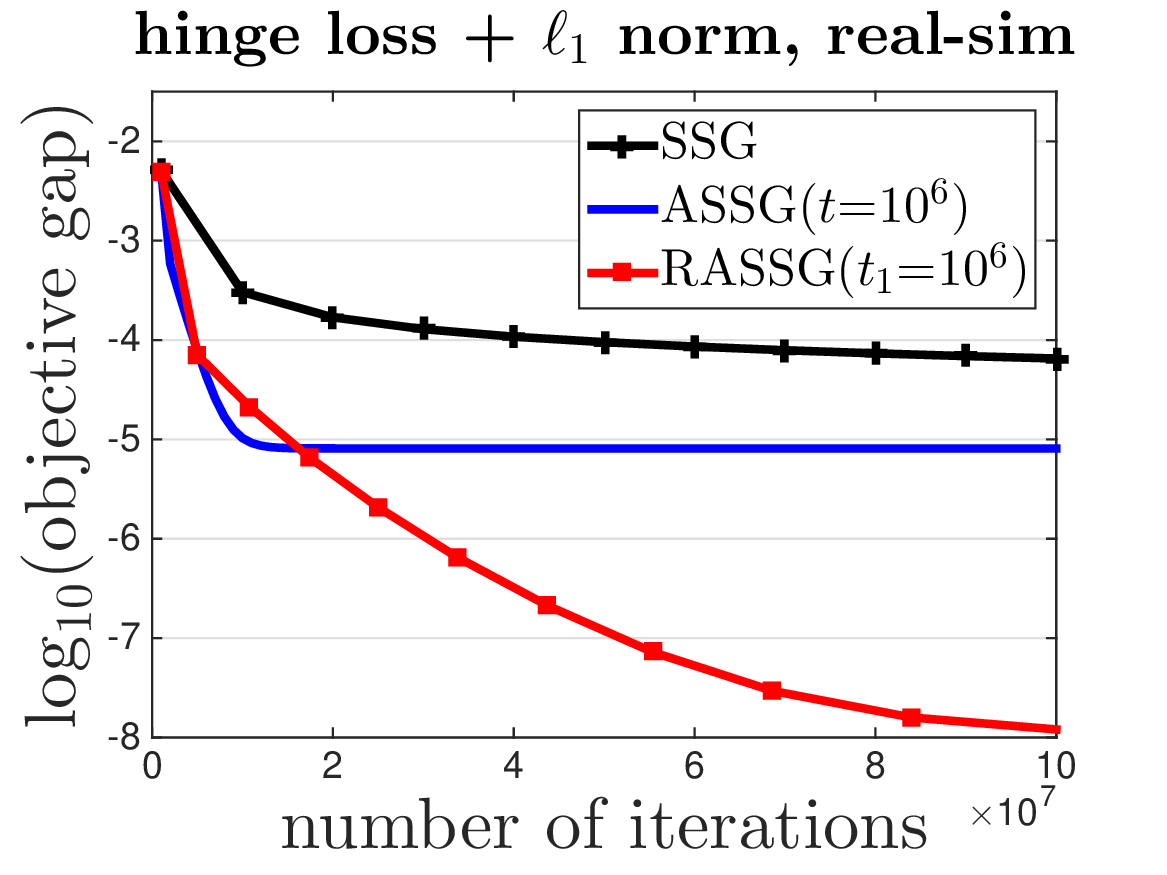}
		\includegraphics[scale=.25]{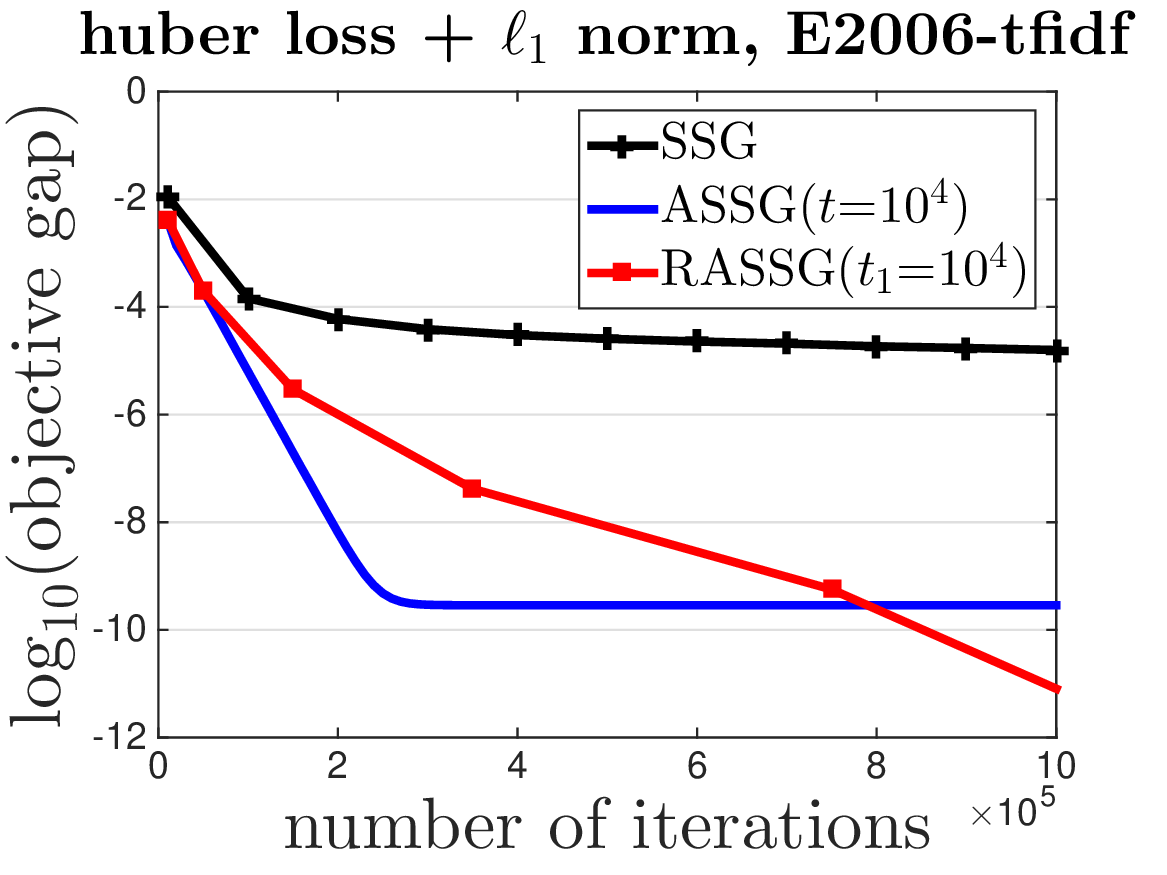}
		\includegraphics[scale=.25]{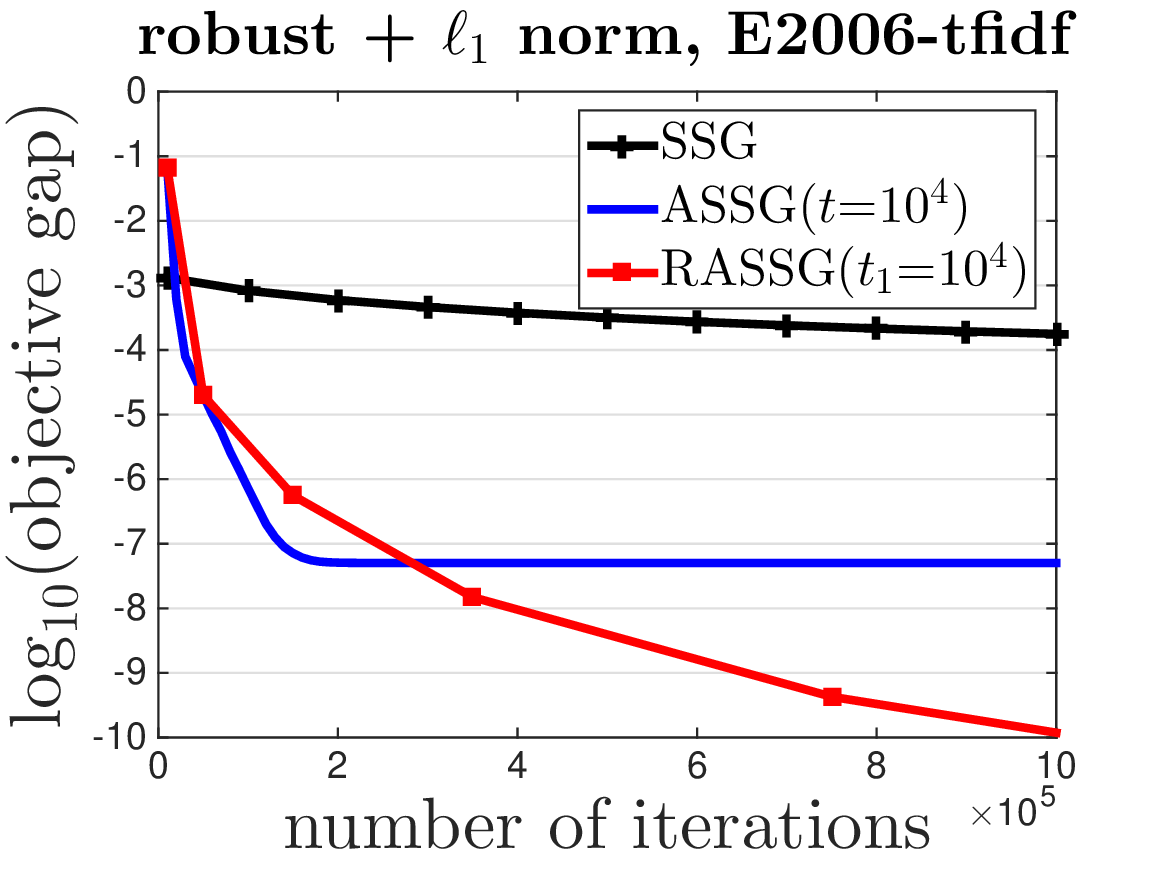}
}
\vspace*{8pt}
\caption{Comparison of different algorithms for solving different problems on different datasets ($\lambda=10^{-2}$).}
\label{fig02}
\end{figure}

\begin{figure}[th]
%\begin{center}
\subfigure[$\lambda=10^{-4}$]{
\centerline{\includegraphics[scale=.25]{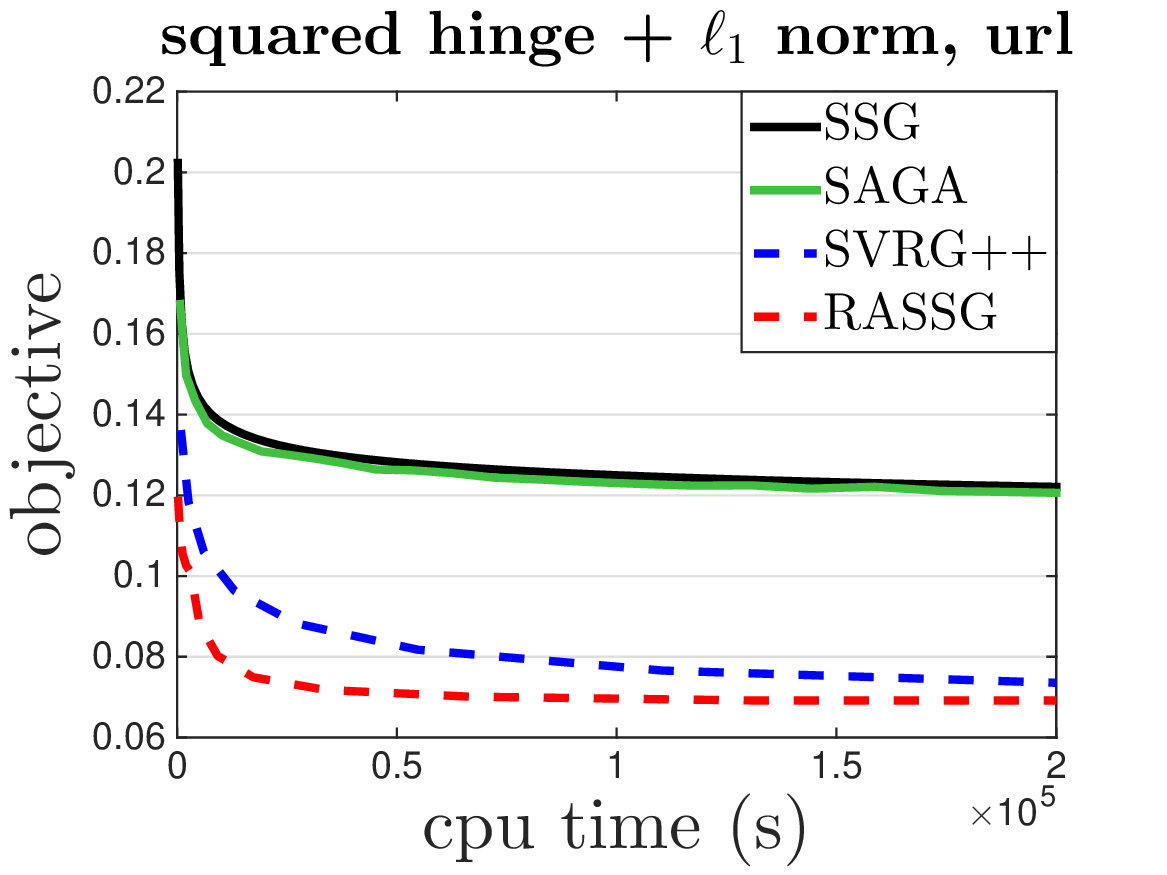}
		\includegraphics[scale=.25]{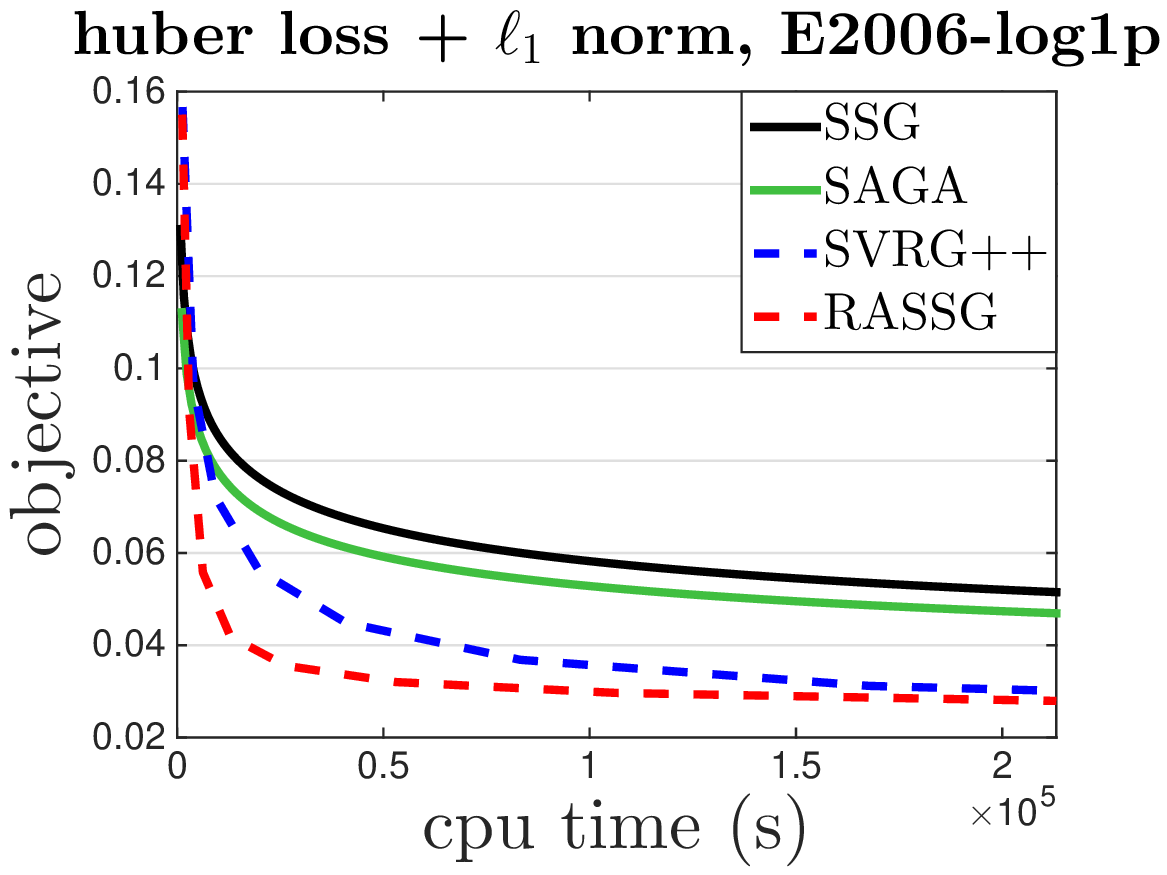}
		}	
}
\subfigure[$\lambda=10^{-2}$]{
\centerline{\includegraphics[scale=.25]{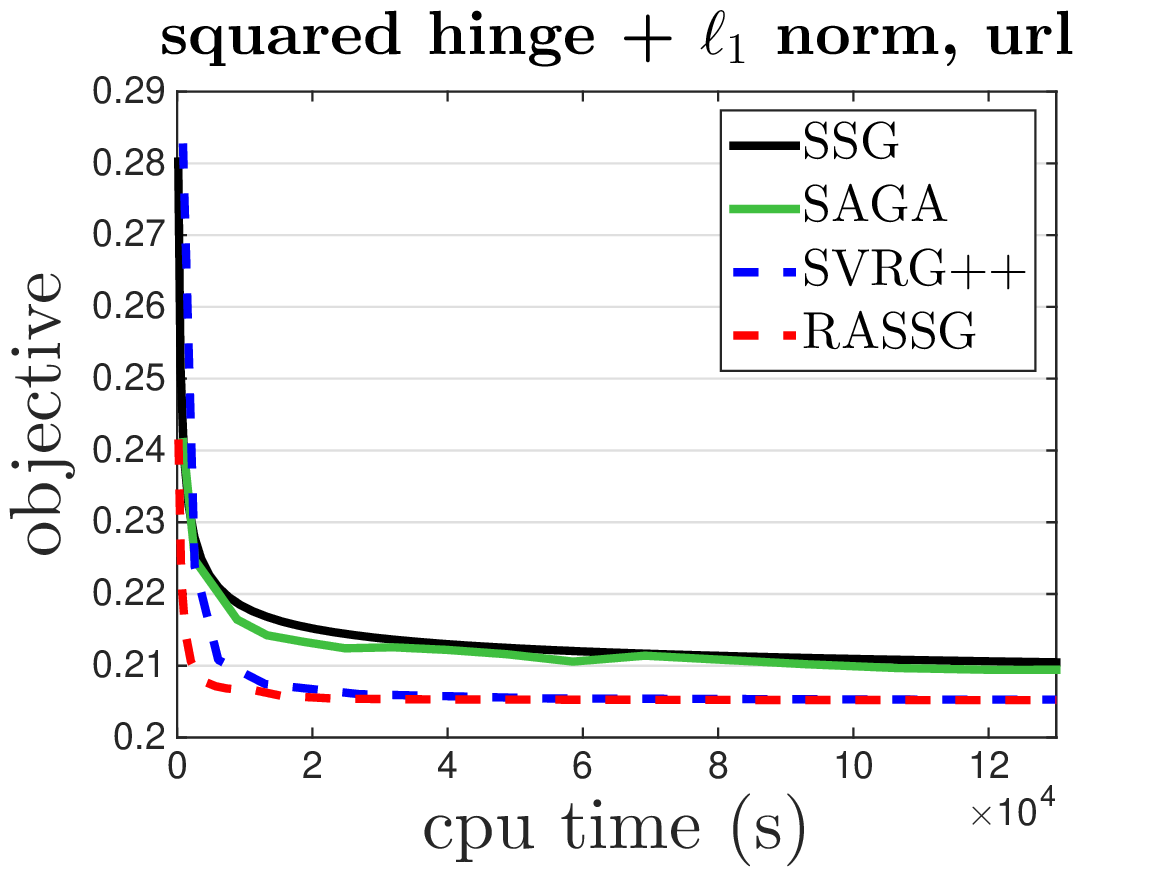}
		\includegraphics[scale=.25]{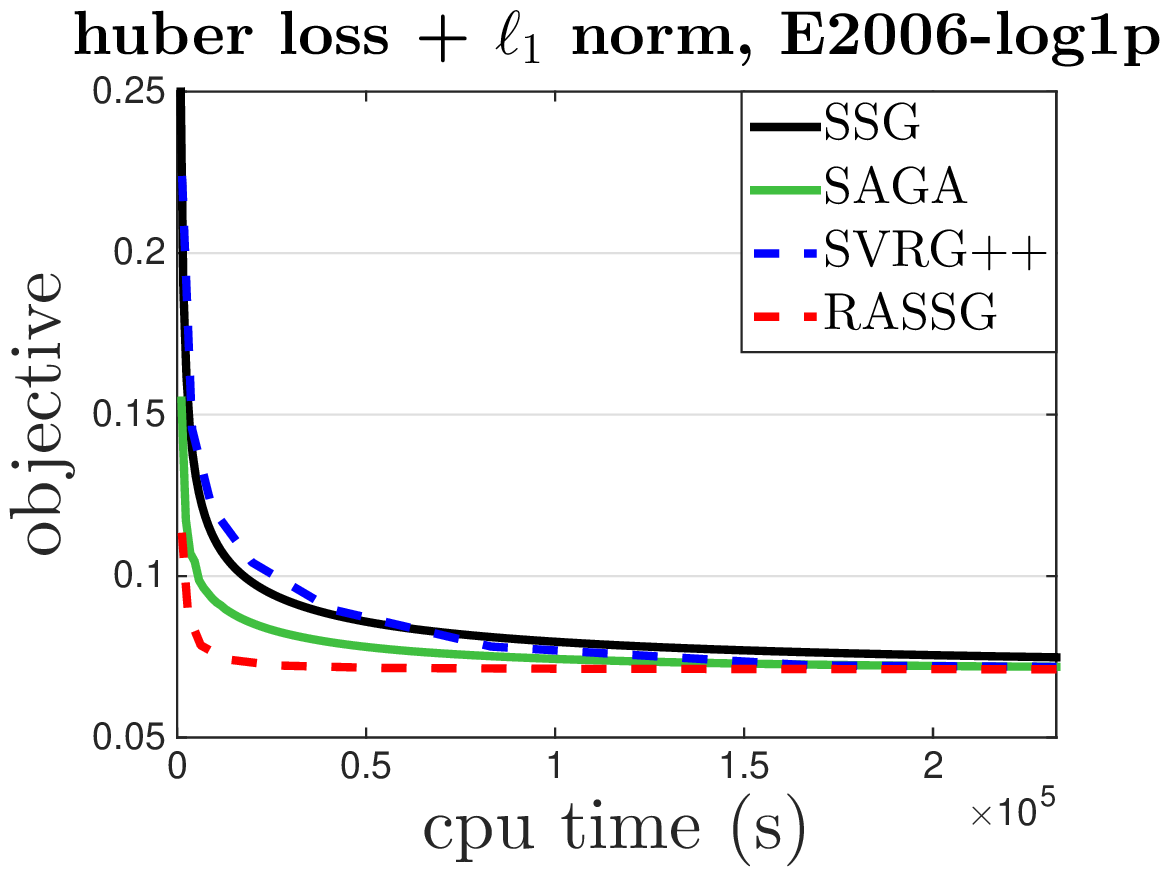}
		}
}
\vspace*{8pt}
\caption{Comparison of different algorithms for solving different problems on different datasets.}
\label{fig03}
\vspace*{8pt}
%\end{center}
\end{figure}

\begin{figure}[th]
\centerline{
		\includegraphics[scale=.25]{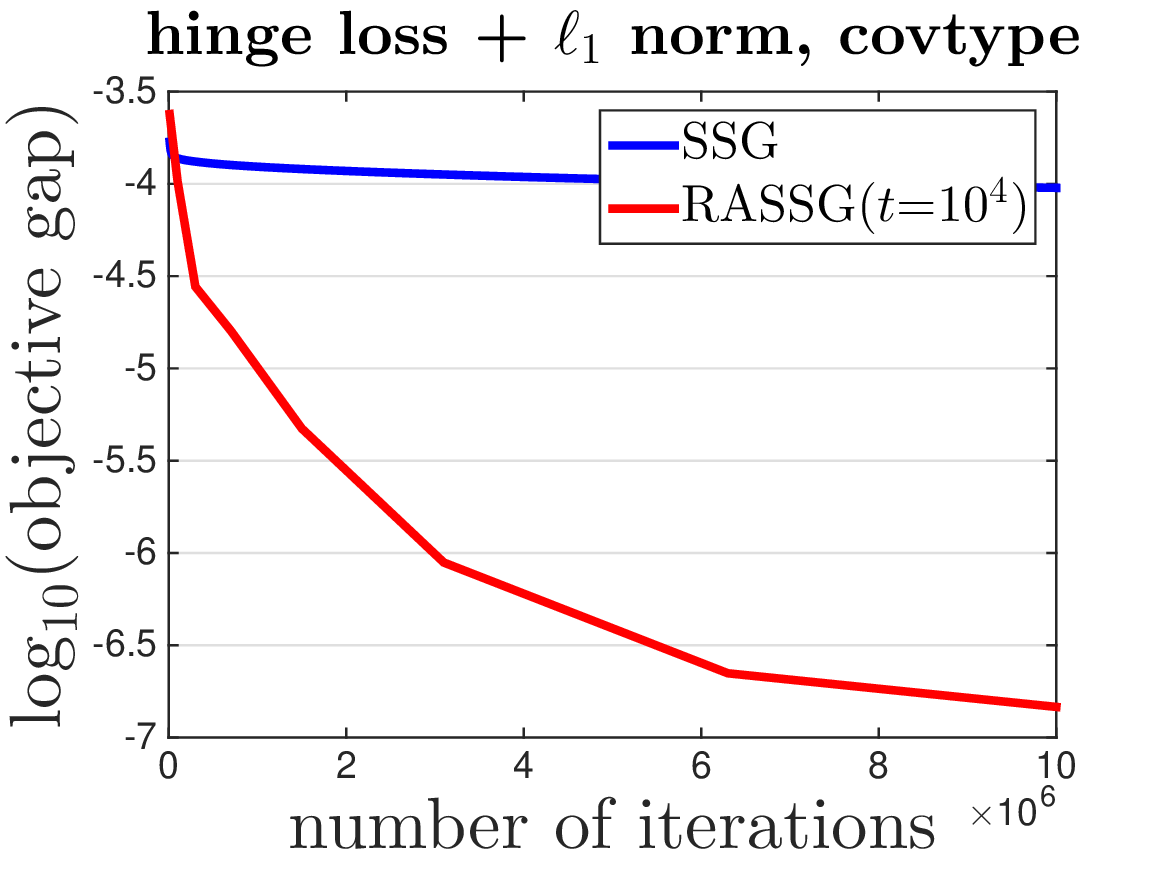}		
		\includegraphics[scale=.25]{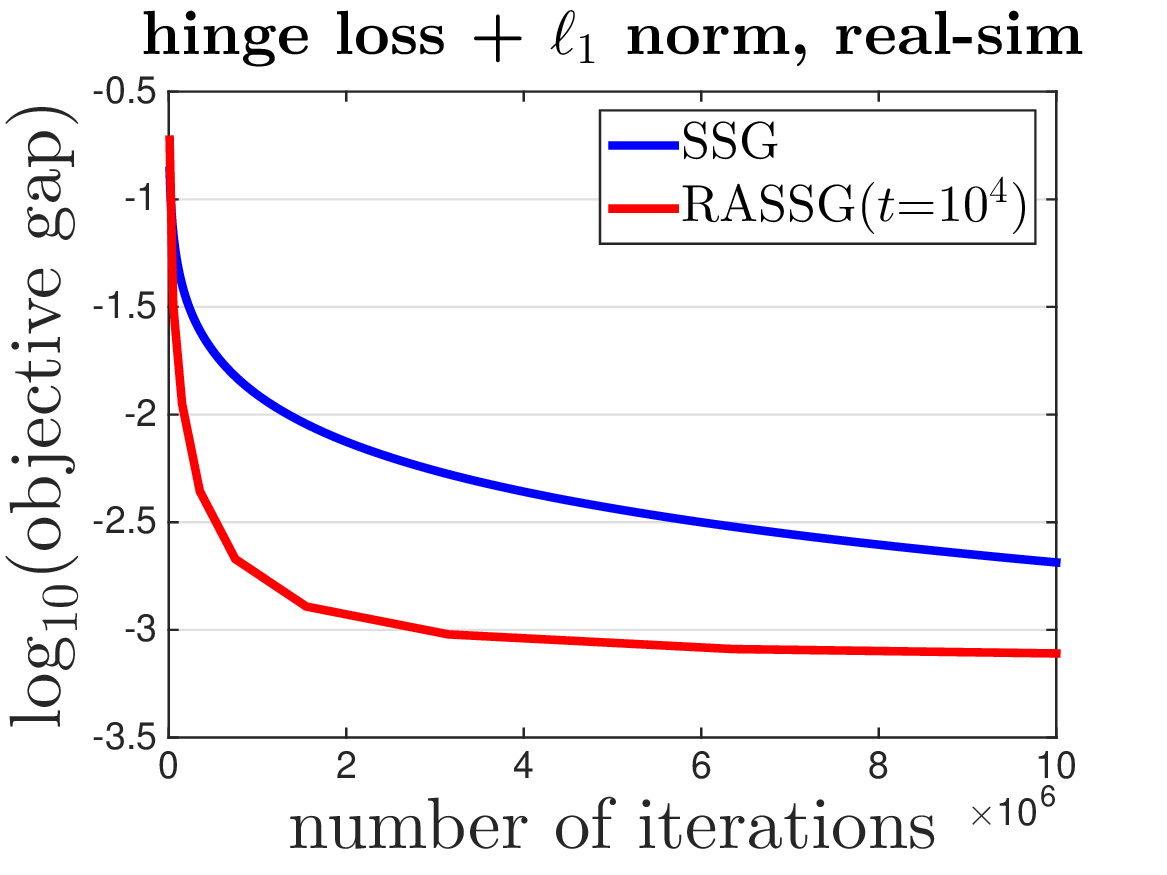}
		\includegraphics[scale=.25]{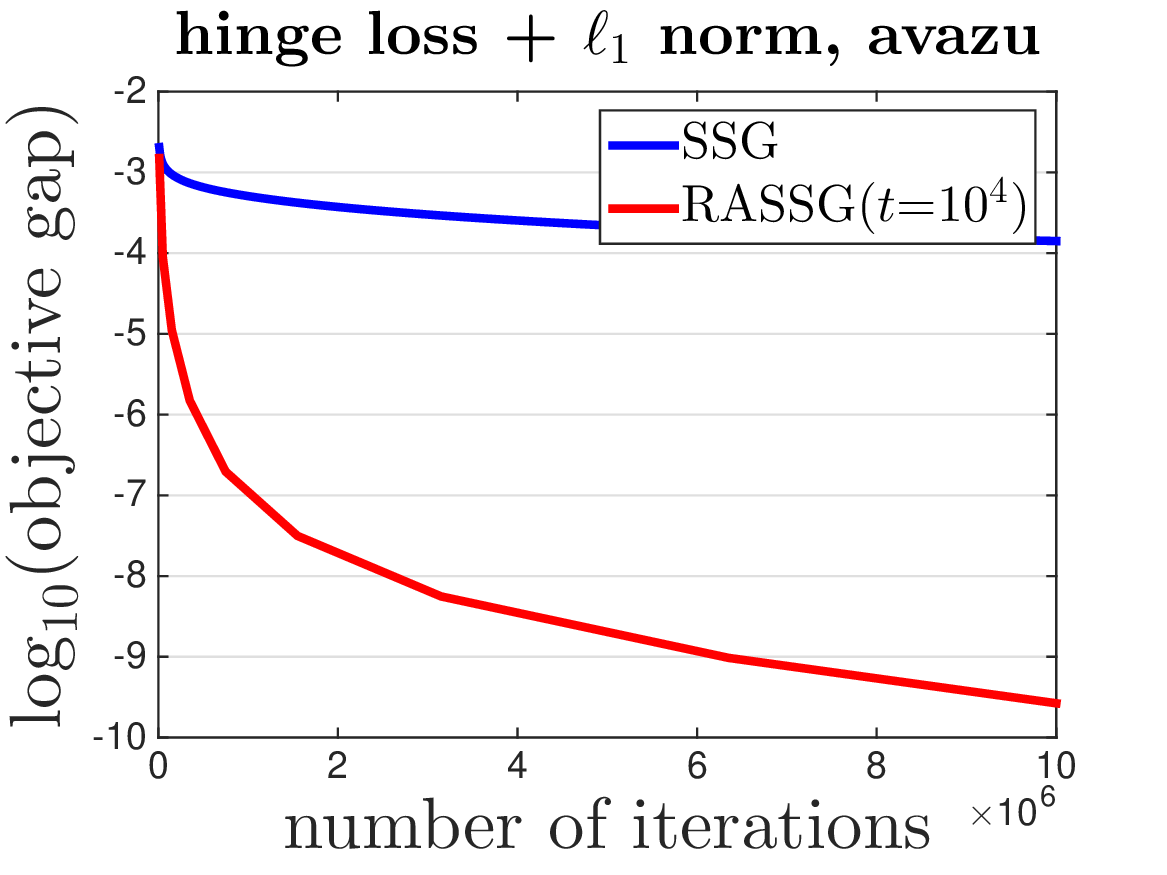}
}
\centerline{
		\includegraphics[scale=.25]{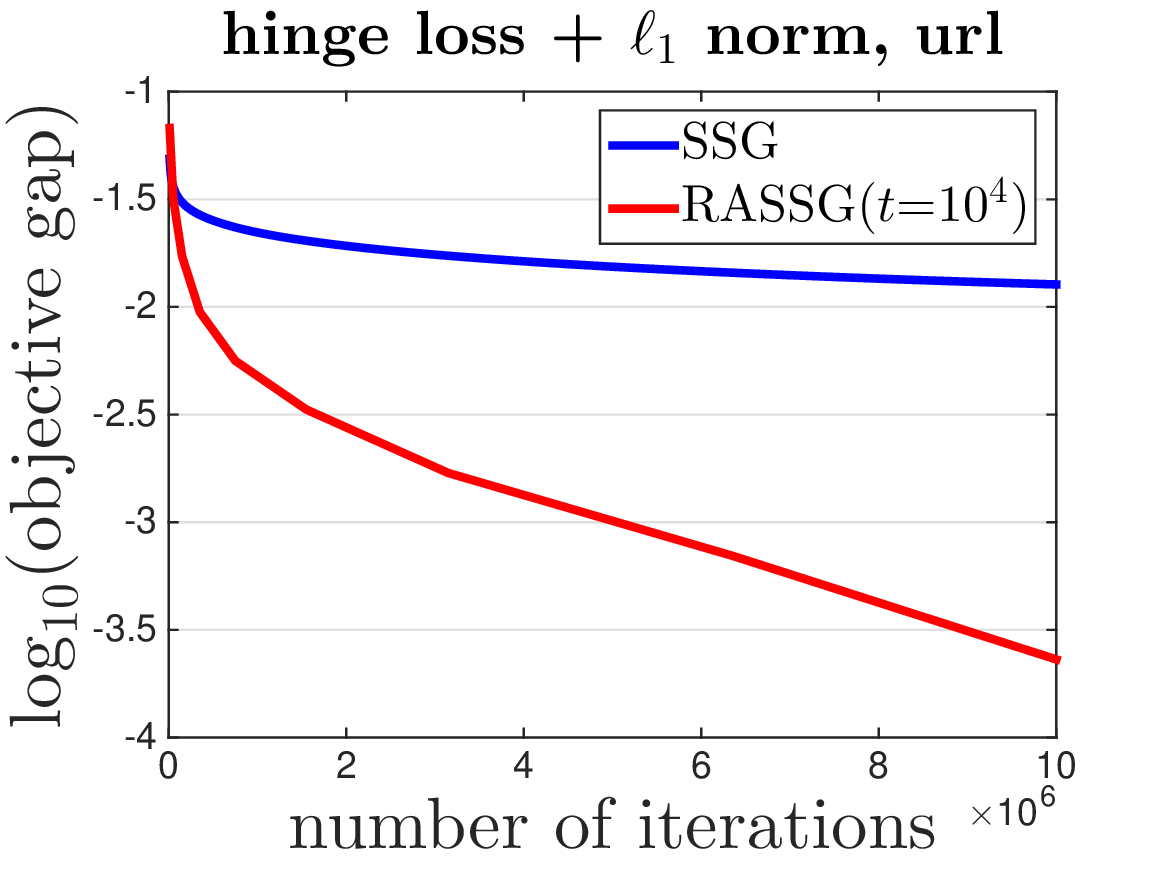}
		\includegraphics[scale=.25]{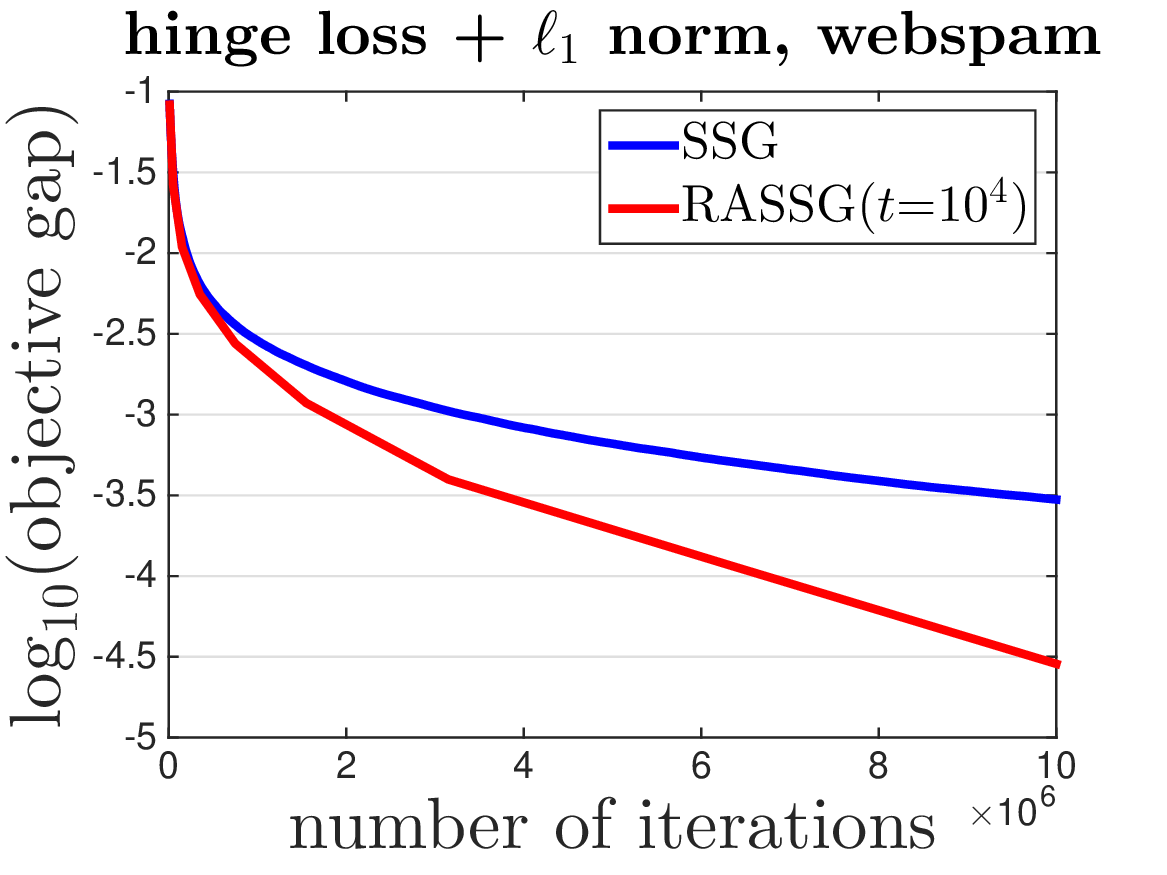}
		\includegraphics[scale=.25]{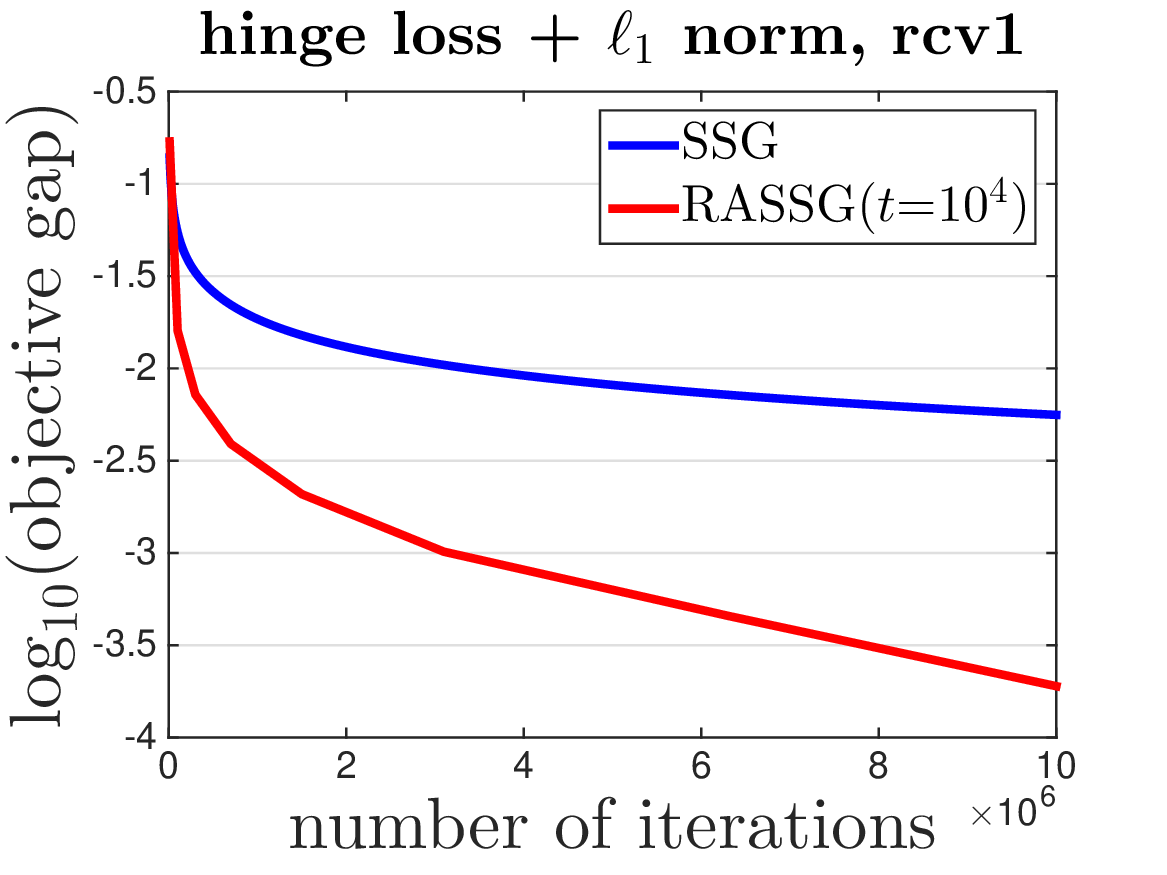}
}
\centerline{								
		\includegraphics[scale=.25]{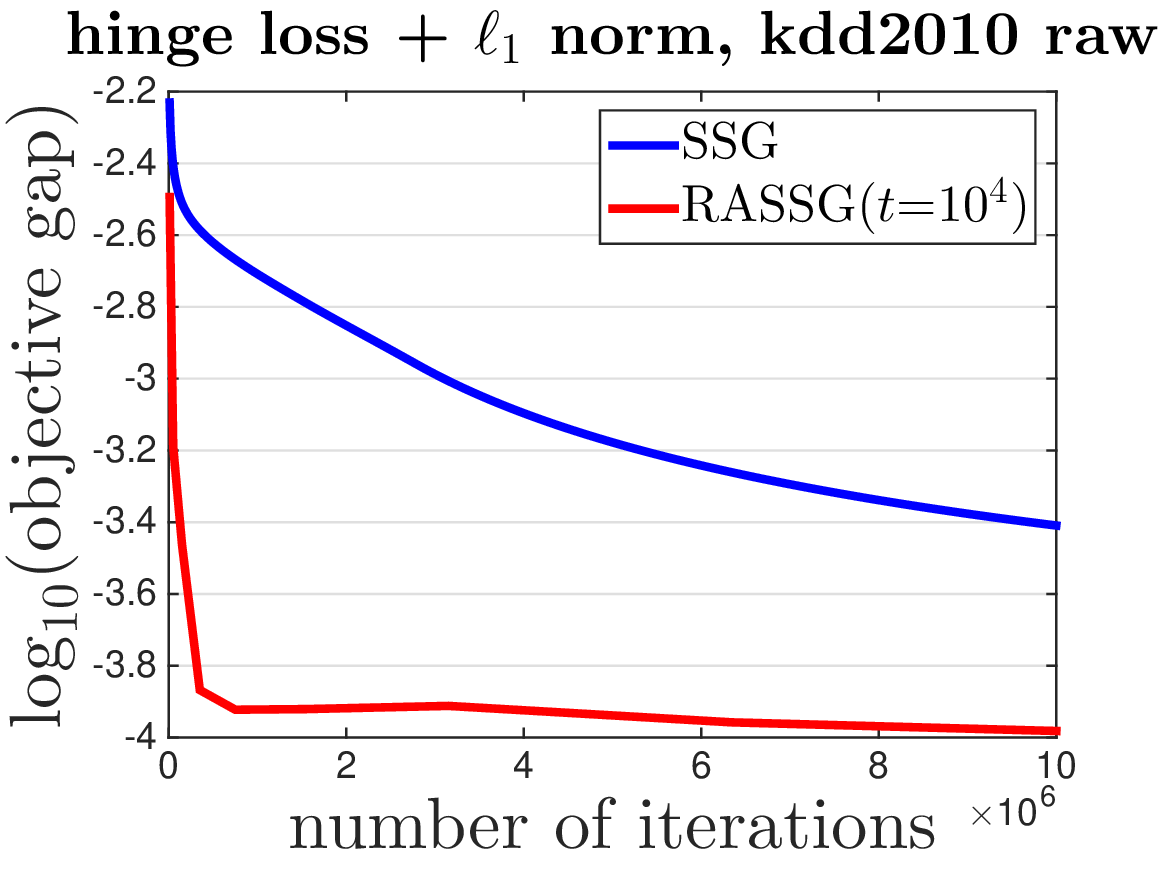}
		\includegraphics[scale=.25]{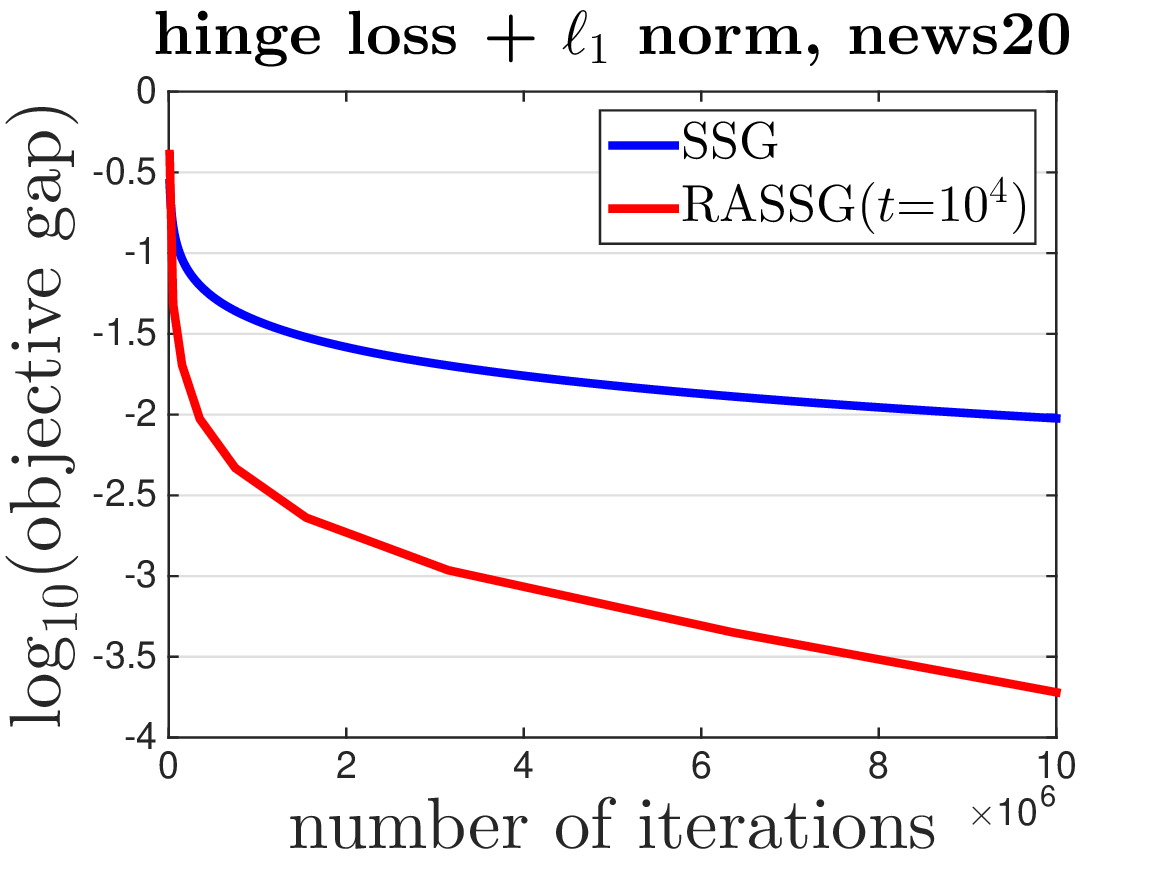}
		\includegraphics[scale=.25]{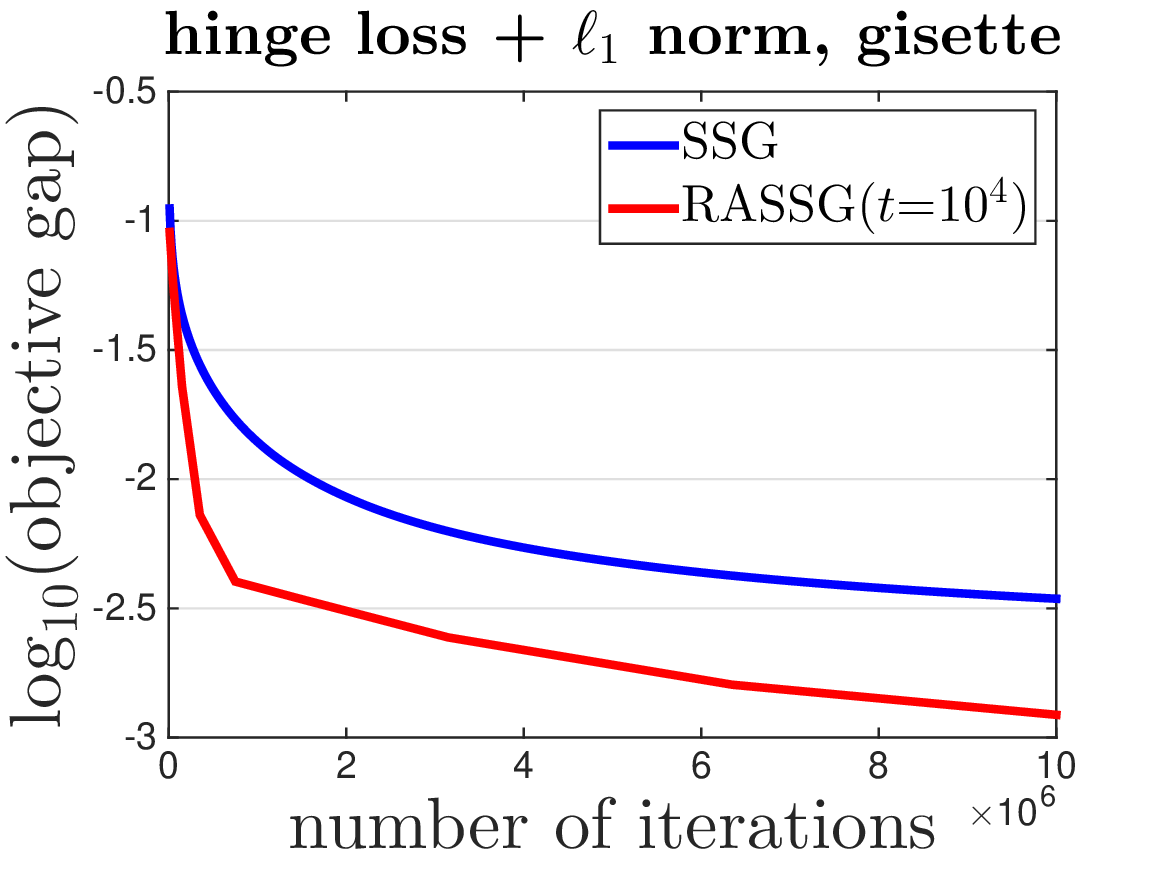}
}
\vspace*{8pt}
\caption{Comparison of SSG and RASSG-s on different datasets ($\lambda=10^{-4}$).}
\label{fig04}
\vspace*{8pt}\end{figure}

\section{Experiments}\label{sec:exp}
In this section, we perform some experiments to demonstrate effectiveness of proposed algorithms. For the first two experimens, we use very large-scale datasets from libsvm website in experiments, including covtype.binary, real-sim, url for classification, million songs, E2006-tfidf, E2006-log1p for regression. While for the last experimenst, we only consider classification problem and use nine datasets from libsvm website including covtype.binary, real-sim, avazu, gisette, kdd 2010 raw, news20.binary, rcv1.binary, url and webspam. The detailed statistics of these datasets are shown in Table~\ref{table1}.

%First, we provide empirical evaluations in support of our theoretical findings on 
\paragraph{Effectiveness of ASSG-c and RASSG-c for non-smooth problems.} We first compare ASSG with SSG on three tasks: $\ell_1$ norm regularized hinge loss minimization for linear classification, $\ell_1$ norm regularized Huber loss minimization for linear regression, and $\ell_1$ norm regularized $p$-norm robust regression with a loss function $\ell(\w^{\top}\x_i, y_i)=|\w^{\top}\x_i - y_i|^p$. The regularization parameter $\lambda$ is set to be $10^{-4}$/$10^{-2}$ in all tasks. We set $\gamma=1$ in Huber loss and $p=1.5$ in robust regression.  %For linear classification, we set the regularization  parameter $\lambda = 0.01$ both for covtype and real-sim datasets. For linear regression with huber loss, we set $\lambda = 0.01$ and $\delta=1$, and for robust regression problems we use $\lambda = 0.01$ and $p=1.5$. 
In all experiments, we use the constrained variant of ASSG, i.e., ASSG-c. 
%To explore the convergence behavior, we compare ASSG with different values of the number of iterations per-stage. In particular, we set the value of $t$ to be $10^3, 10^4$, and $10^5$.  For each task on each dataset, the initial step size $\eta_1$ and initial radius $D_1$ are same for different $t$. We present the results on the left panels of Figure~\ref{fig01}, Figure~\ref{fig02} and Figure~\ref{fig03}. In each figure, we plot the log difference between the objective value and the smallest obtained objective value (to which we refer as objective gap) vs number of stages and log difference between the objective value and the converged value (to which we refer as level gap) vs number of stages. The firgures show that  ASSG converges to an $\epsilon$-level set with different $t$ and the convergence rate is linear in terms of the number of stages, verifying our theoretical analysis. 
%We also compare with SSG to demonstrate the effectiveness of ASSG. 
For fairness, we use the same initial solution with all zero entries for all algorithms. We use a decreasing step size proportional to $1/\sqrt{\tau}$ ($\tau$ is the iteration index) in SSG. The initial step size of SSG is tuned in a wide range to obtain the fastest convergence. The step size of ASSG in the first stage is also tuned around the best initial step size of SSG.  The value of $D_1$ in both ASSG and RASSG is set to $100$ for all problems. In implementing the RASSG, we restart every 5 stages with $t$ increased by a factor of $1.15$, $2$ and $2$ respectively for hinge loss, Huber loss and robust regression. We tune the parameter $\omega$ among $\{0.3, 0.6, 0.9, 1\}$.  We report the results of ASSG with a fixed number of iterations per-stage $t$ and RASSG with an increasing sequence of $t$.  %We set the initial value of $t_1$ in RASSG same as the that in ASSG with smaller value of $t$. 
The results are plotted in Figure~\ref{fig01} and Figure~\ref{fig02} %Figure~\ref{fig02} and Figure~\ref{fig03}.
in which we plot the log difference between the objective value and the smallest obtained objective value (to which we refer as objective gap) versus number of iterations.
%from which we observe 
The figures show that (i) ASSG can quickly converge to a certain level set determined implicitly by $t$; (ii) RASSG converges much faster than SSG to more accurate solutions; (iii) RASSG can gradually decrease the objective value.

\paragraph{Effectiveness of ASSG-c and RASSG-c for smooth problems.}  Second, we compare  RASSG with state-of-art stochastic optimization algorithms for solving a finite-sum problem with a smooth piecewise quadratic loss (e.g., squared hinge loss, huber loss) and an $\ell_1$ norm regularization. In particular, we compare with two variance-reduction algorithms that leverage the smoothness of the function, namely SAGA~\citep{DBLP:conf/nips/DefazioBL14} and SVRG++~\citep{allen2016improved}. We conduct experiments on two  high-dimensional  datasets url and E2006-log1p and fix the regularization parameter $\lambda = 10^{-4}$ or $\lambda= 10^{-2}$. We use $\delta=1$ in Huber loss. For RASSG, we start from $D_1=100$ and $t_1 = 10^3$, then restart it every $5$ stages with $t$ increased by a factor of $2$. We tune the initial step sizes for all algorithms in a wide range and set the values of parameters in SVRG++ followed by~\citep{allen2016improved}. We plot the objective versus the CPU time (second) in Figure~\ref{fig03}. The results show that RASSG converges faster than other three algorithms for the two tasks. This is not surprising considering that RASSG, SAGA and SVRG++  suffer from an iteration complexity of  $\widetilde O(1/\epsilon)$, $O(n/\epsilon)$, and $O(n\log(1/\epsilon)+1/\epsilon)$, respectively.

\paragraph{Effectiveness of RASSG-s.}  Finally, we compare RASSG-s with SSG on $\ell_1$ norm regularized hinge loss minimization for linear classification. The regularization parameter $\lambda$ is set to be $10^{-4}$, and the initial iteration number of RASSG-s is set to be $10,000$. We fixed the total number of iterations as $1,000,000$ both for SSG and RASSG-s. Although the parameter $\theta=1$ in the considered task, we can always reduce it to $\theta = \frac{1}{2}$~\citep{xu2017AdaSVRG}. Thus we set GGC parameter $\theta=\frac{1}{2}$ in this experiment. The other parameters of SSG and RASSG-s are set as same as the first experiment. The results are presented in Figure~\ref{fig04}, showing that RASSG-s converges much faster than SSG to more accurate solutions.

\section{Conclusion}
In this paper, we have proposed accelerated stochastic subgradient methods for solving general non-strongly convex stochastic optimization under the functional local growth condition. The proposed methods enjoy a lower iteration complexity than vanilla stochastic subgradient method and also a logarithmic dependence on the impact of the initial solution. We have also made an extension by developing a more practical variant. Applications in machine learning have demonstrated the faster convergence of the proposed methods. %We have also made several extensions by developing more practical variants. Applications in machine learning have demonstrated the faster convergence of the proposed methods. 

%\section*{Acknowledgments}
%This section should come before the References. Funding information may also be included here.
%We thank the anonymous reviewers for their helpful comments. Y. Xu and T. Yang are partially supported by National Science Foundation (IIS-1545995).

%\newpage
\appendix
\section{Proof of Corollary~\ref{cor:ASSGc}}\label{app:cor:assgc}
\begin{proof}
First, we show that for any $1\leq k \leq K$,
\begin{align}\label{ineq1:cor:assgc}
\|\w_k - \w_0\|_2 \leq 2D_1 
\end{align}
When $k=1$, it is easy to show that $\|\w_1 - \w_0\|_2  \leq D_1$, which satisfies inequality (\ref{ineq1:cor:assgc}). When $k\geq 2$, we have 
\begin{align*}
&\|\w_k - \w_0\|_2 \\
  \leq &\|\w_k - \w_{k-1}\|_2 + \|\w_{k-1} - \w_{k-2}\|_2 + \dots + \|\w_2 - \w_1\|_2 + \|\w_1 - \w_0\|_2 \\
  \leq &D_1/2^{k-1} + D_1/2^{k-2} + \dots + D_1/2 + D_1 \leq 2D_1
\end{align*}
where the second inequality is based on the updates of Algorithm~\ref{alg:rssg}.
With probability 1, we have
\begin{align} \label{ineq2:cor:assgc}
\nonumber F(\w_k) -  F_*  =& F(\w_k) - F(\w_0) + F(\w_0) - F_* \\
\leq& \|\partial F(\w_k)\|_2 \|\w_k-\w_0\|_2 + F(\w_0) - F_*  \leq 2GD_1 + \epsilon_0 
\end{align}
where the last inequality using the fact that $\|\partial F(\w_k)\|_2\leq G$, inequality (\ref{ineq1:cor:assgc}) and Assumption~\ref{ass:1} (a).
Based on Theorem~\ref{thm:RSSG}, ASSG-c guarantees that 
\begin{align}\label{ineq3:cor:assgc}
Prob(F(\w_K) - F_*) \leq \epsilon) \geq 1 - \delta
\end{align}
Then
\begin{align*}
\nonumber \mathbb{E} \left[F(\w_K) - F_* \right]   =& \mathbb{E} \left[F(\w_K) - F_*  | F(\w_K) - F_* \leq \epsilon \right] Prob(F(\w_K) - F_* ) \leq \epsilon) \\
\nonumber   &+  \mathbb{E} \left[F(\w_K) - F_*  | F(\w_K) - F_* \geq \epsilon \right] Prob(F(\w_K) - F_* ) \geq \epsilon) \\
 \leq & \epsilon  +   (2GD_1 + \epsilon_0 ) \delta \leq 2\epsilon
\end{align*}
where the first inequality uses inequalities (\ref{ineq2:cor:assgc}) and (\ref{ineq3:cor:assgc}), and the second inequalty is due to $\delta \leq\frac{\epsilon}{2GD_1 + \epsilon_0 }$.
Therefore, ASSG-c achieves that $\mathbb E \left[F(\w_K) - F_*\right] \leq 2 \epsilon$ using 
 at most $ O \left(\lceil \log_2(\frac{\epsilon_0}{\epsilon})\rceil \log\left(\frac{2GD_1+\epsilon_0}{\epsilon}\right)c^2G^2/\epsilon^{2(1-\theta)}\right)$ iterations provided $D_1=O(\frac{c\epsilon_0}{\epsilon^{(1-\theta)}})$.
\end{proof}

\section{Proof of Lemma~\ref{lem:b}}\label{app:lemb}
\begin{proof}
By the optimality of $\wh_*$, we have for any $\w\in\mathcal K$
\begin{align*}
\left(\partial F(\wh_*) + \frac{1}{\beta} (\wh_* - \w_1)\right)^{\top}(\w - \wh_*)\geq 0.
\end{align*}
Let $\w = \w_1$, we have
\begin{align*}
\partial F(\wh_*)^{\top}(\w_1 - \wh_*)\geq \frac{\|\wh_* - \w_1\|^2_2}{\beta}.
\end{align*}
Because $\|\partial F(\wh_*)\|_2\leq G$ due to $\|\partial f(\w; \xi)\|_2\leq G$, then 
\begin{align*}
\|\wh_* - \w_1\|_2\leq \beta G.
\end{align*}
Next, we bound $\|\w_t - \w_1\|_2$. According to the update of $\w_{t+1}$ we have
\begin{align*}
\|\w_{t+1} - \w_1\|_2 \leq \|\w'_{t+1} - \w_1\|_2 = \|-\eta_t\partial f(\w_t; \xi_t) + (1-\eta_t/\beta) (\w_t-\w_1)\|_2.
\end{align*}
We prove $\|\w_t- \w_1\|_2\leq 2\beta G$ by induction. First, we consider $t=1$, where $\eta_t = 2\beta$, then
\begin{align*}
\|\w_2- \w_1\|_2\leq \left\|2\beta\partial f(\w_t; \xi_t)\right\|_2\leq 2\beta G.
\end{align*}
Then we consider any $t\geq 2$, where $\eta_t/\beta\leq 1$. Then
\begin{align*}
&\|\w_{t+1}- \w_1\|_2\leq \left\|-\frac{\eta_t}{\beta} \beta\partial f(\w_t; \xi_t) + \left(1-\frac{\eta_t}{\beta}\right)(\w_t-\w_1)\right\|_2\\
\leq &\frac{\eta_t}{\beta} \beta G + \left(1-\frac{\eta_t}{\beta}\right) 2\beta G \leq 2\beta G.
\end{align*}
Therefore
\begin{align*}
\|\wh_* - \w_t\|_2\leq 3\beta G.
\end{align*}
\end{proof}

\section{Proof of Lemma~\ref{lemma:ssgs}}\label{app:SG}
In this proof, we need the following lemma. 
\begin{lemma}(Lemma 3~\citep{DBLP:conf/nips/KakadeT08})\label{lem:mart}
Suppose $X_1,\ldots, X_T$ is a martingale difference sequence with $|X_t|\leq b$. 
Let 
\begin{align*}
 \textrm{Var}_tX_t =  \textrm{Var}(X_t|X_1,\ldots, X_{t-1}).
\end{align*}
where $ \textrm{Var}$ denotes the variance. 
Let $V=\sum_{t=1}^T \textrm{Var}_tX_t$ be the sum of conditional variance of $X_t$'s. Further, let $\sigma=\sqrt{V}$. Then we have for any $\delta<1/e$ and $T\geq 3$, 
\begin{align*}
\Pr\left(\sum_{t=1}^TX_t>\max\{2\sigma, 3b\sqrt{\log(1/\delta)}\}\sqrt{\log(1/\delta)}\right)\leq 4\delta\log T.
\end{align*}
\end{lemma}

Then, let us start the proof of Lemma~\ref{lemma:ssgs}.
\begin{proof}
Let $\g_t =\partial f(\w_t; \xi_t) + (\w_t - \w_1)/\beta$ and $\partial \Fh(\w_t) = \partial F(\w_t) + (\w_t - \w_1)/\beta$. Note that $\|\g_t\|_2\leq 3G$. 
According to the standard analysis for the stochastic gradient method we have 
\begin{align*}
\g_t^{\top}(\w_t - \wh_*)\leq \frac{1}{2\eta_t}\|\w_t - \wh_*\|_2^2 - \frac{1}{2\eta_t}\|\w_{t+1} - \wh_*\|_2^2  + \frac{\eta_t}{2} \|\g_t\|_2^2.
\end{align*}
Then 
\begin{align*}
\partial \Fh(\w_t)^{\top}(\w_t - \wh_*)\leq& \frac{1}{2\eta_t}\|\w_t - \wh_*\|_2^2 - \frac{1}{2\eta_t}\|\w_{t+1} - \wh_*\|_2^2  + \frac{\eta_t}{2} \|\g_t\|_2^2 \\
&+ (\partial \Fh(\w_t) - \g_t)^{\top}(\w_t - \wh_*).
\end{align*}
By strong convexity of $\Fh$ we have
\begin{align*}
\Fh(\wh_*) - \Fh(\w_t)\geq \partial \Fh(\w_t)^{\top}(\wh_* - \w_t) + \frac{1}{2\beta}\|\wh_* - \w_t\|_2^2.
\end{align*}
Then 
\begin{align*}
\Fh(\w_t) - \Fh(\wh_*) \leq &\frac{1}{2\eta_t}\|\w_t - \wh_*\|_2^2 - \frac{1}{2\eta_t}\|\w_{t+1} - \wh_*\|_2^2  + \frac{\eta_t}{2} \|\g_t\|_2^2 \\
&+ (\partial \Fh(\w_t) - \g_t)^{\top}(\w_t - \wh_*)- \frac{1}{2\beta}\|\wh_* - \w_t\|_2^2\\
 \leq& \frac{1}{2\eta_t}\|\w_t - \wh_*\|_2^2 - \frac{1}{2\eta_t}\|\w_{t+1} - \wh_*\|_2^2  + \frac{\eta_t}{2} \|\g_t\|_2^2 \\
 &+ \underbrace{(\partial F(\w_t) - \partial f(\w_t; \xi_t))^{\top}(\w_t - \wh_*)}\limits_{\zeta_t}- \frac{1}{2\beta}\|\wh_* - \w_t\|_2^2.
\end{align*}
By summing the above inequalities across $t=1,\ldots, T$, we have
\begin{align*}%\label{eqn:toc}
&\sum_{t=1}^T(\Fh(\w_t) - \Fh(\wh_*)) \leq \sum_{t=1}^{T-1}\frac{1}{2}\left(\frac{1}{\eta_{t+1}} - \frac{1}{\eta_{t}} - \frac{1}{2\beta}\right)\|\wh_* - \w_{t+1}\|_2^2  + \sum_{t=1}^T\zeta_t \nonumber\\
&- \frac{1}{4\beta}\sum_{t=1}^T\|\wh_* - \w_t\|_2^2  - \frac{1}{4\beta}\|\wh_* - \w_1\|_2^2 + \frac{1}{2\eta_1}\|\wh_* - \w_1\|_2^2 + \frac{9G^2}{2}\sum_{t=1}^T\eta_t\nonumber \\
 \leq& \sum_{t=1}^T\zeta_t - \frac{1}{4\beta}\sum_{t=1}^T\|\wh_* - \w_t\|_2^2 + 9\beta G^2(1+\log T).
\end{align*}
where the last inequality uses $\eta_t = \frac{2\beta}{t}$. 

Next, we bound R.H.S of the above inequality by using Lemma~\ref{lem:mart}.
%We need the following lemma. 
%%\begin{lemma}\label{lemma:1}
%%Let  $D_T = \sum_{t=1}^{T} \| \w_t - \wh_* \|_2^2$, $\Lambda_T =\sum_{t=1}^{T}\zeta_t$, and $m = \left \lceil{2\log_2 T}\right \rceil $. Assume that $D_T\geq \frac{D^2}{T}$, with a probability at least $1-\delta$ we have
%%\begin{align*}
%%\Lambda_T \leq 4G\sqrt{D_T \log\frac{m}{\delta}} + \frac{4GD}{3} \log\frac{m}{\delta}  
%%\end{align*}
%%where $D = \frac{3G}{\lambda}$. 
%%\end{lemma}
%\begin{lemma}(Lemma 3~\citep{DBLP:conf/nips/KakadeT08})\label{lem:mart}
%Suppose $X_1,\ldots, X_T$ is a martingale difference sequence with $|X_t|\leq b$. 
%Let 
%\begin{align*}
% \textrm{Var}_tX_t =  \textrm{Var}(X_t|X_1,\ldots, X_{t-1}).
%\end{align*}
%where $ \textrm{Var}$ denotes the variance. 
%Let $V=\sum_{t=1}^T \textrm{Var}_tX_t$ be the sum of conditional variance of $X_t$'s. Further, let $\sigma=\sqrt{V}$. Then we have for any $\delta<1/e$ and $T\geq 3$, 
%\begin{align*}
%\Pr\left(\sum_{t=1}^TX_t>\max\{2\sigma, 3b\sqrt{\log(1/\delta)}\}\sqrt{\log(1/\delta)}\right)\leq 4\delta\log T.
%\end{align*}
%\end{lemma}
To proceed the proof of Lemma~\ref{lemma:ssgs}. We let $X_t = \zeta_t$ and $D_T = \sum_{t=1}^T\|\w_t - \wh_*\|_2^2$. Then $X_1,\ldots, X_T$ is a martingale difference sequence. Let $D=3\beta G$. Note that $|\zeta_t|\leq 2GD$. By Lemma~\ref{lem:mart},  for any $\delta<1/e$ and $T\geq 3$, with a probability $1-\delta$ we have
\begin{align*}
    \sum_{t=1}^T\zeta_t\leq \max\left\{2\sqrt{\log(4\log T/\delta)}\sqrt{\sum_{t=1}^T \textrm{Var}_t\zeta_t}, 6GD\log(4\log T/\delta)\right\}.
\end{align*}
Note that 
\begin{align*}
\sum_{t=1}^T \textrm{Var}_t\zeta_t \leq \sum_{t=1}^{T} \E_t[\zeta_t^2] \leq 4G^2 \sum_{t=1}^{T} \| \w_t - \wh_*\|_2^2  = 4G^2D_T.
\end{align*}
As a result, with a probability $1-\delta$,
\begin{align*}
    \sum_{t=1}^T\zeta_t  \leq& 4G\sqrt{\log(4\log T/\delta)}\sqrt{D_T} + 6GD\log(4\log T/\delta)\\
      \leq &16\beta G^2\log(4\log T/\delta) + \frac{1}{4\beta}D_T + 6GD\log(4\log T/\delta).
\end{align*}
As a result, with a  probability $1-\delta$,
\begin{align*}
    \sum_{t=1}^T\zeta_t - \frac{1}{4\beta}\sum_{t=1}^T\|\wh_* - \w_t\|_2^2  \leq &16\beta G^2\log(4\log T/\delta) + 6GD\log(4\log T/\delta)\\
     =& 34\beta G^2\log(4\log T/\delta).
\end{align*}
Thus, with a  probability $1-\delta$
 \begin{align*}
      \Fh(\wh_T) - \Fh(\wh_*)\leq& \frac{34\beta G^2\log(4\log T/\delta)}{ T}  + \frac{9 \beta G^2(1+\log T)}{T} \\
      \leq&  \frac{34\beta G^2(1+\log T + \log (4\log T/\delta))}{ T}.
 \end{align*}
  Using the facts that $ F(\wh_T) \leq  \Fh(\wh_T)$ and $ \Fh(\wh_*)  \leq  \Fh(\w) = F(\w) + \frac{\|\w-\w_1\|_2^2}{2\beta}$, we have 
  \begin{align*}
     F(\wh_T) - F(\w) -  \frac{\|\w-\w_1\|_2^2}{2\beta} \leq \frac{34\beta G^2(1+\log T + \log (4\log T/\delta))}{ T}.
 \end{align*}
\end{proof}

\section{Proof of Theorem~\ref{thm:RPSG}}\label{proof:thm:RPSG}
\begin{proof}
%We assume that the epoch solutions $\w_k \notin \mathcal S_{2\epsilon}$ for $k = 1, 2, \dots, K-1$ and induce how many stages are sufficient for reaching to a point in $2\epsilon$-sublevel set with a high probability. Let $\w_{k,\epsilon}^\dag$ denote the closest point to $\w_k$ in the $\epsilon$ sublevel set. Since $\mathcal S_{\epsilon}\subseteq \mathcal S_{2\epsilon}$, therefore $\w_k \not\in \mathcal S_\epsilon$ and consequentially  $F(\w_{k,\epsilon}^\dagger)=F_* + \epsilon$ (using the KKT condition).
Let $\w_{k,\epsilon}^\dag$ denote the closest point to $\w_k$ in the $\epsilon$ sublevel set. 
Define  $\epsilon_k \triangleq\frac{\epsilon_0}{2^k}$. First, we note that  $\beta_k \geq  \frac{2c^2\epsilon_{k-1}}{\epsilon^{2(1-\theta)}}$. We will show by induction that $F(\w_k) - F_*\leq \epsilon_k +\epsilon$ for $k=0,1,\dots$ with a high probability, which leads to our conclusion when $k=K$.
The inequality holds obviously for $k=0$. Conditioned on $F(\w_{k-1}) - F_*\leq \epsilon_{k-1} + \epsilon$, we will show that $F(\w_k) - F_*\leq \epsilon_k +\epsilon$ with a high probability.
We apply Lemma~\ref{lemma:ssgs} to the $k$-th stage of Algorithm~\ref{alg:2} conditioned on the randomness in previous stages. With a probability at least $1-\tilde\delta$ we have
\begin{equation}\label{eqn:sg2}
F(\w_k) - F(\w_{k-1,\epsilon}^\dag) \leq  \frac{\|\w_{k-1,\epsilon}^\dag - \w_{k-1}\|_2^2}{2\beta_k}+\frac{34\beta_kG^2\left(1+\log t + \log\left(\frac{4\log t}{\tilde\delta}\right)\right)}{t}. 
\end{equation}
%%%%%%%%%%%%%%%%%%%%%%%%%%%%%%%%%%%%%%%%%%%%%%%%%%%%%%%%%%%%%%%%%%%%%%%%%%%%%%%%%%%%
%%%%%%%%%%%%%%%%%%%%%%%%%%%%%%%%%%%%%%%%%%%%%%%%%%%%%%%%%%%%%%%%%%%%%%%%%%%%%%%%%%%%
%%%%%%%%%%%%%%%%%%%%%%%%%%%%%%%%%%%%%%%%%%%%%%%%%%%%%%%%%%%%%%%%%%%%%%%%%%%%%%%%%%%%
\iffalse
We now consider two cases for $\w_{k-1}$. First, we assume $F(\w_{k-1}) - F_* \leq \epsilon$, i.e. $\w_{k-1}\in\mathcal S_{\epsilon}$. Then we have $\w_{k-1,\epsilon}^\dagger=\w_{k-1}$ and
\begin{equation*}
F(\w_k)  - F(\w_{k-1,\epsilon}^\dagger) \leq  \frac{34\beta_kG^2(1+\log t + \log(4\log t/\tilde\delta))}{t}  \leq  \frac{\epsilon_{k}}{2}.
\end{equation*}
The last inequality uses the fact that $t \geq  \frac{136\beta_1G^2(1+\log (4\log t/\tilde\delta)+\log t)}{\epsilon_0}$.
As a result,
\begin{equation*}
F(\w_k)  - F_* \leq F(\w_{k-1,\epsilon}^\dagger)  -F_* + \frac{\epsilon_{k}}{2} \leq \epsilon + \epsilon_k. 
\end{equation*}
Next, we consider $F(\w_{k-1}) - F_* > \epsilon$, i.e. $\w_{k-1}\notin\mathcal S_{\epsilon}$. Then we have $F(\w_{k-1,\epsilon}^\dagger)-F_* = \epsilon$.
\fi
%%%%%%%%%%%%%%%%%%%%%%%%%%%%%%%%%%%%%%%%%%%%%%%%%%%%%%%%%%%%%%%%%%%%%%%%%%%%%%%%%%%%
%%%%%%%%%%%%%%%%%%%%%%%%%%%%%%%%%%%%%%%%%%%%%%%%%%%%%%%%%%%%%%%%%%%%%%%%%%%%%%%%%%%%
%%%%%%%%%%%%%%%%%%%%%%%%%%%%%%%%%%%%%%%%%%%%%%%%%%%%%%%%%%%%%%%%%%%%%%%%%%%%%%%%%%%% 
Similar to the proof of Theorem~\ref{thm:RSSG}, by Lemma~\ref{lem:key}, we have
\begin{equation}\label{eqn:sg2:leb}
\|\w^\dagger_{k-1,\epsilon} - \w_{k-1} \|_2 \leq  \frac{c \epsilon_{k-1}}{\epsilon^{1-\theta}}.
\end{equation}
%We can get
%\begin{equation*}
%F(\w_k) - F(\w_{k-1,\epsilon}^\dag) \leq  \frac{\lambda_k}{2} \frac{c^2 \epsilon_{k-1}^2}{\epsilon^{2(1-\theta)}} + \frac{34G^2(1+\log t + \log(4\log t/\tilde\delta))}{\lambda_k t} 
% \end{equation*}
Combining (\ref{eqn:sg2}) and (\ref{eqn:sg2:leb}), we have
\begin{equation*}
F(\w_k) - F(\w_{k-1,\epsilon}^\dag) \leq  \frac{1}{2\beta_k} \left(\frac{c \epsilon_{k-1}}{\epsilon^{1-\theta}}\right)^2+\frac{34\beta_kG^2(1+\log t + \log(4\log t/\tilde\delta))}{t}. 
\end{equation*}
Using the fact that $\beta_k\geq  \frac{2c^2\epsilon_{k-1}}{\epsilon^{2(1-\theta)}}$ and  $t\geq \frac{68\beta_kG^2(1+\log t + \log(4\log t/\tilde\delta))}{\epsilon_k}  = \frac{136\beta_1G^2(1+\log t + \log(4\log t/\tilde\delta))}{\epsilon_0}$, we get
\begin{equation*}
F(\w_k) - F(\w_{k-1, \epsilon}^\dagger) \leq  \frac{\epsilon_{k-1} }{4} +\frac{\epsilon_k}{2} = \epsilon_k,
\end{equation*}
which together with the fact that $ F(\w^\dagger_{k-1,\epsilon})-F_*\leq\epsilon$ by definition of $\w^\dagger_{k-1,\epsilon}$ implies
\begin{equation*}
F(\w_k) - F_* \leq \epsilon  + \epsilon_k.
\end{equation*}
Therefore by induction, we have with a probability at least $(1-\tilde \delta)^K$,
\begin{equation*}
F(\w_K) - F_*\leq \epsilon_K + \epsilon = \frac{\epsilon_0}{2^K} +\epsilon\leq 2\epsilon,
\end{equation*}
where the last inequality is due to the value of $K=\lceil \log_2(\frac{\epsilon_0}{\epsilon})\rceil$. Since $\tilde \delta = \delta / K$, then $(1-\tilde \delta)^K \geq 1 - \delta$.
\end{proof}

\section{Proof of Lemma~\ref{lemm:wo}}\label{app:lemm:wo}
\begin{proof}
Let rewrite the update of $\w_{\tau+1}$ in $k$-th epoch as
\begin{equation*}
  \w' = \w_\tau - \eta \partial f(\w_\tau; \xi_\tau),\\
  \w_{\tau+1}  = \Pi_{\mathcal K}[ \w'].
\end{equation*} 
Then for any fixed $\w\in\mathcal K$ we have
\begin{align*}
\frac{1}{2\eta} \|\w_{\tau+1} - \w\|^2 \leq  & \frac{1}{2\eta} \|\w' - \w\|^2 = \frac{1}{2\eta} \|\w_\tau - \eta \partial f(\w_\tau; \xi_\tau) - \w\|^2 \\
= &\frac{1}{2\eta} \|\w_\tau - \w\|^2  - \langle \partial f(\w_\tau; \xi_\tau), \w_\tau - \w\rangle + \frac{\eta}{2} \| \partial f(\w_\tau; \xi_\tau) \|^2,
\end{align*} 
which implies
\begin{align*}
\langle \partial F(\w_\tau), \w_\tau - \w\rangle \leq & \frac{1}{2\eta} \|\w_\tau - \w\|^2  - \frac{1}{2\eta} \|\w_{\tau+1} - \w\|^2  + \frac{\eta}{2} \|  \partial f(\w_\tau; \xi_\tau) \|^2\\
 & - \langle \partial f(\w_\tau; \xi_\tau) - \partial F (\w_\tau), \w_\tau - \w\rangle.
\end{align*} 
By the convexity of $F(\w)$, i.e., $F(\w) - F(\w_\tau)\geq \langle \partial F(\w_\tau), \w - \w_\tau\rangle$, and Assumption~\ref{ass:1} (c), then
\begin{align*}
F(\w_\tau) - F(\w) 
\leq & \frac{1}{2\eta} \|\w_\tau - \w\|^2  - \frac{1}{2\eta} \|\w_{\tau+1} - \w\|^2  + \frac{\eta G^2}{2} \\
 &- \langle \partial f(\w_\tau; \xi_\tau) - \partial F (\w_\tau), \w_\tau - \w\rangle.
\end{align*} 
Taking expectation over $1,\dots, \tau$, we have
\begin{equation*}
\E [ F(\w_\tau) - F(\w)]  \leq  \frac{1}{2\eta} \E [\|\w_\tau - \w\|^2]  - \frac{1}{2\eta} \E [\|\w_{\tau+1} - \w\|^2]  + \frac{\eta G^2}{2},
\end{equation*} 
where uses the fact that $\E[\langle \partial f(\w_\tau; \xi_\tau) - \partial F (\w_\tau), \w_\tau - \w\rangle] = 0$.
By summing the above inequalities across $\tau=1,\ldots, t$, we have
\begin{equation*}%\label{eqn:toc}
\sum_{\tau=1}^t(\E [ F(\w_\tau) - F(\w)] )\leq \frac{1}{2\eta} \|\w_1 - \w\|^2 + \frac{\eta G^2 t}{2}.
\end{equation*}
It implies
\begin{equation*}%\label{eqn:toc}
\E \left[\frac{1}{t}\sum_{\tau=1}^t( F(\w_\tau) - F(\w))\right]\leq \frac{1}{2\eta t} \|\w_1 - \w\|^2 + \frac{\eta G^2}{2}.
\end{equation*}
We complete the proof by using the convexity of $F(\w)$.
\end{proof}

\section{Monotonicity of $B_\epsilon/\epsilon$}\label{app:mono}
\begin{lemma}\label{lem:key2}
$\frac{B_\epsilon}{\epsilon}$ is monotonically decreasing in $\epsilon$. 
\end{lemma}
\begin{proof}
Consider $\epsilon'>\epsilon>0$. Let $\x_{\epsilon'}$ be any point on $\mathcal L_{\epsilon'}$ such that $dist(\x_{\epsilon'},\Omega_*)=B_{\epsilon'}$ and $\x_{\epsilon'}^*$ be the closest point to $\x_{\epsilon'}$ in $\Omega_*$ so that $\|\x_{\epsilon'}^*-\x_{\epsilon'}\|=B_{\epsilon'}$. We define a new point between $\x_{\epsilon'}$ and $\x_{\epsilon'}^*$ as
$$
\bar\x=\frac{B_{\epsilon}}{B_{\epsilon'}}\x_{\epsilon'}+\frac{B_{\epsilon'}-B_{\epsilon}}{B_{\epsilon'}}\x_{\epsilon'}^*.
$$
Since $0<B_{\epsilon}<B_{\epsilon'}$, $\bar\x$ is strictly between $\x_{\epsilon'}$ and $\x_{\epsilon'}^*$ and $dist(\bar\x,\Omega_*)=\|\x_{\epsilon'}^*-\bar\x\|=\frac{B_{\epsilon}}{B_{\epsilon'}}\|\x_{\epsilon'}^*-\x_{\epsilon'}\|=B_{\epsilon}$. 
By the convexity of $F$, we have
$$
\frac{F(\bar\x)-F_*}{dist(\bar\x,\Omega_*)}\leq\frac{F(\x_{\epsilon'})-F_*}{dist(\x_{\epsilon'},\Omega_*)}=\frac{\epsilon'}{B_{\epsilon'}}.
$$
Note that we must have $F(\bar\x)-F_*\geq\epsilon$ since, otherwise, we can move $\bar\x$ towards $\x_{\epsilon'}$ until $F(\bar\x)-F_*=\epsilon$ but $dist(\bar\x,\Omega_*)>B_{\epsilon}$, contradicting with the definition of $B_{\epsilon}$. Then, the proof is completed by applying $F(\bar\x)-F_*\geq\epsilon$ and $dist(\bar\x,\Omega_*)=B_{\epsilon}$ to the previous inequality. 
\end{proof}

\section{Proof of Lemma~\ref{thm:proxSSG}}\label{app:proxSSG}
In this proof, we need the following lemma.
\begin{lemma}[Lemma 2 of \citep{lan2012validation}]\label{lemma:lan2012}
Let $X_1, \dots, X_t$ be a martingale difference sequence, i.e. $\E[X_{\tau}|X_1,\dots,X_{\tau-1}]=0$ for all $\tau$. Suppose that for some values $\sigma_{\tau}$, for $\tau=1,\dots,t$, we have $\E\left[\exp\left(\frac{X_{\tau}^2}{\sigma_{\tau}^2}\right)|X_1,\dots,X_{\tau-1}\right]\leq \exp(1)$. Then with probability at least $1-\delta$, we have
\begin{align*}
\sum_{\tau=1}^{t}X_{\tau} \leq \sqrt{3\log(1/\delta)\sum_{\tau=1}^{t}\sigma_{\tau}^2}.
\end{align*}
\end{lemma}

Then, let us start the proof of Lemma~\ref{thm:proxSSG}.
\begin{proof}
%Note that align~(\ref{eqn:proxssg:1}) can be written as
%\begin{align*}
%\w_{\tau+1} =  {\arg\min}_{\w \in \mathcal K\cap \mathcal B(\w_1,D)} \frac{1}{2} \|\w- \w_{\tau} \|_2^2 + \eta \partial g(\w_{\tau}; \xi_\tau)^\top \w + \eta R(\w) 
%\end{align*}
Based on the fact that $\frac{1}{2} \|\w- \w_{\tau} \|_2^2 + \eta \partial f(\w_{\tau}; \xi_\tau)^\top \w + \eta R(\w)$ is $\frac{1}{2}$-stongly convex in terms of $\w$, then for any $\w \in  \mathcal B(\w_1,D)$, we have 
\begin{align*}%\label{eqn:prop:opt1}
 &\frac{1}{2} \|\w_{\tau+1}- \w_{\tau} \|_2^2 + \eta \partial f(\w_{\tau}; \xi_\tau)^\top \w_{\tau+1} + \eta R(\w_{\tau+1}) \\
 \leq& \frac{1}{2} \|\w- \w_{\tau} \|_2^2 + \eta \partial f(\w_{\tau}; \xi_\tau)^\top \w + \eta R(\w) - \frac{1}{2} \| \w -\w_{\tau+1}\|_2^2.
\end{align*}
Rewrite the inequality and then it becomes
\begin{align}\label{eqn:prop:opt}
\nonumber  & \partial f(\w_{\tau}; \xi_\tau)^\top (\w_{\tau+1}-\w) + R(\w_{\tau+1})-R(\w) \\
 \leq& \frac{1}{2\eta} \|\w- \w_{\tau} \|_2^2  - \frac{1}{2\eta} \| \w -\w_{\tau+1}\|_2^2 - \frac{1}{2\eta} \|\w_{\tau+1}- \w_{\tau} \|_2^2. 
\end{align}
Then we can lower bound the first term, that is
\begin{align}\label{eqn:conv:g}
\nonumber &\partial f(\w_{\tau}; \xi_\tau)^\top (\w_{\tau+1}-\w) \\
 \nonumber =  &\partial f(\w_{\tau}; \xi_\tau)^\top (\w_{\tau+1}-\w_\tau)  + \partial f(\w_{\tau}; \xi_\tau)^\top (\w_{\tau}-\w) \\  
\nonumber =& \partial f(\w_{\tau}; \xi_\tau)^\top (\w_{\tau+1}-\w_\tau)  + [\partial f(\w_{\tau}; \xi_\tau)-\partial f(\w_{\tau})]^\top (\w_{\tau}-\w) \\ \nonumber &+ \partial f(\w_{\tau})^\top (\w_{\tau}-\w) \\
 \geq& \partial f(\w_{\tau}; \xi_\tau)^\top (\w_{\tau+1}-\w_\tau)  + [\partial f(\w_{\tau}; \xi_\tau)-\partial f(\w_{\tau})]^\top (\w_{\tau}-\w) + f(\w_{\tau})-f(\w). 
\end{align}

The last inequality uses the convexity of $f(\w)$. Plugging inequality~(\ref{eqn:conv:g}) into~(\ref{eqn:prop:opt}), we get
\begin{align}\label{ineq1}
\nonumber &f(\w_{\tau}) - f(\w)  + R(\w_{\tau+1})-R(\w) \\
\nonumber \leq &\frac{1}{2\eta} \|\w- \w_{\tau} \|_2^2  - \frac{1}{2\eta} \| \w -\w_{\tau+1}\|_2^2 - \frac{1}{2\eta} \|\w_{\tau+1}- \w_{\tau} \|_2^2 \\
  &- \partial f(\w_{\tau}; \xi_\tau)^\top (\w_{\tau+1}-\w_\tau)  + \underbrace{ [\partial f(\w_{\tau}) - \partial f(\w_{\tau}; \xi_\tau)]^\top (\w_{\tau}-\w)}\limits_{\zeta_{\tau}(\w)}.
\end{align}
On the other hand, by the Cauchy-Shwartz Inequality, 
\begin{align}\label{ineq2}
\nonumber -\partial f(\w_{\tau}; \xi_\tau)^\top (\w_{\tau+1}-\w_\tau)   \leq& \frac{1}{2\eta} \|\w_{\tau+1}-\w_\tau\|_2^2+ \frac{1}{2}\eta\|\partial f(\w_{\tau}; \xi_\tau)\|_2^2  \\
 \leq &\frac{1}{2\eta}\|\w_{\tau+1}-\w_\tau\|_2^2 + \frac{1}{2}\eta G^2. 
\end{align}
Combining inequalities~(\ref{ineq1}) and ~(\ref{ineq2}) it will have 
 \begin{align*}
&f(\w_{\tau}) + R(\w_{\tau+1}) - f(\w)  -R(\w) \\
 \leq & \frac{1}{2\eta} \|\w- \w_{\tau} \|_2^2  - \frac{1}{2\eta} \| \w -\w_{\tau+1}\|_2^2 + \frac{1}{2}\eta G^2  + \zeta_{\tau}(\w).
\end{align*}
Taking summation over $\tau$ from $1$ to $t$ and dividing by $t$ on both sides of the inequality, then
 \begin{align*}
&\frac{1}{t}\sum_{\tau=1}^{t} F(\w_{\tau})  - F(\w)\\
 \leq &\frac{1}{t}(R(\w_1) - R(\w_{t+1}))+\frac{1}{2\eta t } \|\w- \w_1 \|_2^2  + \frac{\eta G^2}{2} + \frac{1}{t}\sum_{\tau=1}^{t}\zeta_{\tau}(\w).
\end{align*}
Since $\|\partial R(\w) \|_2 \leq \rho$ and the convexity of $F(\w)$, let $\w = \w_{1,\epsilon}^{\dagger}$, then we get
\begin{align}\label{ieq:proxASSGc:1}
\nonumber F(\wh_t)  - F(\w_{1,\epsilon}^{\dagger})   \leq &\frac{\rho \|\w_1-\w_{t+1}\|_2}{t}+\frac{ \|\w_1 -\w_{1,\epsilon}^{\dagger}\|_2^2 }{2\eta t} + \frac{\eta G^2}{2} + \frac{1}{t}\sum_{\tau=1}^{t}\zeta_{\tau}(\w_{1,\epsilon}^{\dagger}) \\
   \leq &\frac{\rho D}{t}+\frac{ \|\w_1 -\w_{1,\epsilon}^{\dagger}\|_2^2 }{2\eta t} + \frac{\eta G^2}{2} + \frac{1}{t}\sum_{\tau=1}^{t}\zeta_{\tau}(\w_{1,\epsilon}^{\dagger}).
\end{align}

Next, we will use the Lemma~\ref{lemma:lan2012} of martingale inequality to upper bound $\sum_{\tau=1}^{t}\zeta_{\tau}(\w_{1,\epsilon}^{\dagger})$ with a high probability.
%Next, we will use the following lemma of martingale inequality to upper bound $\sum_{\tau=1}^{t}\zeta_{\tau}(\w_{1,\epsilon}^{\dagger})$ with a high probability.
%\begin{lemma}[\citep{lan2012validation}]
%Let $X_1, \dots, X_t$ be a martingale difference sequence, i.e. $\E[X_{\tau}]=0$ for all $\tau$. Suppose that for some values $\sigma_{\tau}$, for $\tau=1,\dots,t$, we have $\E\left[\exp\left(\frac{X_{\tau}^2}{\sigma_{\tau}^2}\right)\right]\leq \exp(1)$. Then with probability at least $1-\delta$, we have
%\begin{align*}
%\sum_{\tau=1}^{t}X_{\tau} \leq \sqrt{3\log(1/\delta)\sum_{\tau=1}^{t}\sigma_{\tau}^2}.
%\end{align*}
%\end{lemma}
By using the Jensen's inequality, we have $\|\partial f(\w_{\tau}) \|_2 = \|\E [\partial f(\w_{\tau}; \xi_\tau)] \|_2 \leq \E[\|\partial f(\w_{\tau}; \xi_\tau)\|_2] \leq G$. Let's denote $X_{\tau} = \zeta_{\tau}(\w_{1,\epsilon}^{\dagger}) = [\partial f(\w_{\tau}) - \partial f(\w_{\tau}; \xi_\tau)]^\top (\w_{\tau}-\w_{1,\epsilon}^{\dagger})$,  then $\E\left[ X_\tau \right] = 0$ and
\begin{align*}
|X_\tau|  \leq &\|\partial f(\w_{\tau}) - \partial f(\w_{\tau}; \xi_\tau)\|_2  \|\w_{\tau}-\w_{1,\epsilon}^{\dagger}\|_2 \\
 \leq &\left( \|\partial f(\w_{\tau})\|_2 +  \|\partial f(\w_{\tau}; \xi_\tau\|_2 \right)  \left( \|\w_{\tau}-\w_1\|_2 +\|\w_1-\w_{1,\epsilon}^{\dagger}\|_2 \right) \leq 4GD,
\end{align*}
where we use the fact that $\w_\tau \in \mathcal B(\w_{1}, D)$ and $\|\w_1 - \w_{1,\epsilon}\|_2\leq D$.
This implies that
\begin{align*}
\E\left[ \exp\left( \frac{X_\tau^2}{16G^2D^2}\right)\right] \leq \exp(1).
\end{align*}
Then with probability at least $1-\delta$, we have
\begin{align}\label{ieq:proxASSGc:mart}
\sum_{\tau=1}^{t}X_{\tau} \leq \sqrt{3\log(1/\delta)\sum_{\tau=1}^{t}16G^2D^2} = 4GD\sqrt{3\log(1/\delta)t}. 
\end{align}
We complete the proof by combining (\ref{ieq:proxASSGc:1}) and (\ref{ieq:proxASSGc:mart}).
\end{proof}

\section{Proof of Theorem~\ref{thm:proxASSGc}}\label{app:proxASSGc}
\begin{proof}
This proof is similar to that of Theorem~\ref{thm:RSSG}. Let $\w_{k,\epsilon}^\dag$ denote the closest point to $\w_k$ in $\mathcal S_\epsilon$. Define $\epsilon_k = \frac{\epsilon_0}{2^k}$. Note that $D_k = \frac{D_1}{2^{k-1}} \geq\frac{c\epsilon_{k-1}}{\epsilon^{1-\theta}}$ and $\eta_k = \frac{\epsilon_{k-1}}{4G^2}$. We will show by induction that $F(\w_k) - F_*\leq \epsilon_k +\epsilon$ for $k=0,1,\dots$ with a high probability, which leads to our conclusion when $k=K$. The inequality holds obviously for $k=0$. Conditioned on $F(\w_{k-1}) - F_*\leq \epsilon_{k-1} + \epsilon$, we will show that $F(\w_k) - F_*\leq \epsilon_k +\epsilon$ with a high probability. By Lemma~\ref{lem:key}, we have
\begin{align} \label{eqn:Proxassgc:leb}
 \|\w^\dagger_{k-1,\epsilon} - \w_{k-1} \|_2 \leq \frac{c \epsilon_{k-1}}{\epsilon^{1-\theta}}\leq D_k.
\end{align}
We apply Lemma~\ref{thm:proxSSG} to the $k$-th stage of Algorithm~\ref{alg:rssg} conditioned on randomness in previous stages. With a probability $1-\tilde\delta$  we have
\begin{align}\label{eqn:Proxassgc}
F(\w_k)  - F(\w_{k-1,\epsilon}^\dagger) \leq  \frac{\rho D_k}{t} + \frac{\eta_k G^2}{2}+ \frac{\|\w_{k-1} - \w_{k-1,\epsilon}^\dagger\|_2^2}{2\eta_k t}   + \frac{4GD_k\sqrt{3\log(1/\tilde\delta)}}{\sqrt{t}}.
\end{align}
We now consider two cases for $\w_{k-1}$. First, we assume $F(\w_{k-1}) - F_* \leq \epsilon$, i.e. $\w_{k-1}\in\mathcal S_{\epsilon}$. Then we have $\w_{k-1,\epsilon}^\dagger=\w_{k-1}$ and
\begin{align*}
F(\w_k)  - F(\w_{k-1,\epsilon}^\dagger) \leq  \frac{\rho D_k}{t} + \frac{\eta_k G^2}{2} + \frac{4GD_k\sqrt{3\log(1/\tilde\delta)}}{\sqrt{t}} \leq  \frac{\epsilon_{k}}{4} +  \frac{\epsilon_{k}}{4} +  \frac{\epsilon_{k-1}}{8} = \frac{3\epsilon_{k}}{4}.
\end{align*}
The second inequality using the fact that $\eta_k = \frac{\epsilon_k}{2G^2}$, $t \geq 3072\log(1/\tilde\delta) \frac{G^2D_1^2}{\epsilon_0^2}$ and $ t \geq  \frac{8\rho D_1}{\epsilon_0} $. 
As a result,
\begin{align*}
F(\w_k)  - F_* \leq F(\w_{k-1,\epsilon}^\dagger)  -F_* + \frac{3\epsilon_{k}}{4} \leq \epsilon + \epsilon_k. 
\end{align*}
Next, we consider $F(\w_{k-1}) - F_* > \epsilon$, i.e. $\w_{k-1}\notin\mathcal S_{\epsilon}$. Then we have $F(\w_{k-1,\epsilon}^\dagger)-F_* = \epsilon$. 
Combining (\ref{eqn:Proxassgc:leb}) and (\ref{eqn:Proxassgc}), we get
\begin{align*}
F(\w_k)  - F(\w_{k-1,\epsilon}^\dagger) \leq  \frac{\rho D_k}{t} + \frac{\eta_k G^2}{2}+ \frac{D_k^2}{2\eta_k t}   + \frac{4GD_k\sqrt{3\log(1/\tilde\delta)}}{\sqrt{t}}.
\end{align*}
Since $\eta_k = \frac{\epsilon_k}{2G^2}$ and $t \geq \max\left\{ \max\{16, 3072\log(1/\tilde\delta)\}\frac{G^2D_1^2}{\epsilon_0^2}, \frac{8\rho D_1}{\epsilon_0} \right\} $, we have
\begin{align*}
&  \frac{\rho D_k}{t} \leq \frac{\rho D_k\epsilon_0}{8\rho D_1}=\frac{\epsilon_k}{4},\\
& \frac{\eta_kG^2}{2}=\frac{\epsilon_k}{4},\\
& \frac{D_k^2}{2\eta_k t}  \leq \frac{(D_1/2^{k-1})^2}{2\epsilon_{k}/(2G^2)} \frac{\epsilon_0^2}{16G^2D_1^2} =  \frac{\epsilon_k}{4},\\
  &\frac{4GD_k\sqrt{3\log(1/\tilde\delta)}}{\sqrt{t}} \leq  \frac{4G(D_1/2^{k-1})\sqrt{3\log(1/\tilde\delta)} \epsilon_0}{GD_1\sqrt{3072\log(1/\tilde\delta)}} = \frac{\epsilon_k}{4}.
\end{align*} 
As a result,
\begin{align*}
    F(\w_k)- F(\w_{k-1,\epsilon}^\dagger)\leq \epsilon_k\Rightarrow F(\w_k) - F_*\leq \epsilon_k + \epsilon.
\end{align*}
with a probability $1-\tilde\delta$.
Therefore by induction,  with a probability at least $(1-\tilde\delta)^K$ we have, 
\begin{align*}
F(\w_K) - F_*\leq \epsilon_K + \epsilon \leq 2\epsilon. 
\end{align*}
Since $\tilde \delta = \delta / K$, then $(1-\tilde \delta)^K \geq 1 - \delta$ and we complete the proof. 
\end{proof}

\section{Proximal ASSG based on the regularized variant}\label{app:proxASSGR}
In this section, we will present a proximal ASSG based on the  regularized variant, which is referred to ProxASSG-r. Similar to ASSG-r, we construct a new problem by adding a strongly convex term $\frac{1}{2\beta}\|\w-\w_1\|_2^2$ to the original problem~(\ref{eqn:composite}): 
\begin{align}\label{eqn:composite:reg}
\min_{\w\in\R^d}\widehat F(\w) = F(\w) + \frac{1}{2\beta}\|\w-\w_1\|_2^2,
\end{align}
where $F(\w)$ is defined in (\ref{eqn:composite}). We denote $\widehat \w_*$ the optimal solution to problem (\ref{eqn:composite:reg}) given the regularization reference point $\w_1$. We first extend SSGS to its proximal version as presented in Algorithm~\ref{alg:ssgs-prox}. To give the convergence analysis of ProxASSG-r for solving~(\ref{eqn:composite}), we first present a lemma below to bound $\|\wh_*-\w_t\|_2$  and $\|\w_t - \w_1\|_2$. %Its proof can be found in the appendix.
\begin{lemma}\label{lem:b:prox}
For any $t\geq 1$, we have $\|\wh_* - \w_t\|_2\leq 3\beta(G+\rho)$ and $\|\w_t - \w_1\|_2\leq 2\beta(G+\rho)$.
\end{lemma}
\begin{proof}
By the optimality of $\wh_*$, we have for any $\w\in\R^d$
\begin{align*}
(\partial F(\wh_*) + (\wh_* - \w_1)/\beta)^{\top}(\w - \wh_*)\geq 0.
\end{align*}
Let $\w = \w_1$, we have
\begin{align*}
\partial F(\wh_*)^{\top}(\w_1 - \wh_*)\geq\frac{\|\wh_* - \w_1\|^2_2}{\beta}.
\end{align*}
We have $\|\partial F(\wh_*)\|_2\leq G+\rho$ due to $\|\partial g(\w; \xi)\|_2\leq G$ and $\|\partial R(\w)\|_2\leq \rho$, then 
\begin{align*}
\|\wh_* - \w_1\|_2\leq \beta(G+\rho).
\end{align*}
Next, we bound $\|\w_t - \w_1\|_2$. According to the update of $\w_{t+1}$, there exists a subgradient $\partial R(\w_{t+1})$ such that 
\begin{align*}
\w_{t+1} - \left(\w_t - \eta_t[\partial f(\w_t; \xi_t) + \frac{1}{\beta}(\w_t - \w_1)]\right) + \eta_t\partial R(\w_{t+1})=0,
\end{align*}
where $\eta_t = \frac{2\beta}{t}$.  Thus,
\begin{align*}
\|\w_{t+1} - \w_1\|_2 = \|-\eta_t(\partial f(\w_t; \xi_t) + \partial R(\w_{t+1})) + (1-\eta_t/\beta) (\w_t-\w_1)\|_2.
\end{align*}
We prove $\|\w_t- \w_1\|_2 \leq 2\beta(G+\rho)$ by induction. First, we consider $t=1$, where $\eta_t = 2\beta$, then
\begin{align*}
\|\w_2- \w_1\|_2= \left\| 2\beta(\partial f(\w_t; \xi_t) + \partial R(\w_{t+1}))\right\|_2 \leq 2\beta(G+\rho).
\end{align*}
Then we consider any $t\geq 2$, where $\frac{\eta_t}{\beta}\leq 1$. Then
\begin{align*}
\|\w_{t+1}- \w_1\|_2   = &\left\|-\frac{\eta_t}{\beta} \beta (\partial f(\w_t; \xi_t) + \partial R(\w_{t+1})) + (1-\frac{\eta_t}{\beta}) (\w_t-\w_1)\right\|_2\\
 \leq &\frac{\eta_t}{\beta} \beta(G+\rho) + (1-\frac{\eta_t}{\beta}) 2\beta(G+\rho)\leq 2\beta(G+\rho).
\end{align*}
Therefore
\begin{align*}
\|\wh_* - \w_t\|_2\leq 3\beta(G+\rho).
\end{align*}
\end{proof}
%{\bf Remark:} The lemma implies that the regularization term implicitly imposes a constraint on the intermediate solutions to center around the regularization reference point, which achieves a similar effect as the ball constraint in Algorithm~\ref{alg:assgc-prox}.  The proof can be found in the appendix. 

\begin{algorithm}[t]
\caption{Proxmal SSG for solving~(\ref{eqn:composite}) with a Strongly convex regularizer: $ \textrm{ProxSSGS}(\w_1,\beta, T)$} \label{alg:ssgs-prox}
\begin{algorithmic}[1]
\FOR{$t=1,\ldots, T$}
    \STATE Compute $\w_{t+1}=  \textrm{Prox}_{\R^d}^{2\beta/t,R} \left[ (1 - \frac{2}{t})\w_t + \frac{2}{t}\w_1 - \frac{2\beta}{t} \partial f(\w_t; \xi_t) \right]$
    %\STATE Compute $\w_{t+1} = \Pi_\mathcal K(\w_{t+1}')$
\ENDFOR
\STATE \textbf{Output}:  $\widehat \w_T = \sum_{t=1}^{T}\w_t / T$
\end{algorithmic}
\end{algorithm}

Next, we present a high probability convergence bound  of ProxSSGS for optimizing  $\Fh(\w)$. %For completeness, we include its proof in Appendix.
\begin{theorem}\label{thm:proxSG}
Suppose  Assumption \ref{ass:2}.c holds. Let $\wh_T$ be the returned solution of Algorithm~\ref{alg:ssgs-prox}.  Given $\w_1\in\R^d$, $\delta<1/e$ and $T\geq 3$, with a high probability $1-\delta$ we have
\begin{align*}%\label{eqn:thm:proxPSG}
F(\widehat \w_T) - F(\w)  \leq \frac{1}{2\beta}\|\w - \w_1\|_2^2  +  \frac{34\beta(G+\rho)^2(1+\log T + \log(4\log T/\delta)}{T}.  
\end{align*}
\end{theorem}
\begin{proof}
Based on the update of $\w_{t+1}$ and the fact that $\frac{1}{2} \|\w- \w_{t} \|_2^2 + \eta_t [\partial f(\w_{t}; \xi_t) + (\w_t-\w_1)/\beta]^\top \w + \eta_t R(\w)$ is $\frac{1}{2}$-stongly convex in terms of $\w$, then for any $\w \in \R^d$, we have 
\begin{align*}%\label{eqn:prop:opt1}
 &\frac{1}{2} \|\w_{t+1}- \w_{t} \|_2^2 + \eta_t [\partial f(\w_{t}; \xi_t) + (\w_t-\w_1)/\beta]^\top \w_{t+1} + \eta_t R(\w_{t+1}) \\
 \leq &\frac{1}{2} \|\w- \w_{t} \|_2^2 + \eta_t [\partial f(\w_{t}; \xi_t) + (\w_t-\w_1)/\beta]^\top \w + \eta_t R(\w) - \frac{1}{2} \| \w -\w_{t+1}\|_2^2.
\end{align*}
Rearranging the inequality gives 
\begin{align}\label{eqn:prop:opt:r}
\nonumber  & \eta_t [\partial f(\w_{t})+ (\w_t-\w_1)/\beta]^\top( \w_{t+1}-\w) + \eta_t (R(\w_{t+1}) - R(\w)) \\
\nonumber    \leq&  \frac{1}{2} \|\w- \w_{t} \|_2^2 - \frac{1}{2} \| \w -\w_{t+1}\|_2^2 - \frac{1}{2} \|\w_{t+1}- \w_{t} \|_2^2  \\
&+ \eta_t [\partial f(\w_{t}) - \partial f(\w_{t}; \xi_t) ]^\top( \w_{t+1}-\w).
\end{align}
By the strong convexity of $f(\w) + \frac{1}{2\beta}\|\w-\w_1\|_2^2$, we have
\begin{align}\label{eqn:conv:g:r}
\nonumber   &\eta_t [\partial f(\w_{t})+ (\w_t-\w_1)/\beta]^\top( \w_{t+1}-\w) \\ 
\nonumber  = & \eta_t [\partial f(\w_{t})+ (\w_t-\w_1)/\beta]^\top( \w_{t}-\w) + \eta_t [\partial f(\w_{t})+ (\w_t-\w_1)/\beta]^\top( \w_{t+1}-\w_t) \\  
\nonumber  \geq & \eta_t \left[f(\w_{t})+ \frac{1}{2\beta}\|\w_t-\w_1\|_2^2\right] - \eta_t \left[f(\w)+ \frac{1}{2\beta}\|\w-\w_1\|_2^2)\right] + \frac{\eta_t }{2\beta}\|\w_t-\w\|_2^2 \\   
  &+ \eta_t [\partial f(\w_{t})+ (\w_t-\w_1)/\beta]^\top( \w_{t+1}-\w_t). 
\end{align}
Plugging inequality~(\ref{eqn:conv:g:r}) into~(\ref{eqn:prop:opt:r}), we get
\begin{align}\label{ineq1:r}
\nonumber   & \eta_t \left[f(\w_{t})+ \frac{1}{2\beta}\|\w_t-\w_1\|_2^2\right] - \eta_t \left[f(\w)+ \frac{1}{2\beta}\|\w-\w_1\|_2^2\right]  + \eta_t (R(\w_{t+1}) - R(\w))\\
\nonumber  \leq &\frac{1}{2} \|\w- \w_{t} \|_2^2   - \frac{1}{2} \| \w -\w_{t+1}\|_2^2  + \eta_t [\partial f(\w_{t}) - \partial f(\w_{t}; \xi_t) ]^\top( \w_{t+1}-\w) \\
  &- \frac{1}{2} \|\w_{t+1}- \w_{t} \|_2^2 - \frac{\eta_t}{2\beta}\|\w_t-\w\|_2^2 - \eta_t [\partial f(\w_{t})+ (\w_t-\w_1)/\beta]^\top( \w_{t+1}-\w_t). 
\end{align}
On the other hand, by the Cauchy-Shwartz inequality we have
\begin{align}\label{ineq2:r}
\nonumber &- \eta_t [\partial f(\w_{t})+ (\w_t-\w_1)/\beta]^\top( \w_{t+1}-\w_t)\\
\nonumber \leq  & \frac{1}{4} \|\w_{t+1}-\w_t\|_2^2+  \eta_t^2 \|\partial f(\w_{t})+(\w_t-\w_1)/\beta\|_2^2 \\
\leq & \frac{1}{4} \|\w_{t+1}-\w_t\|_2^2+ 2 \eta_t^2 [G^2 + 4(G+\rho)^2]
\end{align}
and
\begin{align}\label{ineq2:r2}
\nonumber  &\eta_t [\partial f(\w_{t}) - \partial f(\w_{t}; \xi_t) ]^\top( \w_{t+1}-\w) \\
\nonumber  = &\eta_t [\partial f(\w_{t}) - \partial f(\w_{t}; \xi_t) ]^\top( \w_{t}-\w) + \eta_t [\partial f(\w_{t}) - \partial f(\w_{t}; \xi_t) ]^\top( \w_{t+1}-\w_t) \\
\nonumber  \leq& \eta_t [\partial f(\w_{t}) - \partial f(\w_{t}; \xi_t) ]^\top( \w_{t}-\w) + \frac{1}{4} \|\w_{t+1}-\w_t\|_2^2+  \eta_t^2 \|\partial f(\w_{t}) - \partial f(\w_{t}; \xi_t) \|_2^2 \\
 \leq& \eta_t [\partial f(\w_{t}) - \partial f(\w_{t}; \xi_t) ]^\top( \w_{t}-\w) + \frac{1}{4} \|\w_{t+1}-\w_t\|_2^2+  4\eta_t^2 G^2.
\end{align}
Plugging inequalities~(\ref{ineq2:r}) and ~(\ref{ineq2:r2}) into inequality ~(\ref{ineq1:r}), we get
\begin{align*}
  & \left[f(\w_{t})+R(\w_{t+1})+ \frac{1}{2\beta}\|\w_t-\w_1\|_2^2\right] - \left[f(\w)+ R(\w) + \frac{1}{2\beta}\|\w-\w_1\|_2^2\right] \\
 \leq& \frac{1}{2\eta_t} \|\w- \w_{t} \|_2^2   - \frac{1}{2\eta_t} \| \w -\w_{t+1}\|_2^2  + \underbrace{[\partial f(\w_{t}) - \partial f(\w_{t}; \xi_t) ]^\top( \w_{t}-\w)}\limits_{\zeta_t(\w)} \\
  &- \frac{1}{2\beta}\|\w-\w_t\|_2^2 + 2 \eta_t [3G^2 + 4(G+\rho)^2]. 
\end{align*}
By summing the above inequalities across $t=1,\ldots, T$ and setting $\w = \wh_*$, we have
\begin{align*}%\label{eqn:toc}
 & \sum_{t=1}^T(\Fh(\w_t) - \Fh(\wh_*)) \nonumber\\
 \leq& \sum_{t=1}^{T-1}\frac{1}{2}\left(\frac{1}{\eta_{t+1}} - \frac{1}{\eta_{t}} - \frac{1}{2\beta}\right)\|\wh_* - \w_{t+1}\|_2^2  + \sum_{t=1}^T\zeta_t(\wh_*)  - \frac{1}{4\beta}\sum_{t=1}^T\|\wh_* - \w_t\|_2^2 \nonumber\\
 &- \frac{1}{4\beta}\|\wh_* - \w_1\|_2^2 + \frac{1}{2\eta_1}\|\wh_* - \w_1\|_2^2 + (6G^2 + 8(G+\rho)^2)\sum_{t=1}^T\eta_t + R(\w_1) - R(\w_{T+1}) \nonumber \\
 \leq &\sum_{t=1}^T\zeta_t(\wh_*) - \frac{1}{4\beta}\sum_{t=1}^T\|\wh_* - \w_t\|_2^2 +\beta[(12G^2 + 16(G+\rho)^2)(1+\log T)+2\rho(G+\rho)],
\end{align*}
where the last inequality uses $\eta_t = \frac{2\beta}{t}$ and $R(\w_1) - R(\w_{T+1}) \leq \rho\|\w_1-\w_{T+1}\|_2\leq 2\beta\rho(G+\rho)$. Next, we bound R.H.S of the above inequality. 
%Similar to the proof of Theorem~\ref{thm:SG} by using Lemma~\ref{lem:mart}, we let $X_t = \zeta_t$. 
%Then $X_1,\ldots, X_T$ is a martingale difference sequence. Let $D=\frac{3(G+\rho)}{\lambda}$. Note that $|\zeta_t|\leq 2GD$. By Lemma~\ref{lem:mart},  for any $\delta<1/e$ and $T\geq 3$, with a probability $1-\delta$ we have
%\begin{align}
%    \sum_{t=1}^T\zeta_t\leq \max\left\{2\sqrt{\log(4\log T/\delta)}\sqrt{\sum_{t=1}^T \textrm{Var}_t\zeta_t}, 6GD\log(4\log T/\delta)\right\}
%\end{align}
%Note that 
%\begin{align*}
%\sum_{t=1}^T \textrm{Var}_t\zeta_t = \sum_{t=1}^{T} \E_t[\zeta_t^2] \leq 4G^2 \sum_{t=1}^{T} \| \w_t - \wh_*\|_2^2  = 4G^2D_T
%\end{align*}
%As a result, with a probability $1-\delta$
%
%\begin{align*}
%    \sum_{t=1}^T\zeta_t  \leq 4G\sqrt{\log(4\log T/\delta)}\sqrt{D_T} + 6GD\log(4\log T/\delta)\\
%      \leq \frac{16G^2\log(4\log T/\delta)}{\lambda} + \frac{\lambda}{4}D_T + 6GD\log(4\log T/\delta)
%\end{align*}
By using Lemma~\ref{lem:mart}, we employ the same technique in the proof of Theorem~\ref{thm:RPSG} to proceed our proof. The only difference is that we set $D=3\beta(G+\rho)$. We omit the detailed steps but present the key results:
with a  probability $1-\delta$, we have 
\begin{align*}
    \sum_{t=1}^T\zeta_t - \frac{1}{4\beta}\sum_{t=1}^T\|\wh_* - \w_t\|_2^2  \leq &16\beta G^2\log(4\log T/\delta)+ 6GD\log(4\log T/\delta)\\
     = &\beta(34G^2+18G\rho)\log(4\log T/\delta).
\end{align*}
Thus, with a  probability $1-\delta$,
 \begin{align*}
     \Fh(\wh_T) - \Fh(\wh_*)  \leq& \beta\frac{(34G^2+18G\rho)\log(4\log T/\delta)}{T}  \\
&+ \beta\frac{(12G^2 + 16(G+\rho)^2)(1+\log T)+2\rho(G+\rho)}{T} \\
      \leq& \frac{34\beta(G+\rho)^2 (1+\log T + \log(4\log T/\delta))}{T}.
 \end{align*}
 We complete the proof by using the facts that $F(\wh_T)  \leq \Fh(\wh_T) $ and $\Fh(\wh_*) \leq F(\w) + \frac{\lambda}{2} \|\w-\w_1\|_2^2$.
\end{proof}
%Note that the constant factor $34$ is not optimized. Then we present the details of Proximal ASSG-r in Algorithm~\ref{alg:assgr-prox}, and the convergence result is stated in Theorem~\ref{thm:proxASSGr}.

\begin{algorithm}[t]
\caption{the ProxASSG-r algorithm for solving~(\ref{eqn:composite})} \label{alg:assgr-prox}
\begin{algorithmic}[1]
\STATE \textbf{Input}: the number of stages $K$ and  the  number of iterations $t$ per-stage,  the initial solution $\w_0\in\mathcal K$, and $\beta_1 \geq \frac{2c^2\epsilon_0}{\epsilon^{2(1-\theta)}}$
\FOR{$k=1,\ldots, K$}	
   	\STATE Let $\w_{k} =  \textrm{ProxSSGS}(\w_{k-1}, \beta_k, t)$     
	\STATE Update $\beta_{k+1} = \beta_k/2$ 
   \ENDFOR
\STATE \textbf{Output}:  $\w_K$
\end{algorithmic}
\end{algorithm}
Finally, we present ProxASSG-r in Algorithm~\ref{alg:assgr-prox} and its convergence guaratnee is presented in theorem below. 
\begin{theorem}\label{thm:proxASSGr}
%Suppose Assumption \ref{ass:2}  holds and $F(\w)$ obeys the local error bound condition.
Suppose Assumptions \ref{ass:1} and \ref{ass:2}  hold for a target $\epsilon\ll 1$. Given $\delta\in(0,1/e)$, let $\tilde\delta = \delta/K$ and $K = \lceil \log_2(\frac{\epsilon_0}{\epsilon})\rceil$ and $t$ be the smallest integer such that $t \geq  \max\{\frac{136\beta_1(G+\rho)^2(1+\log (4\log t/\tilde\delta)+\log t)}{\epsilon_0},3\}$.  Then ProxASSG-r guarantees that, with a probability $1-\delta$, %there exists $k \in \{1, 2, \dots, K \}$ such that
\begin{align*}
F(\w_K) - F_* \leq 2 \epsilon.
\end{align*}
As a result, the iteration complexity of ASSG-r for achieving an $2\epsilon$-optimal solution with a high probability $1-\delta$ is $\widetilde O (\log(1/\delta)/\epsilon^{2(1-\theta)})$ provided $\beta_1=\Omega(\frac{2c^2\epsilon_0}{\epsilon^{2(1-\theta)}})$.
\end{theorem}
\begin{proof}
The proof is the  same to the proof of Theorem~\ref{thm:RPSG} by replacing $G$ by $G+\rho$.
\end{proof}

\section{ASSG for Piecewise Convex Quadratic Minimization} \label{app:PCQM}
In this section, we develop an ASSG for piecewise convex quadratic minimization under the global error bound condition: 
\begin{align}\label{eqn:cpqa}
    dist(\w,\mathcal K_*)\leq c  [F(\w) - F_* + (F(\w) - F_*)^{1/2}], \quad \forall \w\in\R^d.
\end{align}
We assume that an upper bound of $c\leq \hat c$ is given. 
%We present the convergence rate of ASSG both for constraint and regularized variants to solve the $\ell_1$ regularized Huber loss minimization problem~(\ref{prob:huber}). 
Here, we only show the results for the constrained variant of ASSG, which is presented in Algorithm~\ref{alg:rssg:Huber}. The regularized variant is a simple exercise.  %The ASSG-c algorithm for solving (\ref{prob:huber}) is presented in Algorithm \ref{alg:rssg:Huber}. There are two differences compared to Algorithm \ref{alg:rssg}: (i) set $D_1 \geq c[\epsilon_0+\epsilon_0^\theta]$ and then gradually decrease it by dividing $2^\theta$ at each epoch; (ii) set an initial numer of iteration $t_1$ and then gradually increase it by dividing $2^\theta$ at each epoch, where $\theta = \frac{1}{2}$. 
\begin{algorithm}[t]
\caption{the ASSG-c algorithm  under the global error bound condition~(\ref{eqn:cpqa})} \label{alg:rssg:Huber}
\begin{algorithmic}[1]
\STATE \textbf{Input}: the  number of stages $K$, the number of iterations $t$ per stage, and the initial solution $\w_0$, $\eta_1=\epsilon_0/(3G^2)$ and $\hat c\geq c$
\FOR{$k=1,\ldots, K$}
\STATE Let $\w^k_1 = \w_{k-1}$ and $D_k\geq \hat c(\epsilon_{k-1} + \sqrt{\epsilon_{k-1}})$
\FOR{$\tau=1,\ldots, t_k$}
    \STATE Update $\w^k_{\tau+1} = \Pi_{\mathcal K\cap \mathcal B(\w_{k-1},D_k)}[\w^k_{\tau} - \eta_k \partial f(\w^k_{\tau}; \xi^k_\tau)]$
   \ENDFOR
\STATE Let $\w_k = \frac{1}{t_k}\sum_{\tau=1}^{t_k}\w^k_\tau$
\STATE Let $\eta_{k+1} = \eta_k/2$ 
\ENDFOR
\STATE \textbf{Output}:  $\w_K$
\end{algorithmic}
\end{algorithm}
\begin{theorem}\label{thm:RSSG-Huber}
Suppose Assumption \ref{ass:1}  holds and $F(\w)$ is convex and piecewise convex quadratic function. Given $\delta\in(0,1)$, let $\tilde\delta = \delta/K$,  $K = \lceil \log_2(\frac{\epsilon_0}{\epsilon})\rceil$, and $t_k$ be the smallest integer such that $t_k \geq 6912G^2\hat c^2\log(1/\tilde\delta)\max\{1, \frac{1}{\epsilon_k}\}$. Then Algorithm \ref{alg:rssg:Huber} guarantees that, with a probability $1-\delta$, %there exists $k \in \{1, 2, \dots, K \}$ such that
\begin{align*}
F(\w_K) - F_* \leq \epsilon.
\end{align*}
As a result, the iteration complexity of Algorithm \ref{alg:rssg:Huber} for achieving an $\epsilon$-optimal solution with a high probability $1-\delta$ is $ O (\log(1/\delta)/\epsilon)$.
\end{theorem}
\begin{proof}
Define $\epsilon_k = \frac{\epsilon_0}{2^k}$. Note that  $\eta_k = \frac{\epsilon_{k-1}}{3G^2}$. We will show by induction that $F(\w_k) - F_*\leq \epsilon_k$ for $k=0,1,\dots$ with a high probability, which leads to our conclusion when $k=K$. The inequality holds obviously for $k=0$. Conditioned on $F(\w_{k-1}) - F_*\leq \epsilon_{k-1} $, we will show that $F(\w_k) - F_*\leq \epsilon_k $ with a high probability. First, we have 
\begin{align*} %\label{eqn:assgc:geb1}
\nonumber \|\w_{k-1}^* - \w_{k-1} \|_2    \leq& c [F(\w_{k-1}) - F_* + (F(\w_{k-1}) - F_*)^{1/2}] \\
  \leq &c(\epsilon_{k-1}+\sqrt{\epsilon_{k-1}}) \leq D_k,
\end{align*}
where $\w^*_{k-1} \in \mathcal K_*$ is the closest point to $\w_{k-1}$ in the optimal set, 
the second inequality follows the global error bound (\ref{eqn:cpqa}) and the last inequality uses the value of $D_k$. We apply Lemma~\ref{thm:ssg} replacing $\w_{1,\epsilon}^\dagger$ with $\w_1^*$ to the $k$-th stage of Algorithm~\ref{alg:rssg} conditioned on randomness in previous stages. With a probability $1-\tilde\delta$  we have
\begin{align*}%\label{eqn:assgc:Huber}
\nonumber F(\w_k)  - F_*   \leq&  \frac{\eta_k G^2}{2}+ \frac{\|\w_{k-1} - \w_*\|_2^2}{2\eta_k t_k}   + \frac{4GD_k\sqrt{3\log(1/\tilde\delta)}}{\sqrt{t_k}} \\
 \leq & \frac{\eta_k G^2}{2}+ \frac{D_k^2}{\eta_k t_k}   + \frac{4GD_k\sqrt{3\log(1/\tilde\delta)}}{\sqrt{t_k}}.
\end{align*}
Since $\eta_k = \frac{2\epsilon_k}{3G^2}$ and $t_k \geq 6912G^2\hat c^2\log(1/\tilde\delta)\max\{1, \frac{1}{\epsilon_k}\}$, we can derive that $F(\w_k) - F_*\leq \epsilon_k$  
with a probability $1-\tilde\delta$.
Therefore by induction,  with a probability at least $(1-\tilde\delta)^K$ we have $F(\w_K) - F_*\leq \epsilon_K  \leq \epsilon$. 
Since $\tilde \delta = \delta / K$, then $(1-\tilde \delta)^K \geq 1 - \delta$ and we complete the proof. In fact, the total number of iterations of ASSG-c  is bounded by 
%\begin{align*} 
$T = \sum_{k=1}^{K} t_k \leq O\left( \log(1/\tilde\delta) \sum_{k=1}^{K} \frac{1}{\epsilon_k} \right) = O \left( \frac{ \log(1/\tilde\delta) }{\epsilon}\right)$.
%\end{align*}
\end{proof}

\bibliography{all}

\begin{thebibliography}{59}
\providecommand{\natexlab}[1]{#1}
\providecommand{\url}[1]{\texttt{#1}}
\expandafter\ifx\csname urlstyle\endcsname\relax
  \providecommand{\doi}[1]{doi: #1}\else
  \providecommand{\doi}{doi: \begingroup \urlstyle{rm}\Url}\fi

\bibitem[Allen{-}Zhu(2018)]{DBLP:conf/nips/Allen-Zhu18}
Zeyuan Allen{-}Zhu.
\newblock How to make the gradients small stochastically: Even faster convex
  and nonconvex {SGD}.
\newblock In \emph{Advances in Neural Information Processing Systems 31
  (NeurIPS 2018)}, pages 1165--1175, 2018.

\bibitem[Allen-Zhu and Yuan(2016)]{allen2016improved}
Zeyuan Allen-Zhu and Yang Yuan.
\newblock Improved svrg for non-strongly-convex or sum-of-non-convex
  objectives.
\newblock In \emph{International Conference on Machine Learning (ICML)}, pages
  1080--1089, 2016.

\bibitem[Attouch et~al.(2013)Attouch, Bolte, and
  Svaiter]{journals/mp/AttouchBS13}
Hedy Attouch, J{\'e}r\^{o}me Bolte, and Benar~Fux Svaiter.
\newblock Convergence of descent methods for semi-algebraic and tame problems:
  proximal algorithms, forward-backward splitting, and regularized
  {G}auss-seidel methods.
\newblock \emph{Mathematical Programming}, 137\penalty0 (1-2):\penalty0
  91--129, 2013.

\bibitem[Bartlett and Wegkamp(2008)]{bartlett2008classification}
Peter~L Bartlett and Marten~H Wegkamp.
\newblock Classification with a reject option using a hinge loss.
\newblock \emph{Journal of Machine Learning Research}, 9:\penalty0 1823--1840,
  2008.

\bibitem[Bauschke and Combettes(2011)]{Bauschke:2011:CAM:2028633}
Heinz~H. Bauschke and Patrick~L. Combettes.
\newblock \emph{Convex Analysis and Monotone Operator Theory in Hilbert
  Spaces}.
\newblock Springer Publishing Company, Incorporated, 1st edition, 2011.
\newblock ISBN 1441994661, 9781441994660.

\bibitem[Blanchard and Kr{\"a}mer(2016)]{blanchard2016convergence}
Gilles Blanchard and Nicole Kr{\"a}mer.
\newblock Convergence rates of kernel conjugate gradient for random design
  regression.
\newblock \emph{Analysis and Applications}, 14\penalty0 (06):\penalty0
  763--794, 2016.

\bibitem[Bolte et~al.(2006)Bolte, Daniilidis, and
  Lewis]{Bolte:2006:LIN:1328019.1328299}
J{\'e}r\^{o}me Bolte, Aris Daniilidis, and Adrian Lewis.
\newblock The {$\L$}ojasiewicz inequality for nonsmooth subanalytic functions
  with applications to subgradient dynamical systems.
\newblock \emph{SIAM Journal on Optimization}, 17:\penalty0 1205--1223, 2006.

\bibitem[Bolte et~al.(2017)Bolte, Nguyen, Peypouquet, and
  Suter]{arxiv:1510.08234}
J{\'e}r{\^o}me Bolte, Trong~Phong Nguyen, Juan Peypouquet, and Bruce~W Suter.
\newblock From error bounds to the complexity of first-order descent methods
  for convex functions.
\newblock \emph{Mathematical Programming}, 165\penalty0 (2):\penalty0 471--507,
  2017.

\bibitem[Davis and Drusvyatskiy(2019)]{davis2018stochastic}
Damek Davis and Dmitriy Drusvyatskiy.
\newblock Stochastic model-based minimization of weakly convex functions.
\newblock \emph{SIAM Journal on Optimization}, 29\penalty0 (1):\penalty0
  207--239, 2019.

\bibitem[Defazio et~al.(2014)Defazio, Bach, and
  Lacoste{-}Julien]{DBLP:conf/nips/DefazioBL14}
Aaron Defazio, Francis~R. Bach, and Simon Lacoste{-}Julien.
\newblock {SAGA:} {A} fast incremental gradient method with support for
  non-strongly convex composite objectives.
\newblock In \emph{Advances in Neural Information Processing Systems (NIPS)},
  pages 1646--1654, 2014.

\bibitem[Duchi et~al.(2010)Duchi, Shalev-Shwartz, Singer, and
  Tewari]{duchi2010composite}
John~C Duchi, Shai Shalev-Shwartz, Yoram Singer, and Ambuj Tewari.
\newblock Composite objective mirror descent.
\newblock In \emph{Proceedings of the 23rd Annual Conference on Learning Theory
  (COLT)}, pages 14--26, 2010.

\bibitem[Fang et~al.(2018)Fang, Xu, and Ying]{fang2018faster}
Qin Fang, Min Xu, and Yiming Ying.
\newblock Faster convergence of a randomized coordinate descent method for
  linearly constrained optimization problems.
\newblock \emph{Analysis and Applications}, 16\penalty0 (05):\penalty0
  741--755, 2018.

\bibitem[Foster et~al.(2019)Foster, Sekhari, Shamir, Srebro, Sridharan, and
  Woodworth]{DBLP:journals/corr/abs-1902-04686}
Dylan~J. Foster, Ayush Sekhari, Ohad Shamir, Nathan Srebro, Karthik Sridharan,
  and Blake~E. Woodworth.
\newblock The complexity of making the gradient small in stochastic convex
  optimization.
\newblock \emph{arXiv preprint arXiv:1902.04686}, 2019.

\bibitem[Ghadimi and Lan(2013)]{DBLP:journals/siamjo/Lan13b}
Saeed Ghadimi and Guanghui Lan.
\newblock Optimal stochastic approximation algorithms for strongly convex
  stochastic composite optimization, ii: Shrinking procedures and optimal
  algorithms.
\newblock \emph{{SIAM} Journal on Optimization}, 23\penalty0 (4):\penalty0
  2061Ã¢ÂÂ2089, 2013.

\bibitem[Goebel and Rockafellar(2008)]{Goebel_localstrong}
Rafal Goebel and Ralph~Tyrell Rockafellar.
\newblock Local strong convexity and local lipschitz continuity of the gradient
  of convex functions.
\newblock \emph{Journal of Convex Analysis}, 15\penalty0 (2):\penalty0 263,
  2008.

\bibitem[Gong and Ye(2014)]{DBLP:journals/corr/GongY14}
Pinghua Gong and Jieping Ye.
\newblock Linear convergence of variance-reduced projected stochastic gradient
  without strong convexity.
\newblock \emph{arXiv preprint arXiv:1406.1102}, 2014.

\bibitem[Guo et~al.(2017)Guo, Xiang, Guo, and Zhou]{guo2017thresholded}
Zheng-Chu Guo, Dao-Hong Xiang, Xin Guo, and Ding-Xuan Zhou.
\newblock Thresholded spectral algorithms for sparse approximations.
\newblock \emph{Analysis and Applications}, 15\penalty0 (03):\penalty0
  433--455, 2017.

\bibitem[Hastie et~al.(2009)Hastie, Tibshirani, and
  Friedman]{The-Elements-of-Statistical-Learning-2009}
Trevor Hastie, Robert Tibshirani, and Jerome Friedman.
\newblock \emph{The Elements of Statistical Learning}.
\newblock Springer Series in Statistics. Springer, 2009.

\bibitem[Hazan and Kale(2011)]{hazan-20110-beyond}
Elad Hazan and Satyen Kale.
\newblock Beyond the regret minimization barrier: an optimal algorithm for
  stochastic strongly-convex optimization.
\newblock In \emph{Proceedings of the 24th Annual Conference on Learning Theory
  (COLT)}, pages 421--436, 2011.

\bibitem[Hou et~al.(2013)Hou, Zhou, So, and Luo]{DBLP:conf/nips/HouZSL13}
Ke~Hou, Zirui Zhou, Anthony~Man{-}Cho So, and Zhi{-}Quan Luo.
\newblock On the linear convergence of the proximal gradient method for trace
  norm regularization.
\newblock In \emph{Advances in Neural Information Processing Systems (NIPS)},
  pages 710--718, 2013.

\bibitem[Johnson and Zhang(2013)]{NIPS2013_4937}
Rie Johnson and Tong Zhang.
\newblock Accelerating stochastic gradient descent using predictive variance
  reduction.
\newblock In \emph{Advances in Neural Information Processing Systems (NIPS)},
  pages 315--323, 2013.

\bibitem[Juditsky and Nesterov(2014)]{juditsky2014}
Anatoli Juditsky and Yuri Nesterov.
\newblock Deterministic and stochastic primal-dual subgradient algorithms for
  uniformly convex minimization.
\newblock \emph{Stochastic Systems}, 4:\penalty0 44--80, 2014.

\bibitem[Kakade and Tewari(2008)]{DBLP:conf/nips/KakadeT08}
Sham~M. Kakade and Ambuj Tewari.
\newblock On the generalization ability of online strongly convex programming
  algorithms.
\newblock In \emph{Advances in Neural Information Processing Systems (NIPS)},
  pages 801--808, 2008.

\bibitem[Karimi et~al.(2016)Karimi, Nutini, and
  Schmidt]{DBLP:conf/pkdd/KarimiNS16}
Hamed Karimi, Julie Nutini, and Mark~W. Schmidt.
\newblock Linear convergence of gradient and proximal-gradient methods under
  the {P}olyak-{$\L$}ojasiewicz condition.
\newblock In \emph{Joint European Conference on Machine Learning and Knowledge
  Discovery in Databases (ECML-PKDD)}, pages 795--811, 2016.

\bibitem[Lan et~al.(2012)Lan, Nemirovski, and Shapiro]{lan2012validation}
Guanghui Lan, Arkadi Nemirovski, and Alexander Shapiro.
\newblock Validation analysis of mirror descent stochastic approximation
  method.
\newblock \emph{Mathematical programming}, 134\penalty0 (2):\penalty0 425--458,
  2012.

\bibitem[Li(2013)]{DBLP:journals/mp/Li13}
Guoyin Li.
\newblock Global error bounds for piecewise convex polynomials.
\newblock \emph{Mathematical programming}, 137\penalty0 (1-2):\penalty0 37--64,
  2013.

\bibitem[Li and Pong(2017)]{guoyincalculus2016}
Guoyin Li and Ting~Kei Pong.
\newblock Calculus of the exponent of {K}urdyka-{$\L$}ojasiewicz inequality and
  its applications to linear convergence of first-order methods.
\newblock \emph{Foundations of Computational Mathematics}, pages 1--34, 2017.

\bibitem[Liu and Wright(2015)]{DBLP:journals/siamjo/LiuW15}
Ji~Liu and Stephen~J. Wright.
\newblock Asynchronous stochastic coordinate descent: Parallelism and
  convergence properties.
\newblock \emph{{SIAM} Journal on Optimization}, 25:\penalty0 351--376, 2015.

\bibitem[Liu et~al.(2015)Liu, Wright, R{\'e}, Bittorf, and
  Sridhar]{Liu:2015:APS:2789272.2789282}
Ji~Liu, Stephen~J. Wright, Christopher R{\'e}, Victor Bittorf, and Srikrishna
  Sridhar.
\newblock An asynchronous parallel stochastic coordinate descent algorithm.
\newblock \emph{Journal Machine Learning Research}, 16:\penalty0 285--322,
  2015.

\bibitem[Liu and Yang(2017)]{Ming2017adaAGC}
Mingrui Liu and Tianbao Yang.
\newblock Adaptive accelerated gradient converging method under holderian error
  bound condition.
\newblock In \emph{Advances in Neural Information Processing Systems (NIPS)},
  pages 3107--3117, 2017.

\bibitem[Luo and Tseng(1992{\natexlab{a}})]{Luo:1992a}
Zhi-Quan Luo and Paul Tseng.
\newblock On the convergence of coordinate descent method for convex
  differentiable minization.
\newblock \emph{Journal of Optimization Theory and Applications}, 72\penalty0
  (1):\penalty0 7--35, 1992{\natexlab{a}}.

\bibitem[Luo and Tseng(1992{\natexlab{b}})]{Luo:1992b}
Zhi-Quan Luo and Paul Tseng.
\newblock On the linear convergence of descent methods for convex essenially
  smooth minization.
\newblock \emph{SIAM Journal on Control and Optimization}, 30\penalty0
  (2):\penalty0 408--425, 1992{\natexlab{b}}.

\bibitem[Luo and Tseng(1993)]{Luo:1993}
Zhi-Quan Luo and Paul Tseng.
\newblock Error bounds and convergence analysis of feasible descent methods: a
  general approach.
\newblock \emph{Annals of Operations Research}, 46:\penalty0 157--178, 1993.

\bibitem[Necoara et~al.(2016)Necoara, Nesterov, and
  Glineur]{DBLP:journals/corr/nesterov16linearnon}
Ion Necoara, Yu~Nesterov, and Francois Glineur.
\newblock Linear convergence of first order methods for non-strongly convex
  optimization.
\newblock \emph{Mathematical Programming}, pages 1--39, 2016.

\bibitem[Nemirovsky~A.S. and Yudin(1983)]{opac-b1091338}
Arkadii~Semenovich. Nemirovsky~A.S. and D.~B Yudin.
\newblock \emph{Problem complexity and method efficiency in optimization}.
\newblock Wiley-Interscience series in discrete mathematics. Wiley, Chichester,
  New York, 1983.

\bibitem[Nesterov(2004)]{opac-b1104789}
Yurii Nesterov.
\newblock \emph{Introductory lectures on convex optimization : a basic course}.
\newblock Applied optimization. Kluwer Academic Publ., 2004.

\bibitem[Nyquist(1983)]{doi:10.1080/03610928308828618}
H.~Nyquist.
\newblock The optimal lp norm estimator in linear regression models.
\newblock \emph{Communications in Statistics - Theory and Methods}, 12\penalty0
  (21):\penalty0 2511--2524, 1983.

\bibitem[Qu et~al.(2016)Qu, Xu, and Ong]{DBLP:conf/icml/QuXO16}
Chao Qu, Huan Xu, and Chong~Jin Ong.
\newblock Fast rate analysis of some stochastic optimization algorithms.
\newblock In \emph{International Conference on Machine Learning (ICML)}, pages
  662--670, 2016.

\bibitem[Rakhlin et~al.(2012)Rakhlin, Shamir, and Sridharan]{RakhlinSS12}
Alexander Rakhlin, Ohad Shamir, and Karthik Sridharan.
\newblock Making gradient descent optimal for strongly convex stochastic
  optimization.
\newblock In \emph{International Conference on Machine Learning (ICML)}, pages
  1571--1578, 2012.

\bibitem[Ramdas and Singh(2013)]{DBLP:conf/icml/RamdasS13}
Aaditya Ramdas and Aarti Singh.
\newblock Optimal rates for stochastic convex optimization under {T}sybakov
  noise condition.
\newblock In \emph{International Conference on Machine Learning (ICML)}, pages
  365--373, 2013.

\bibitem[Rockafellar(1976)]{rockafellar76}
R.~Tyrrell Rockafellar.
\newblock Monotone operators and the proximal point algorithm.
\newblock \emph{SIAM Journal on Control and Optimization}, 14:\penalty0
  877--898, 1976.

\bibitem[Rosasco et~al.(2004)Rosasco, De~Vito, Caponnetto, Piana, and
  Verri]{rosasco2004loss}
Lorenzo Rosasco, Ernesto De~Vito, Andrea Caponnetto, Michele Piana, and
  Alessandro Verri.
\newblock Are loss functions all the same?
\newblock \emph{Neural Computation}, 16\penalty0 (5):\penalty0 1063--1076,
  2004.

\bibitem[Roux et~al.(2012)Roux, Schmidt, and Bach]{DBLP:conf/nips/RouxSB12}
Nicolas~Le Roux, Mark~W. Schmidt, and Francis Bach.
\newblock A stochastic gradient method with an exponential convergence rate for
  finite training sets.
\newblock In \emph{Advances in Neural Information Processing Systems (NIPS)},
  pages 2672--2680, 2012.

\bibitem[Vapnik(1998)]{Vapnik1998}
V.~N. Vapnik.
\newblock \emph{Statistical Learning Theory}.
\newblock John Wiley \& Sons, 1998.

\bibitem[Wang and Lin(2014)]{DBLP:journals/jmlr/WangL14}
Po{-}Wei Wang and Chih{-}Jen Lin.
\newblock Iteration complexity of feasible descent methods for convex
  optimization.
\newblock \emph{Journal of Machine Learning Research}, 15\penalty0
  (1):\penalty0 1523--1548, 2014.

\bibitem[Xiao(2010)]{xiao2010dual}
Lin Xiao.
\newblock Dual averaging methods for regularized stochastic learning and online
  optimization.
\newblock \emph{Journal of Machine Learning Research}, 11\penalty0
  (Oct):\penalty0 2543--2596, 2010.

\bibitem[Xiao and Zhang(2014)]{xiao2014proximal}
Lin Xiao and Tong Zhang.
\newblock A proximal stochastic gradient method with progressive variance
  reduction.
\newblock \emph{SIAM Journal on Optimization}, 24\penalty0 (4):\penalty0
  2057--2075, 2014.

\bibitem[Xu et~al.(2016)Xu, Yan, Lin, and
  Yang]{DBLP:journals/corr/abs-1607-03815}
Yi~Xu, Yan Yan, Qihang Lin, and Tianbao Yang.
\newblock Homotopy smoothing for non-smooth problems with lower complexity than
  $\textrm{O}(1/\epsilon)$.
\newblock In \emph{Advances in Neural Information Processing Systems (NIPS)},
  pages 1208--1216, 2016.

\bibitem[Xu et~al.(2017)Xu, Lin, and Yang]{xu2017AdaSVRG}
Yi~Xu, Qihang Lin, and Tianbao Yang.
\newblock Adaptive svrg methods under error bound conditions with unknown
  growth parameter.
\newblock In \emph{Advances In Neural Information Processing Systems 30
  (NIPS)}, pages 3279--3289, 2017.

\bibitem[Yang and Lin(2018)]{DBLP:journals/corr/arXiv:1512.03107}
Tianbao Yang and Qihang Lin.
\newblock Rsg: Beating subgradient method without smoothness and strong
  convexity.
\newblock \emph{The Journal of Machine Learning Research}, 19\penalty0
  (1):\penalty0 236--268, 2018.

\bibitem[Yang et~al.(2015)Yang, Mahdavi, Jin, and
  Zhu]{DBLP:journals/ml/YangMJZ14Non}
Tianbao Yang, Mehrdad Mahdavi, Rong Jin, and Shenghuo Zhu.
\newblock An efficient primal dual prox method for non-smooth optimization.
\newblock \emph{Machine Learning}, 98\penalty0 (3):\penalty0 369--406, 2015.

\bibitem[Zadorozhnyi et~al.(2016)Zadorozhnyi, Benecke, Mandt, Scheffer, and
  Kloft]{DBLP:conf/pkdd/ZadorozhnyiBMSK16}
Oleksandr Zadorozhnyi, Gunthard Benecke, Stephan Mandt, Tobias Scheffer, and
  Marius Kloft.
\newblock Huber-norm regularization for linear prediction models.
\newblock In \emph{Joint European Conference on Machine Learning and Knowledge
  Discovery in Databases (ECML-PKDD)}, pages 714--730, 2016.

\bibitem[Zhang(2016)]{DBLP:journals/corr/abs/1606.00269}
Hui Zhang.
\newblock New analysis of linear convergence of gradient-type methods via
  unifying error bound conditions.
\newblock \emph{Mathematical Programming}, pages 1--46, 2016.

\bibitem[Zhang(2017)]{zhang2017restricted}
Hui Zhang.
\newblock The restricted strong convexity revisited: analysis of equivalence to
  error bound and quadratic growth.
\newblock \emph{Optimization Letters}, 11\penalty0 (4):\penalty0 817--833,
  2017.

\bibitem[Zhang and Yin(2013)]{DBLP:journals/corr/abs-1303-4645}
Hui Zhang and Wotao Yin.
\newblock Gradient methods for convex minimization: better rates under weaker
  conditions.
\newblock \emph{arXiv preprint arXiv:1303.4645}, 2013.

\bibitem[Zhang et~al.(2013)Zhang, Mahdavi, and Jin]{DBLP:conf/nips/0005MJ13}
Lijun Zhang, Mehrdad Mahdavi, and Rong Jin.
\newblock Linear convergence with condition number independent access of full
  gradients.
\newblock In \emph{Advances in Neural Information Processing Systems (NIPS)},
  pages 980--988, 2013.

\bibitem[Zhou and So(2017)]{DBLP:journals/corr/ZhouS15a}
Zirui Zhou and Anthony Man-Cho So.
\newblock A unified approach to error bounds for structured convex optimization
  problems.
\newblock \emph{Mathematical Programming}, 165\penalty0 (2):\penalty0 689--728,
  2017.

\bibitem[Zhou et~al.(2015)Zhou, Zhang, and So]{DBLP:conf/icml/ZhouZS15}
Zirui Zhou, Qi~Zhang, and Anthony~Man{-}Cho So.
\newblock L1p-norm regularization: Error bounds and convergence rate analysis
  of first-order methods.
\newblock In \emph{International Conference on Machine Learning (ICML)}, pages
  1501--1510, 2015.

\bibitem[Zhu et~al.(2016)Zhu, Chatterjee, Duchi, and
  Lafferty]{DBLP:conf/nips/ChatterjeeDLZ16}
Yuancheng Zhu, Sabyasachi Chatterjee, John~C. Duchi, and John~D. Lafferty.
\newblock Local minimax complexity of stochastic convex optimization.
\newblock In \emph{Advances In Neural Information Processing Systems (NIPS)},
  pages 3423--3431, 2016.

\end{thebibliography}

\end{document}